\documentclass{amsart}

\usepackage{amsfonts}
\usepackage{amssymb}
\usepackage{amsmath}
\usepackage{amscd}
\usepackage{amsthm}
\usepackage[mathscr]{euscript}
\usepackage{graphicx}
\usepackage{epsfig}
\usepackage{wrapfig}
\usepackage{xypic}
\usepackage{color}
\usepackage{mathdots}
\usepackage[enableskew]{youngtab}
\usepackage{mathabx}
\usepackage[enableskew]{youngtab}
\usepackage{ytableau}
\usepackage{comment}

\usepackage{tikz}
\usetikzlibrary{matrix,arrows,decorations.pathmorphing}

\usepackage{amsmath,amssymb,amscd,mathrsfs,epic,wasysym,latexsym,tikz,mathrsfs,cite,hyperref,array,amsthm,hhline,stmaryrd,latexsym}

\colorlet{darkgreen}{green!50!black}
\tikzset{dots/.style={very thick,loosely dotted},
         greendot/.style={fill,circle,color=darkgreen,inner sep=1.5pt,outer sep=0},
         bluedot/.style={fill,circle,color=blue,inner sep=1.5pt,outer sep=0}
}
\def\greendot(#1,#2){\node[greendot] at(#1,#2){}}
\def\bluedot(#1,#2){\node[bluedot] at(#1,#2){}}

\newenvironment{braid}{
  \begin{tikzpicture}[baseline=6mm,black,line width=1pt, scale=0.33,
                      every node/.append style={font=\fontsize{5}{5}\selectfont}]%
  }{\end{tikzpicture}
}

\def\Grid(#1,#2){
  \draw[very thin,gray,step=2mm] (0,0)grid(#1,#2);
  \draw[very thin,darkgreen,step=10mm] (0,0)grid(#1,#2);
}

\newcommand\Tableau[2][\relax]{
  \begin{tikzpicture}[scale=0.5,draw/.append style={thick,black}]
    \ifx\relax#1\relax%
    \else 
      \foreach\box in {#1} { \filldraw[blue!30]\box+(-.5,-.5)rectangle++(.5,.5); }
    \fi
    \newcount\row\newcount\col
    \row=0
    \foreach \Row in {#2} {
       \col=1
       \foreach\k in \Row {
          \draw(\the\col,\the\row)+(-.5,-.5)rectangle++(.5,.5);
          \draw(\the\col,\the\row)node{\k};
          \global\advance\col by 1
       }
       \global\advance\row by -1
    }
  \end{tikzpicture}
}

\newcommand\YoungDiagram[2][\relax]{
  \begin{tikzpicture}[scale=0.5,draw/.append style={thick,black}]
    \ifx\relax#1\relax%
    \else 
    \foreach\box in {#1} {
      \filldraw[blue!30]\box rectangle ++(1,1);
    }
    \fi
    \newcount\row
    \row=0
    \foreach \col in {#2} {
       \draw(1,\the\row)grid ++(\col,1);
       \global\advance\row by -1
    }
  \end{tikzpicture}
}

\usepackage[all]{xy}
\DeclareMathOperator{\Hom}{Hom}

\newcommand{\CC}{\mathbb{C}}

\newcommand{\ZZ}{\mathbb{Z}}

\newcommand{\id}{\textup{id}}

\newcommand{\CH}{\textup{ch}}

\newcommand{\ch}{\textup{char}}

\newcommand{\sh}{\textup{sh}}

\newcommand{\Res}{\textup{Res}}
\newcommand{\res}{\textup{res}}
\newcommand{\Ind}{\textup{Ind}}
\newcommand{\Gar}{\textup{Gar}}
\newcommand{\codeg}{\textup{codeg}}
\newcommand{\St}{\textup{St}}
\newcommand{\Tab}{\textup{Tab}}

\newcommand{\bi}{\boldsymbol{i}}
\newcommand{\bj}{\boldsymbol{j}}
\newcommand{\bk}{\boldsymbol{k}}
\newcommand{\bI}{\boldsymbol{I}}

\newcommand{\bs}{\boldsymbol{s}}

\newcommand{\br}{\boldsymbol{r}}
\newcommand{\TTT}{{\tt T}}
\newcommand{\SSS}{{\tt S}}
\newcommand{\UUU}{{\tt U}}
\newcommand{\GGG}{{\tt G}}
\newcommand{\YYY}{{\tt Y}}
\newcommand{\ttt}{{\tt t}}
\newcommand{\sss}{{\tt s}}
\newcommand{\uuu}{{\tt u}}

\newcommand{\xxx}{{\tt x}}
\newcommand{\gggg}{{\tt g}}
\newcommand{\blam}{\boldsymbol{\lambda}}
\newcommand{\bmu}{\boldsymbol{\mu}}
\newcommand{\bnu}{\boldsymbol{\nu}}
\newcommand{\bxi}{\boldsymbol{\xi}}
\newcommand{\bzeta}{\boldsymbol{\zeta}}

\usepackage[final]{showkeys}

\newtheorem{thm}{Theorem}[section]
\newtheorem{prop}[thm]{Proposition}
\newtheorem{lem}[thm]{Lemma}
\newtheorem{cor}[thm]{Corollary}

\theoremstyle{definition}
\newtheorem{definition}[thm]{Definition}
\newtheorem{example}[thm]{Example}

\theoremstyle{remark}

\newtheorem{remark}[thm]{Remark}

\numberwithin{equation}{section}

\begin{document}

\title[Skew Specht modules and cuspidal modules]{Graded skew Specht modules and cuspidal modules for Khovanov-Lauda-Rouquier algebras of affine type A}

\author{Robert Muth}
\address{Department of Mathematics, University of Oregon, 
Eugene, OR 97403}
\email{muth@uoregon.edu}

\begin{abstract}
Kleshchev, Mathas and Ram (2012) gave a presentation for graded Specht modules over Khovanov-Lauda-Rouquier algebras of finite and affine type A. We show that this construction can be applied more generally to skew shapes to give a presentation of graded skew Specht modules, which arise as subquotients of restrictions of Specht modules. As an application, we show that cuspidal modules associated to a balanced convex preorder in affine type A are skew Specht modules for certain hook shapes.

\end{abstract}

\maketitle

\section{Introduction}
\noindent Let \(\mathcal{O}\) be a commutative ring with identity, and let \(\mathfrak{S}_d\) be the symmetric group on \(d\) letters. To every partition \(\lambda\) of \(d\), or equivalently, every Young diagram with \(d\) nodes, there is an associated \(\mathcal{O}\mathfrak{S}_d\)-module \(S^\lambda_{\mathcal{O}}\) called a {\it Specht module}, which has \(\mathcal{O}\)-basis in correspondence with standard \(\lambda\)-tableaux. Over the complex numbers, the group algebra of \(\mathfrak{S}_d\) is semisimple, and it is well known that \(\{ S^\lambda_{\CC} \mid \lambda \vdash d\}\) is a complete set of {\it irreducible} representations. For \(k \leq d\), we consider \(\mathfrak{S}_k\) a subgroup of \(\mathfrak{S}_d\) with respect to the first \(k\) letters, and denote the copy of  \(\mathfrak{S}_{k}\) embedded in \(\mathfrak{S}_d\) with respect to the last \(k\) letters as \(\mathfrak{S}_k'\). For \(\lambda \vdash d\) and \(\mu \vdash k\),
\begin{align}\label{skewhom}
S^{\lambda/\mu}_\CC:= \Hom_{\mathfrak{S}_{k}}(S_\CC^{\mu}, \Res_{\mathfrak{S}_{k}}S_\CC^\lambda)
\end{align}
is a \(\CC\mathfrak{S}_{d-k}'\)-module. In fact, \(S^{\lambda/\mu}_\CC \neq 0\) if and only if the Young diagram for \(\mu\) is contained in that of \(\lambda\), so going forward we assume that is the case. The set of nodes  \(\lambda/\mu\) in the complement is called a {\it skew diagram}, and \(S^{\lambda/\mu}_\CC\) is called a {\it skew Specht module}. As a \(\CC\)-vector space, \(S^{\lambda/\mu}_\CC\) has basis in correspondence with standard \(\lambda/\mu\)-tableaux, and there is an analogue of Young's orthogonal form for skew Specht modules, see e.g. \cite[\S2]{klesh}. When \(\mathcal{O}=\mathbb{F}\) is a field of positive characteristic, semisimplicity fails, but skew Specht modules still arise as subquotients of restrictions of Specht modules to Young subgroups. James and Peel studied the structure of skew Specht modules \(S^{\lambda/\mu}_{\mathbb{F}}\) over \(\mathbb{F}\mathfrak{S}_n\) in \cite{JP}.

More generally, to an \(l\)-multipartition \(\blam\), one may associate a Specht module \(S^{\blam}\) over a cyclotomic Hecke algebra of level \(l\), of which the group algebra of \(\mathfrak{S}_d\) is a special (level one) case. Brundan and Kleshchev \cite{bk} showed that over an arbitrary field such algebras are isomorphic to a certain cyclotomic quotient \(R_d^\Lambda\) of the Khovanov-Lauda-Rouquier (KLR) algebra \(R_d = \bigoplus_{\textup{ht}(\alpha)=d}R_\alpha\). Importantly, KLR algebras and their cyclotomic quotients carry a grading, thereby allowing one to consider the {\it graded} representation theory of cyclotomic Hecke algebras via this isomorphism. In \cite{bkw}, Brundan, Kleshchev and Wang showed Specht modules are gradable.  In \cite{kmr}, Kleshchev, Mathas and Ram gave a presentation for \(S^{\blam}\) over \(R_\alpha\), in terms of a `highest weight' generator \(v^{\blam}\) and relations which include a homogeneous version of the classical Garnir relations for Specht modules. 

In this paper we define graded skew Specht modules over \(R_\alpha\) by extending, in the most obvious way, the presentation of \cite{kmr} to skew diagrams \(\blam/\bmu\). We prove that this yields a graded \(R_{\alpha}\)-module \(S^{\blam/\bmu}\) with homogeneous basis in correspondence with standard \(\blam/\bmu\)-tableaux. We show that for \(\blam\) of {\it content} \(\beta+ \alpha\), the \(R_\beta \otimes R_\alpha\)-module \(\Res_{\beta,\alpha}S^{\blam}\) has an explicit graded filtration with subquotients of the form \(S^{\bmu} \boxtimes S^{\blam/\bmu}\). This result may be compared with \cite[Theorem 3.1]{JP}, which gives a similar result for restrictions of classical Specht modules over the symmetric group algebra to Young subgroups. However, the connection between the skew Specht \(R_\alpha\)-modules defined in this paper and the skew Specht \(\mathbb{F}\mathfrak{S}_n\)-modules considered in \cite{JP} is not as direct as may be expected, see Remark \ref{cycconnect}.

Our motivation for constructing graded skew Specht modules arose from the study of {\it cuspidal modules} over KLR algebras of affine type {\tt A}. The theory of cuspidal systems for affine KLR algebras, described in \cite{cusp}, \cite{imagsw}, \cite{mcn}, building on the ideas of \cite{lyndon}, \cite{mcnfin} for finite types, provides for a classification of irreducible modules over \(R_\alpha\), and plays a pivotal role in stratification theory of \(R_\alpha\) and categorification of PBW bases for the quantum group \cite{KMStrat}, \cite{mcn}, \cite{TW}. 

Cuspidal modules (and semicuspidal modules) are the building blocks of this theory; to every positive real root \(\alpha \in \Phi_+\), one associates an irreducible \(R_\alpha\)-module \(L_\alpha\) characterized by specific properties (see \S\ref{cuspidalsec}). We show that under a balanced convex preorder (part of the data of the cuspidal system), \(L_\alpha\) is isomorphic to \(S^{\lambda/\mu}\) up to some shift, where \(\lambda/\mu\) is a {\it skew hook diagram} described by an inductive process and dependent on the chosen preorder. This gives a presentation for cuspidal modules, along with a description of the graded character which can be read off from the skew hook diagram. This result can be seen as an affine analogue of a result by Kleshchev and Ram \cite[\S8.4]{lyndon}, which showed that in {\it finite} type \({\tt A}\), the cuspidal modules are Specht modules associated to certain hook partitions.

In Section \ref{prelimsec} we collect combinatorial facts and notation surrounding Young diagrams, skew diagrams, and their tableaux. In Section \ref{klrsec} we recall the definition of the KLR algebra \(R_\alpha\), and some facts about its representation theory. In Section \ref{skewspechtsec} we define skew permutation modules and skew Specht modules. Readers familiar with the construction in \cite{kmr} and the `spanning' half of the Specht module basis theorem, and to whom the phrase `extend to skew diagrams in the obvious way' makes sense, may reasonably skip ahead to Section \ref{linind}, as up to this point most of the arguments from \cite{kmr} carry over to the skew diagram case with little significant alteration. In the key Section \ref{linind}, we prove that skew Specht modules arise as subquotients of restrictions of Specht modules, and complete the `linear independence' half of the basis theorem for skew Specht modules. In Section \ref{joinablesec}, we prove some useful results on characters of certain skew Specht modules. In Section \ref{cuspidalsec} we briefly describe the theory of cuspidal systems, along with some related notions we'll need in Section \ref{cuspidalmodsec}, where we demonstrate the connection between cuspidal modules and skew hook Specht modules.

\subsection*{Acknowledgements} This paper was written under the supervision of Alexander Kleshchev at the University of Oregon. The author would like to thank Dr. Kleshchev for his helpful comments and guidance.



\section{Preliminaries}\label{prelimsec}
\subsection{Lie theoretic notation}\label{lie}
We use notation similar to \cite{kmr}, \cite{cusp}. Let \(e \in \{0,2,3,4, \ldots\}\) and \(I = \ZZ/e\ZZ\). Let \(\Gamma\) be the quiver with vertex set \(I\) and a directed edge \(i \to j\) if \(j=i-1 \pmod{e}\). Thus \(\Gamma\) is a quiver of type \({\tt A}_\infty\) if \(e=0\) or \({\tt A}^{(1)}_{e-1}\) if \(e>0\). The corresponding {\it Cartan matrix} \({\tt C} = (a_{i,j})_{i,j \in I}\) is defined by
\begin{align*}
a_{i,j} := \begin{cases}
2 & \textup{if } i=j;\\
0 &  \textup{if }j \neq i, i\pm 1;\\
-1 &  \textup{if }i \to j \textup{ or } i\leftarrow j; \\
-2 &  \textup{if }i \leftrightarrows j.
\end{cases}
\end{align*}
Let \((\mathfrak{h},\Pi,\Pi^\vee)\) be a realization of  \((a_{i,j})_{i,j \in I}\), with root system \(\Phi\), positive roots \(\Phi_+\), simple roots \(\{ \alpha_i \mid i \in I\}\), fundamental dominant weights \(\{\Lambda_i \mid i \in I\}\), and normalized invariant form \((\cdot, \cdot)\) such that \((\alpha_i, \alpha_j) = a_{ij}\) and \((\Lambda_i, \alpha_j) = \delta_{i,j}\). Let \(P_+\) be the set of dominant integral weights, and let \(Q_+ := \bigoplus_{i \in I} \ZZ_{\geq 0} \alpha_i\) be the positive root lattice. For \(\alpha = \sum_{i \in I} m_i\alpha_i\in Q_+\), define the {\it height} of \(\alpha\) to be \(\textup{ht}(\alpha)=\sum_{i \in I} m_i\). When \(e>0\), we label the {\it null-root} \(\delta= \sum_{i \in I} \alpha_i\). Finally fix a level \(l \in \ZZ_{>0}\) and a {\it multicharge} \(\kappa=(k_1, \ldots, k_l) \in I^l\). 

\subsection{Words} Sequences of elements of \(I\) will be called {\it words}, and the set of all words is denoted \(\langle I \rangle\). If \(\bi = i_1 \cdots i_d \in \langle I \rangle\), then \(|\bi|:= \alpha_{i_1} + \cdots + \alpha_{i_d} \in Q_+\). For \(\alpha \in Q_+\), denote
\begin{align*}
\langle I \rangle_\alpha := \{ \bi \in \langle I \rangle \mid |\bi| = \alpha\}.
\end{align*}
If \(\alpha\) is of height \(d\), then \(\mathfrak{S}_d\) with simple transpositions \(s_1, \ldots, s_{d-1}\) has a left action on  \(\langle I \rangle_\alpha\) via place permutations.

\subsection{Young diagrams}
An \(l\)-multipartition \(\blam\) of \(d\) is an \(l\)-tuple of partitions \((\lambda^{(1)}, \ldots, \lambda^{(l)})\) such that \(\sum_{i=1}^l |\lambda^{(i)}| = d\). For \(1 \leq i \leq l\), let \(n(\blam,i)\) be the number of nonzero parts of \(\lambda^{(i)}\). When \(l=1\), we will usually write \(\blam = \lambda = \lambda^{(1)}\). The {\it Young diagram} of the partition \(\blam\) is
\begin{align*}
\{(a,b,m) \in \ZZ_{> 0} \times \ZZ_{> 0}\times \{1,\ldots,l\} \mid 1 \leq b \leq \lambda_a^{(m)}\}. 
\end{align*}
We call the elements of this set {\it nodes} of \(\blam\). We will usually identify the multipartition with its Young diagram. To each node \(A=(a,b,m)\) we associate its {\it residue}
\begin{align*}
\textup{res} A = \textup{res}^\kappa A = k_m + (b-a) \pmod{e}.
\end{align*}
An \(i\)-{\it node} is a node of residue \(i\). The {\it residue content of} \(\blam\) is 
\(\textup{cont}(\blam):= \sum_{A \in \blam} \alpha_{\textup{res} A} \in Q_+.\)
Denote 
\begin{align*}
\mathscr{P}^\kappa_\alpha := \{ \blam \in \mathscr{P}^\kappa \mid \textup{cont}(\blam) = \alpha\}, \hspace{10mm} (\alpha \in Q_+).
\end{align*}
and set \(\mathscr{P}^\kappa_d := \bigcup_{\textup{ht}(\alpha)=d}\mathscr{P}^\kappa_\alpha\). For \(\blam, \bmu \in \mathscr{P}^\kappa_d\), we say \(\blam\) {\it dominates} \(\bmu\), and write \(\blam \trianglerighteq \bmu\), if 
\begin{align*}
\sum_{a=1}^{m-1}|\lambda^{(a)}| + \sum_{b=1}^c \lambda_b^{(m)} \geq \sum_{a=1}^{m-1}|\mu^{(a)}| + \sum_{b=1}^c \mu_b^{(m)}
\end{align*}
for all \(1\leq m \leq l\) and \(c\geq 1\).

A node \(A \in \blam\) is {\it removable} if \(\blam \backslash \{A\}\) is a Young diagram, and a node \(B \notin \blam\) is {\it addable} if \(\blam \cup \{B\}\) is a Young diagram. Define  \(\blam_A:= \blam \backslash \{A\}\) and \(\blam^B:= \blam \cup \{B\}\).

Let \(\blam' = (\lambda^{(l)'}, \ldots, \lambda^{(1)'})\) signify the {\it conjugate partition} to \(\blam\), where \(\lambda^{(i)'}\) is obtained by swapping the rows and columns of \(\lambda^{(i)}\).

\subsection{Tableaux} Let \(\blam \in \mathscr{P}^\kappa_d\). A \(\blam\)-tableau \(\TTT\) is an injective map \(\TTT:\{1, \ldots, d\} \to \blam\), i.e. a labeling of the nodes of \(\blam\) with the integers \(1, \ldots, d\). We also label the inverse of this bijection with \(\TTT\); if \(\TTT(r) = (a,b,m)\) we will also write \(\TTT(a,b,m) = r\). We set \(\textup{res}_{\TTT}(r) = \textup{res}\, \TTT(r)\). The {\it residue sequence} of \(\TTT\) is the word
\begin{align*}
\bi(\TTT) = \bi^k(\TTT) = \textup{res}_{\TTT}(1)\cdots\textup{res}_{\TTT}(d) \in \langle I \rangle.
\end{align*}

A \(\blam\)-tableau is {\it row-strict} if \(\TTT(a,b,m) < \TTT(a,c,m)\) when \(b<c\), and {\it column-strict} if \(\TTT(a,b,m) < \TTT(c,b,m)\) when \(a<c\). We say \(\TTT\) is standard if it is row- and column-strict. Let \(\Tab(\blam)\) (resp. \(\St(\blam)\)) be the set of all (resp. standard) \(\blam\)-tableaux.

Let \(\blam \in \mathscr{P}^\kappa\), \(i \in I\), \(A\) be a removable \(i\)-node, and \(B\) be an addable \(i\)-node of \(\blam\). We set
\begin{align*}
d_A(\blam)&:= \#\{ \textup{addable \(i\)-nodes strictly below } A\} - \#\{\textup{removable \(i\)-nodes strictly below } A\}\\
d^B(\blam)&:= \#\{ \textup{addable \(i\)-nodes strictly above } B\} - \#\{\textup{removable \(i\)-nodes strictly above } B\}.
\end{align*}

In \cite[Section 3.5]{bkw}, the {\it degree} of \(\TTT\) is defined inductively as follows. If \(d=0\), then \(\TTT= \varnothing\) and  \(\deg \TTT:=0\). For \(d>0\), let \(A\) be the node occupied by \(d\) in \(\TTT\). Let \(\TTT_{<d} \in \St(\blam_A)\) be the tableau obtained by removing this node, and set
\begin{align*}
\deg \TTT := d_A(\blam) + \deg \TTT_{<d}.
\end{align*}
Similarly, define the dual notion of {\it codegree} of \(\TTT\) by \(\codeg\; \varnothing=0\) and
\begin{align*}
\codeg\; \TTT:= d^A(\blam) + \codeg\;\TTT_{<d}.
\end{align*}

The group \(\mathfrak{S}_d\) acts on the set of \(\blam\)-tableaux on the left by acting on entries; considering \(\TTT\) as a function \(\blam \to \{1,\ldots,d\}\), we have \(w\cdot \TTT = w \circ \TTT\). Let \(\TTT^{\blam} \) be the \(\blam\)-tableau in which the numbers \(1,2,\ldots, d\) appear in order from left to right along the successive rows, starting from the top. Let \(\TTT_{\blam}:= (\TTT^{\blam})'\), where we define the conjugate tableau in the obvious way.

For each \(\blam\)-tableau \(\TTT\), define permutations \(w^\TTT\) and \(w_\TTT \in \mathfrak{S}_d\) such that
\begin{align*}
w^\TTT \TTT^{\blam} = \TTT = w_\TTT \TTT_{\blam}.
\end{align*}

\subsection{Bruhat order}
Let \(\ell\) be the length function on \(\mathfrak{S}_d\) with respect to the Coxeter generators \(s_1, \ldots, s_{d-1}\). Let \(\trianglelefteq\) be the Bruhat order on \(\mathfrak{S}_d\), so that \(1 \trianglelefteq w\) for all \(w \in \mathfrak{S}_d\). Define a partial order \(\trianglelefteq\) on \(\St(\blam)\) as follows:
\begin{align*}
\SSS \trianglelefteq \TTT \iff w^\SSS \trianglelefteq w^\TTT.
\end{align*}


\subsection{Skew diagrams and tableaux} Let \(\blam, \bmu \in \mathscr{P}^\kappa\), with \(\bmu \subset \blam\) as Young diagrams. Then we call \(
\blam/\bmu := \blam \backslash \bmu
\)
a {\it skew diagram}. A (level one) skew diagram is called a {\it skew hook} if it is connected and does not have two nodes on the same diagonal. We may consider a Young diagram as a skew diagram with empty inner tableau. With \(\bmu\) fixed, let \(\mathscr{S}^{\kappa}_{\bmu,d}\) be the of skew diagrams \(\blam/\bmu\) such that \(|\blam/\bmu|=d\). Let \(\mathscr{S}^{\kappa}_{\bmu}=\bigcup \mathscr{S}^{\kappa}_{\bmu,d}\). Residue and content for skew diagrams are defined as before; for example 
\(\textup{cont}(\blam/\bmu) := \sum_{ A \in \blam/\bmu} \alpha_{\textup{res} A} \in Q_+.\)
Denote 
\begin{align*}
\mathscr{S}^{\kappa}_{\bmu, \alpha} = \{ \blam/\bmu\in \mathscr{S}^\kappa_{\bmu} \mid \textup{cont}(\blam/\bmu)=\alpha\}.
\end{align*}
For \(\blam/\bmu, \bnu/\bmu \in \mathscr{S}^\kappa_{\bmu}\), we say that \(\blam/\bmu\) dominates \(\bnu/\bmu\), or \(\blam/\bmu \trianglerighteq \bnu/\bmu\), if \(\blam \trianglerighteq \bnu\).

For \(\blam/\bmu \in \mathscr{S}^\kappa_{\bmu,d}\), a \(\blam/\bmu\)-tableau is a bijection \(\ttt:\{1, \ldots, d\} \to \blam/\bmu\). Let \(\textup{Tab}(\blam/\bmu)\) be the set of \(\blam/\bmu\)-tableaux. We define the residue sequence of \(\bi(\ttt)\) in the same manner as for Young tableaux, and \(\ttt^{\blam/\bmu}\) we define to be the \(\blam/\bmu\)-tableau in which the numbers \(1, \ldots, d\) appear in order from left to right, starting from the top. We will write \(\bi^{\blam/\bmu}:= \bi(\ttt^{\blam/\bmu})\). For every \(\blam/\bmu\)-tableau \(\ttt\), define a \(\blam\)-tableau \(\YYY(\ttt)\) by setting \(\YYY(\ttt)(a,b,m)=\TTT^{\bmu}(a,b,m)\) for \((a,b,m) \in \bmu\) and \(\YYY(\ttt)(a,b,m) = \ttt(a,b,m) + |\bmu|\) for \((a,b,m) \in \blam/\bmu\). For example, if \(l=1\), \(\lambda=(4,4,1)\),  and \(\mu=(2,1,1)\), then
\begin{align*}
\ttt^{\lambda/\mu}={\Yvcentermath1\young(:::12,::345)}\;, 
\hspace{5mm}\text{ and } \hspace{5mm}
\YYY(\ttt^{\lambda/\mu})={\Yvcentermath1\young(1256,3789,4)}\;.
\end{align*}

Let \(\textup{St}(\blam/\bmu)\) be the set of standard (i.e. row- and column-strict) \(\blam/\bmu\)-tableaux. For \(\ttt \in \textup{St}(\blam/\bmu)\), we define 
\begin{align*}
\deg \ttt:= \deg \YYY(\ttt) - \deg \TTT^{\bmu}.
\end{align*}

The symmetric group \(\mathfrak{S}_d\) acts on \(\Tab(\blam/\bmu)\) in the obvious fashion. For \(\ttt \in \textup{Tab}(\blam/\bmu)\), define \(w^\ttt\) by \(w^\ttt \ttt^{\blam/\bmu} = \ttt\). Define a partial order on \(\textup{Tab}(\blam/\bmu)\) as follows:
\begin{align*}
\sss \trianglelefteq \ttt \hspace{5mm}\textup{if and only if}\hspace{5mm} w^\sss \trianglelefteq w^\ttt.
\end{align*}

\begin{lem} \label{skeworder}
Let \(\sss, \ttt \in \Tab(\blam/\bmu)\). Then \(\sss \trianglelefteq \ttt\) if and only if \(\YYY(\sss) \trianglelefteq \YYY(\ttt)\).
\end{lem}
\begin{proof}
Let \(\widehat{w^\ttt}\) be the image of \(w^\ttt\) under the `right side' embedding
\(\mathfrak{S}_d \hookrightarrow \mathfrak{S}_{|\bmu|} \times \mathfrak{S}_d \hookrightarrow \mathfrak{S}_{|\blam|}\). Then \(w^{\YYY(\ttt)} = \widehat{w^\ttt}w^{\YYY(\ttt^{\blam/\bmu})}\), with \(\ell(w^{\YYY(\ttt)}) =\ell( \widehat{w^\ttt})+\ell(w^{\YYY(\ttt^{\blam/\bmu})})\), and similarly for \(w^{\YYY(\sss)}\). Since \(w^\sss \trianglelefteq w^\ttt\) if and only if \(\widehat{w^\sss} \trianglelefteq \widehat{w^\ttt}\), the result follows.
\end{proof}

\begin{remark} \label{orderremarks} In order to translate between the orders in the various papers cited, we provide the following dictionary. Our partial order on partitions and tableaux agrees with that of \cite{bkw}. In \cite{kmr} the order on tableaux (which we'll call \(\trianglelefteq_U\)) amounts to \(\SSS \trianglelefteq_U \TTT \iff w^{\SSS'} \trianglelefteq w^{\TTT'}\). As is shown in \cite[Lemma 2.18(ii)]{kmr}, when \(\SSS,\TTT \in \textup{St}(\bmu)\), we have \(\SSS \trianglelefteq_U \TTT \iff \SSS \trianglerighteq \TTT\). In \cite{mathas}, the reverse Bruhat order (\(1 \geq w\)) is used on elements of \(\mathfrak{S}_d\), and the order on tableaux (which we'll call \(\trianglelefteq_M\)) is defined (on row-strict tableaux) by the shape condition in Lemma \ref{ordershape}. Thus Lemma \ref{ordershape} will give \(\SSS \trianglelefteq_M \TTT \iff \SSS \trianglerighteq \TTT\) when \(\SSS, \TTT\) are row-strict.
\end{remark}

For nodes \(A,B\) in \(\blam/\bmu\), we say that \(A\) is {\it earlier} than \(B\) if \(\ttt^{\blam/\bmu}(A) < \ttt^{\blam/\bmu}(B)\); i.e. \(A\) is above or directly to the left of \(B\), or in an earlier component.

Let \(\TTT\) be a \(\blam\)-tableau and suppose that  \(r=\TTT(a_1,b_1,m_1)\) and \(s=\TTT(a_2,b_2,m_2)\). We write \(r \nearrow_\TTT s\) if \(m_1 = m_2\), \(a_1 > a_2\) and \(b_1 < b_2\); informally, if \(r\) and \(s\) are in the same component of \(\blam\) and \(s\) is strictly to the northeast of \(r\). We write \(r \rotatebox[origin=c]{45}{$\Rightarrow$}_{\hspace{-1mm}\TTT} s\) if \(r \nearrow_\TTT s\) or \(m_1 > m_2\). Other rotations of the symbols \(\nearrow_\TTT\) and \(\rotatebox[origin=c]{45}{$\Rightarrow$}_{\hspace{-1mm}\TTT}\) have similarly obvious meanings.

The following lemmas are proved in \cite{bkw} and \cite{mathas} in the context of Young diagrams, but the proofs carry over to skew shapes without significant alteration. The first lemma is obvious.

\begin{lem}\label{arrows}
Let \(\ttt \in \St(\blam/\bmu)\). Then \(s_r \ttt \in \St(\blam/\bmu)\) if and only if \(r \rotatebox[origin=c]{45}{$\Rightarrow$}_{\hspace{-1mm}\TTT} r+1\), or \(r+1 \rotatebox[origin=c]{45}{$\Rightarrow$}_{\hspace{-1mm}\TTT} r\).
\end{lem}

\begin{lem}\label{transpositions}
Let \(\sss,\ttt \in \textup{Tab}(\blam/\bmu)\). Then \(\sss \trianglelefteq \ttt\) if and only if \(\sss = (a_1 b_1) \cdots (a_r b_r) \ttt\) for some transpositions \((a_1 b_1), \ldots, (a_r b_r)\) such that for each \(1 \leq n \leq r\) we have \(a_n < b_n\) and \(b_n\) is in an earlier node in \((a_{n+1} b_{n+1}) \cdots (a_r b_r)\ttt\) than \(a_n\).
\end{lem}
\begin{proof}
This follows from applying Lemma \ref{skeworder} and \cite[Lemma 3.4]{bkw} to \(\YYY(\sss)\) and \(\YYY(\ttt)\).
\end{proof}
Given \(\blam/\bmu \in \mathscr{S}^\kappa_{\bmu}\) and a row-strict \(\ttt \in \textup{Tab}(\blam/\bmu)\), for all \(1 \leq a \leq d\) define \(\ttt_{\leq a}\) to be the tableau obtained by erasing all nodes occupied by entries greater than \(a\).
\begin{lem} \label{ordershape} 
Let \(\sss,\ttt\) be row-strict \(\blam/\bmu\)-tableaux. Then \(\sss \trianglelefteq \ttt\) if and only if \(\textup{sh}(\YYY(\sss)_{\leq a}) \trianglerighteq \textup{sh}(\YYY(\ttt)_{\leq a})\) for each \(a=|\bmu|+1,\ldots, |\bmu| + d\).
\end{lem}
\begin{proof}
This follows from Lemma \ref{skeworder} and \cite[Theorem 3.8]{mathas}.
\end{proof}
\begin{lem}\label{srt}
Let \(\blam/\bmu \in \mathscr{S}^\kappa_{\bmu}\) and \(\sss,\ttt \in \St(\blam/\bmu)\), and \(r \in \{1, \ldots, d-1\}\) such that \(r \downarrow_\ttt r+1\) or \(r \rightarrow_\ttt r+1\). Then \(\sss \triangleleft s_r\ttt\) implies \(\sss \trianglelefteq \ttt\).
\end{lem}
\begin{proof}
By \cite[Lemma 3.7]{bkw}, \(\YYY(\sss) \triangleleft \YYY(s_r\ttt)=s_{r+|\bmu|}\YYY(\ttt)\) if and only if \(\YYY(\sss) \trianglelefteq \YYY(\ttt)\), and the result follows by Lemma \ref{skeworder}.
\end{proof}


\section{KLR algebras}\label{klrsec}
\subsection{Definition}
Let \(\mathcal{O}\) be a commutative ring with identity, let \(\alpha \in Q_+\), and assume \(e>0\) (resp. \(e=0\)). The {\it KLR algebra} \(R_{\alpha} = R_\alpha(\mathcal{O})\)  of type \({\tt A}^{(1)}_{e-1}\) (resp. \({\tt A}_\infty\)) is an associative graded unital \(\mathcal{O}\)-algebra generated by 
\begin{align*}
\{ 1_{\bi} \mid  \bi \in \langle I \rangle_\alpha\} \cup \{y_1, \ldots, y_d\} \cup \{\psi_1,\ldots, \psi_{d-1}\}
\end{align*}
and subject only to the following relations:
\begin{gather*}
1_{\bi}1_{\bj}=\delta_{\bi,\bj}1_{\bi};\hspace{15mm}
\sum_{\bi \in \langle I \rangle_\alpha} 1_{\bi} =1;\hspace{15mm}
y_r1_{\bi}=1_{\bi}y_r;\hspace{15mm}
y_ry_t = y_t y_r;\\
\psi_r1_{\bi} = 1_{s_r\bi}\psi_r;\hspace{15mm}
(y_t \psi_r - \psi_r y_{s_r(t)})1_{\bi} = \delta_{i_r,i_{r+1}}(\delta_{t,r+1}-\delta_{t,r})1_{\bi};
\end{gather*}
\begin{align*}
\psi_r^21_{\bi}&= \begin{cases}
0 & a_{i_r,i_{r+1}}=2\\
1_{\bi} & a_{i_r,i_{r+1}}=0\\
(\delta_{i_r+1,i_{r+1}}- \delta_{i_r-1,i_{r+1}})(y_r-y_{r+1})1_{\bi} & a_{i_r,i_{r+1}}=-1\\
(y_r-y_{r+1})(y_{r+1}-y_r)1_{\bi} & a_{i_r,i_{r+1}}=-2
\end{cases}
\\
(\psi_{r+1}\psi_r\psi_{r+1} - \psi_r \psi_{r+1} \psi_r)1_{\bi} &=
\begin{cases}
(\delta_{i_r+1,i_{r+1}}- \delta_{i_r-1,i_{r+1}})1_{\bi} &i_r = i_{r+2}, a_{i_r,i_{r+1}}=-1\\
(-y_r + 2y_{r+1} - y_{r+2})1_{\bi} &  i_r =i_{r+2}, a_{i_r,i_{r+1}}=-2\\
0 & \textup{otherwise}.
\end{cases}
\end{align*}
The grading on \(R_\alpha\) is defined by setting
\begin{align*}
\deg(1_{\bi})=0, \hspace{10mm} \deg(y_r)=2, \hspace{10mm} \deg(\psi_r 1_{\bi})=-a_{i_r,i_{r+1}}.
\end{align*}
\begin{remark}
In \cite{kl}, Khovanov and Lauda present a convenient diagrammatic approach to computations in \(R_\alpha\) which will be used often in arguments in this paper.
\end{remark}
\subsection{Properties}
Let \(\alpha \in Q_+\) and \(\textup{ht}(\alpha)=d\). For every \(w \in \mathfrak{S}_d\), fix a {\it preferred reduced expression} \(w=s_{r_1} \cdots s_{r_m}\), and define \(\psi_w = \psi_{r_1} \cdots \psi_{r_m} \in R_\alpha\). In general \(\psi_w\) depends on the choice of reduced expression. When \(w\) is {\it fully commutative} however, i.e., when one can go from any reduced expression for \(w\) to any other using only the braid relations of the form \(s_r s_t = s_t s_r\) for \(|r-t|>1\), the element \(\psi_w\) is independent of the choice of reduced expression.

For \(\blam/\bmu \in \mathscr{S}^\kappa_{\bmu,\alpha}\) and \(\ttt \in \textup{Tab}(\blam/\bmu)\), define
\begin{align*}
\psi^\ttt := \psi_{w^\ttt}.
\end{align*}
\begin{lem}\label{degmatch}
Let \(\ttt \in \St(\blam/\bmu)\). If \(w^\ttt=s_{r_1} \cdots s_{r_m}\) is a reduced decomposition in \(\mathfrak{S}_d\), then 
\begin{align*}
\deg \ttt -\deg \ttt^{\blam/\bmu}  = \deg(\psi_{r_1} \cdots \psi_{r_m} 1_{\bi^{\blam/\bmu}}).
\end{align*}
\end{lem}
\begin{proof}
Write \(c=|\bmu|\). Let \(s_{t_1} \cdots s_{t_n}\) be a reduced decomposition for \(w^{\YYY(\ttt^{\blam/\bmu})}\). Then \(\widehat{w^\ttt} = s_{r_1+c} \cdots s_{r_m+c}\) is reduced and \(w^{\YYY(\ttt)}=\widehat{w^\ttt}w^{\YYY(\ttt^{\blam/\bmu})}=s_{r_1+c} \cdots s_{r_m+c}s_{t_1} \cdots s_{t_n}\) is reduced. Then by \cite[Corollary 3.13]{bkw} we have 
\begin{align*}
\deg \YYY(\ttt)-\deg \TTT^{\blam} &= \deg(\psi_{r_1+c} \cdots \psi_{r_m+c} \psi_{t_1} \cdots \psi_{t_n} 1_{\bi^{\blam}})\\
&= \deg(\psi_{r_1+c} \cdots \psi_{r_m+c}1_{\bi^{\YYY(\ttt^{\blam/\bmu})}}) + \deg(\psi_{t_1} \cdots \psi_{t_n}1_{\bi^{\blam}})\\
&= \deg(\psi_{r_1} \cdots \psi_{r_m}1_{\bi^{\blam/\bmu}}) + \deg(\psi_{t_1} \cdots \psi_{t_n}1_{\bi^{\blam}})
\end{align*}
and
\begin{align*}
\deg \YYY(\ttt^{\blam/\bmu})-\deg \TTT^{\blam}  = \deg(\psi_{t_1} \cdots \psi_{t_n} 1_{\bi^{\blam}}),
\end{align*}
which implies the result.
\end{proof}
\begin{prop}\label{rewrite}
Let \(f(y)=f(y_1, \ldots, y_d) \in \mathcal{O}[y_1, \ldots y_d]\) be a polynomial in the generators \(y_r\) of \(R_\alpha\). Let \(1 \leq r_1, \ldots, r_m \leq d-1\). Then
\begin{enumerate}
\item[\textup{(i)}] \(f(y)\psi_{r_1} \cdots \psi_{r_m} 1_{\bi}\) is an \(\mathcal{O}\)-linear combination of elements of the form \(\psi_{r_1}^{\epsilon_1} \cdots \psi_{r_m}^{\epsilon_m} g(y) 1_{\bi}\), where \(g(y) \in \mathcal{O}[y_1,\ldots,y_d]\), each \(\epsilon_i \in \{0,1\}\), and \(s_{r_1}^{\epsilon_1} \cdots s_{r_m}^{\epsilon_m}\) is a reduced expression.
\item[\textup{(ii)}] If \(w=s_{r_1} \cdots s_{r_m}\) is reduced, and \(s_{t_1} \cdots s_{t_m}\) is another reduced expression for \(w\), then 
\begin{align*}
\psi_{r_1} \cdots \psi_{r_m} 1_{\bi} = \psi_{t_1} \cdots \psi_{t_m} 1_{\bi} + \sum_{u \triangleleft w} d_u\psi_ug_u(y) 1_{\bi},
\end{align*}
where each \(d_u \in \mathcal{O}\), \(g_u(y) \in \mathcal{O}[y_1, \ldots, y_d]\), and each \(u\) in the sum is such that \(\ell(u) \leq m-3\). Alternatively, 
\begin{align*}
\psi_{r_1} \cdots \psi_{r_m} 1_{\bi} = \psi_{t_1} \cdots \psi_{t_m} 1_{\bi} + (*),
\end{align*}
where \((*)\) is an \(\mathcal{O}\)-linear combination of elements of the form \(\psi_{r_1}^{\epsilon_1} \cdots \psi_{r_m}^{\epsilon_m} g(y)\), where \(g(y)\in \mathcal{O}[y_1, \ldots, y_d]\), \(\epsilon_i \in \{0,1\}\), \(\epsilon_i = 0\) for at least three distinct \(i\)'s, and \(s_{r_1}^{\epsilon_1} \cdots s_{r_m}^{\epsilon_m}\) is a reduced expression.
\end{enumerate}
\end{prop}
\begin{proof}
This is proved in \cite[Lemma 2.4]{bkw} and \cite[Proposition 2.5]{bkw}, for the case of cyclotomic KLR algebras, but the cyclotomic relation is not used in the proof.
\end{proof}
\begin{thm}\label{KLRbasis}\textup{\cite[Theorem 2.5]{kl}, \cite[Theorem 3.7]{rouq}} Let \(\alpha \in Q_+\). Then
\begin{align*}
\{\psi_w y_1^{m_1} \cdots y_d^{m_d} 1_{\bi} \mid w \in \mathfrak{S}_d, m_1, \ldots, m_d \in \ZZ_{\geq 0}, \bi \in \langle I \rangle_\alpha\}
\end{align*}
is an \(\mathcal{O}\)-basis for \(R_\alpha\).
\end{thm}
\subsection{Representation theory of \(R_\alpha\)}
Let \(\mathscr{A}=\ZZ[q,q^{-1}]\). Denote by \(R_\alpha\)-mod the abelian category of finite dimensional graded left \(R_\alpha\)-modules, with degree-preserving module homomorphisms. Let \([R_\alpha\textup{-mod}]\) denote the corresponding Grothendieck group. \([R_\alpha\textup{-mod}]\) is a \( \mathscr{A}\)-module via \(q^m[M]:= [M\langle m \rangle]\), where \(M\langle m\rangle\) denotes the module obtained by shifting the grading up by \(m\). The {\it graded dimension} of a module \(M\) is 
\begin{align*}
\dim_q M:= \sum_{n \in \ZZ} (\dim V_n)q^n \in  \mathscr{A}.
\end{align*}
Every irreducible graded \(R_\alpha\)-module is finite dimensional \cite[Proposition 2.12]{kl}, and there are finitely many irreducible \(R_\alpha\)-modules up to grading shift \cite[\S 5.2]{kl}. For every irreducible module \(L\) there is a unique choice of grading shift so that \(L^{\circledast} \cong L\), and in particular, the grading of \(L\) is symmetric about zero. We often assume that the grading of irreducible modules has been chosen in accordance with this choice.

For \(\bi \in \langle I \rangle_\alpha\) and \(M \in R_\alpha\)-mod, the {\it \(\bi\)-word space} of \(M\) is \(M_{\bi}:=1_{\bi}M\). We say \(\bi\) is a word of \(M\) if \(M_{\bi} \neq 0\). We have \(M = \bigoplus_{ \bi \in \langle I \rangle_\alpha}M_{\bi}\). The {\it character} of \(M\) is
\begin{align*}
\CH_q M := \sum_{\bi \in \langle I \rangle_\alpha} (\dim_q M_{\bi})\bi \in \mathscr{A}\langle I \rangle_\alpha.
\end{align*}
The character map \(\CH_q: R_\alpha\textup{-mod} \to  \mathscr{A}\langle I \rangle_\alpha\) factors through to give an {\it injective} \( \mathscr{A}\)-linear map \(\ch_q:[R_\alpha\textup{-mod}] \to  \mathscr{A}\langle I \rangle_\alpha\) by \cite[Theorem 3.17]{kl}.

\subsection{Induction and restriction}

Given $\alpha, \beta \in Q_+$, we set $
R_{\alpha,\beta} := R_\alpha \otimes 
R_\beta$.  
There is an injective algebra homomorphism 
$R_{\alpha,\beta}\,\hookrightarrow\, R_{\alpha+\beta}$ 
sending $1_{\bi} \otimes 1_{\bj}$ to $1_{\bi\bj}$,
where $\bi\bj$ is the concatenation of the two sequences. The image of the identity
element of $R_{\alpha,\beta}$ under this map is \(1_{\alpha,\beta}:= \sum_{\bi \in \langle I \rangle_\alpha,\ \bj \in I^\beta} 1_{\bi\bj}.\)

Let $\Ind_{\alpha,\beta}$ and $\Res_{\alpha,\beta}$
be the corresponding induction and restriction functors: 
\begin{align*}
\Ind_{\alpha,\beta}&:= R_{\alpha+\beta} 1_{\alpha,\beta}
\otimes_{R_{\alpha,\beta}} ?:R_{\alpha,\beta}\textup{-mod} \rightarrow R_{\alpha+\beta}\textup{-mod},\\
\Res_{\alpha,\beta} &:= 1_{\alpha,\beta} R_{\alpha+\beta}
\otimes_{R_{\alpha+\beta}} ?:R_{\alpha+\beta}\textup{-mod}\rightarrow R_{\alpha,\beta}\textup{-mod}.
\end{align*}
The functor $\Ind_{\alpha,\beta}$ is left adjoint to $\Res_{\alpha,\beta}$. Both functors are exact and send finite dimensional modules to finite dimensional modules. These functors have obvious generalizations to $n\geq 2$ factors. If $M_a\in R_{\gamma_a}\textup{-mod}$, for $a=1,\dots,n$, we define 
\begin{equation}
M_1\circ\dots\circ M_n:=\Ind_{\gamma_1,\dots,\gamma_n}
M_1\boxtimes\dots\boxtimes M_n. 
\end{equation}

\section{Skew Specht modules}\label{skewspechtsec}
\noindent In this section we define the graded skew Specht module \(S^{\blam/\bmu}\). In fact, the construction is exactly the same as the one given for graded (row) Specht modules (associated to Young diagrams) in \cite{kmr}, only applied in the more general context of skew diagrams. For the reader's convenience, and since the particulars will be put to use often in Section \ref{linind}, we provide the construction of skew Specht modules here. The `spanning' result  \cite[Prop. 5.14]{kmr} and the proof given in that paper also generalize without much difficulty to the skew case.

\subsection{Garnir skew tableaux} Let \(A=(a,b,m)\) be a node of \(\blam/\bmu \in \mathscr{S}^\kappa\). We say \(A\) is a {\it Garnir node} if \((a+1,b,m)\) is also a node of \(\blam/\bmu\). The {\it \(A\)-Garnir belt} \({\bf B}^A\) is the set of nodes
\begin{align*}
{\bf B}^A = \{(a,c,m) \in \blam/\bmu \mid c \geq b\} \cup \{ (a+1,c,m) \in \blam/\bmu \mid c \leq b\}.
\end{align*}
For example, if \(l=1\), \(A=(2,6,1)\), \(\lambda = (11,10,7,2,2)\), and \(\mu = (7,4,3,1)\), then \({\bf B}^A\) is highlighted below:
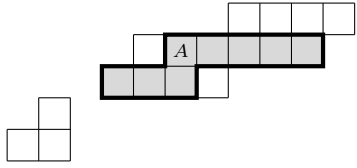
\begin{align*}
\begin{tikzpicture}[scale=0.42]
\tikzset{baseline=0mm}
\draw [ fill=gray!30] (3,-2)--(3,-3)--(6,-3)--(6,-2)--(10,-2)--(10,-1)--(5,-1)--(5,-2)--cycle;
\draw(0,-4)--(0,-5);
\draw(1,-3)--(1,-5);
\draw(2,-3)--(2,-5);
\draw(3,-2)--(3,-3);
\draw(4,-1)--(4,-3);
\draw(5,-1)--(5,-3);
\draw(6,-1)--(6,-3);
\draw(7,0)--(7,-3);
\draw(8,0)--(8,-2);
\draw(9,0)--(9,-2);
\draw(10,0)--(10,-2);
\draw(11,0)--(11,-1);
\draw(7,0)--(11,0);
\draw(4,-1)--(11,-1);
\draw(3,-2)--(10,-2);
\draw(3,-3)--(7,-3);
\draw(1,-3)--(2,-3);
\draw(0,-4)--(2,-4);
\draw(0,-5)--(2,-5);
\draw(5.5,-1.5) node{$\scriptstyle A$};
\draw [ultra thick] (3,-2)--(3,-3)--(6,-3)--(6,-2)--(10,-2)--(10,-1)--(5,-1)--(5,-2)--cycle;
\end{tikzpicture}
\end{align*}
The {\it \(A\)-Garnir tableau} is the \(\blam/\bmu\)-tableau \({\gggg}^A\) that is equal to \(\ttt^{\blam/\bmu}\) outside the Garnir belt, and with numbers \(\ttt^{\blam/\bmu}(a,b,m)\) through \(\ttt^{\blam/\bmu}(a+1,b,m)\) inserted into the Garnir belt, in order from bottom left to top right. Continuing the example, we have:
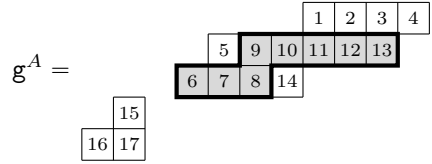
\begin{align*}
\gggg^A = 
\begin{tikzpicture}[scale=0.42]
\tikzset{baseline=-10mm}
\draw [ fill=gray!30] (3,-2)--(3,-3)--(6,-3)--(6,-2)--(10,-2)--(10,-1)--(5,-1)--(5,-2)--cycle;
\draw(0,-4)--(0,-5);
\draw(1,-3)--(1,-5);
\draw(2,-3)--(2,-5);
\draw(3,-2)--(3,-3);
\draw(4,-1)--(4,-3);
\draw(5,-1)--(5,-3);
\draw(6,-1)--(6,-3);
\draw(7,0)--(7,-3);
\draw(8,0)--(8,-2);
\draw(9,0)--(9,-2);
\draw(10,0)--(10,-2);
\draw(11,0)--(11,-1);
\draw(7,0)--(11,0);
\draw(4,-1)--(11,-1);
\draw(3,-2)--(10,-2);
\draw(3,-3)--(7,-3);
\draw(1,-3)--(2,-3);
\draw(0,-4)--(2,-4);
\draw(0,-5)--(2,-5);
\draw(7.5,-0.5) node{$\scriptstyle 1$};
\draw(8.5,-0.5) node{$\scriptstyle 2$};
\draw(9.5,-0.5) node{$\scriptstyle 3$};
\draw(10.5,-0.5) node{$\scriptstyle 4$};
\draw(4.5,-1.5) node{$\scriptstyle 5$};
\draw(5.5,-1.5) node{$\scriptstyle 9$};
\draw(6.5,-1.5) node{$\scriptstyle 10$};
\draw(7.5,-1.5) node{$\scriptstyle 11$};
\draw(8.5,-1.5) node{$\scriptstyle 12$};
\draw(9.5,-1.5) node{$\scriptstyle 13$};
\draw(3.5,-2.5) node{$\scriptstyle 6$};
\draw(4.5,-2.5) node{$\scriptstyle 7$};
\draw(5.5,-2.5) node{$\scriptstyle 8$};
\draw(6.5,-2.5) node{$\scriptstyle 14$};
\draw(1.5,-3.5) node{$\scriptstyle 15$};
\draw(0.5,-4.5) node{$\scriptstyle 16$};
\draw(1.5,-4.5) node{$\scriptstyle 17$};
\draw [ultra thick] (3,-2)--(3,-3)--(6,-3)--(6,-2)--(10,-2)--(10,-1)--(5,-1)--(5,-2)--cycle;
\end{tikzpicture}
\end{align*}

\begin{lem} \label{sameoutside} 
Suppose that \(\blam/\bmu \in \mathscr{S}^\kappa_{\bmu}\), \(A\) is a Garnir node of \(\blam/\bmu\), and \(\ttt \in \textup{Tab}(\blam/\bmu)\). If \(\ttt \trianglelefteq \gggg^A\), then \(\ttt\) agrees with \(\ttt^{\blam/\bmu}\) outside the \(A\)-Garnir belt.
\end{lem}
\begin{proof}
Since \(w^{\gggg^A}\) fixes the entries outside the Garnir belt, \(w^\ttt \leq w^{\gggg^A}\) must do the same.
\end{proof}

\subsection{Bricks} Take \(\blam/\bmu \in \mathscr{S}^\kappa_{\bmu}\) and Garnir node \(A=(a,b,m) \in \blam/\bmu\). A brick is a set of nodes
\begin{align*}
\{(c,d,m'),(c,d+1,m'), \ldots, (c,d+e-1,m')\} \subseteq {\bf B}^A
\end{align*}
such that \(\textup{res}(c,d,m')=\textup{res}(A)\). Let \(k^A\) be the total number of bricks in \({\bf B}^A\), and let \(f^A\) be the number of bricks in row \(a\) of \({\bf B}^A\). Label the bricks \(B_1^A, \ldots, B_{k^A}^A\) in order from left to right, beginning at the bottom left. 

Going back to our example, for the case \(e=2\), the bricks in \({\bf B}^A\) are as labeled below:
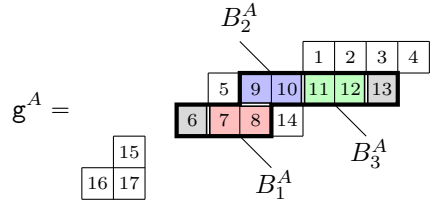
\begin{align*}
\gggg^A = 
\begin{tikzpicture}[scale=0.42]
\tikzset{baseline=-10mm}
\draw [ fill=gray!30] (3,-2)--(3,-3)--(4,-3)--(4,-2)--cycle;
\draw [ fill=gray!30] (10,-2)--(10,-1)--(9,-1)--(9,-2)--cycle;
\draw [fill=red!50, opacity=0.5] (4,-2)--(4,-3)--(6,-3)--(6,-2)--cycle;
\draw [ fill=blue!50, opacity=0.5] (5,-1)--(5,-2)--(7,-2)--(7,-1)--cycle;
\draw [fill=green!50, opacity=0.5] (7,-1)--(7,-2)--(9,-2)--(9,-1)--cycle;
\draw(0,-4)--(0,-5);
\draw(1,-3)--(1,-5);
\draw(2,-3)--(2,-5);
\draw(3,-2)--(3,-3);
\draw(4,-1)--(4,-3);
\draw(5,-1)--(5,-3);
\draw(6,-1)--(6,-3);
\draw(7,0)--(7,-3);
\draw(8,0)--(8,-2);
\draw(9,0)--(9,-2);
\draw(10,0)--(10,-2);
\draw(11,0)--(11,-1);
\draw(7,0)--(11,0);
\draw(4,-1)--(11,-1);
\draw(3,-2)--(10,-2);
\draw(3,-3)--(7,-3);
\draw(1,-3)--(2,-3);
\draw(0,-4)--(2,-4);
\draw(0,-5)--(2,-5);
\draw(7.5,-0.5) node{$\scriptstyle 1$};
\draw(8.5,-0.5) node{$\scriptstyle 2$};
\draw(9.5,-0.5) node{$\scriptstyle 3$};
\draw(10.5,-0.5) node{$\scriptstyle 4$};
\draw(4.5,-1.5) node{$\scriptstyle 5$};
\draw(5.5,-1.5) node{$\scriptstyle 9$};
\draw(6.5,-1.5) node{$\scriptstyle 10$};
\draw(7.5,-1.5) node{$\scriptstyle 11$};
\draw(8.5,-1.5) node{$\scriptstyle 12$};
\draw(9.5,-1.5) node{$\scriptstyle 13$};
\draw(3.5,-2.5) node{$\scriptstyle 6$};
\draw(4.5,-2.5) node{$\scriptstyle 7$};
\draw(5.5,-2.5) node{$\scriptstyle 8$};
\draw(6.5,-2.5) node{$\scriptstyle 14$};
\draw(1.5,-3.5) node{$\scriptstyle 15$};
\draw(0.5,-4.5) node{$\scriptstyle 16$};
\draw(1.5,-4.5) node{$\scriptstyle 17$};
\draw [semithick, double] (4,-2)--(4,-3);
\draw [semithick,double] (7,-1)--(7,-2);
\draw [semithick,double] (9,-1)--(9,-2);
\draw [semithick,double] (5,-2)--(6,-2);
\draw [ultra thick] (3,-2)--(3,-3)--(6,-3)--(6,-2)--(10,-2)--(10,-1)--(5,-1)--(5,-2)--cycle;
\draw (5,-3)--(6,-4);
\draw (6,-1)--(5,0);
\draw (8,-2)--(9,-3);
\draw (6.1,-3.8) node[below]{$B_1^A$};
\draw (5,0) node[above]{$B_2^A$};
\draw (9.1,-2.8) node[below]{$B_3^A$};
\end{tikzpicture}
\end{align*}
For \(1 \leq r \leq k^A\), let \(w_r^A \in \mathfrak{S}_d\) be the element that swaps \(B_r^A\) and \(B_{r+1}^A\). Define the group of {\it brick permutations}
\begin{align*}
\mathfrak{S}^A:= \langle w_1^A, \ldots, w_{k^A-1}^A \rangle \cong \mathfrak{S}_{k^A}.
\end{align*}
This is the trivial group if \(k^A=0\), e.g. if \(e=0\).

Let \(\Gar^A\) be the set of row-strict \(\blam/\bmu\)-tableaux which are are obtained by the action of \(\mathfrak{S}^A\) on \(\gggg^A\). All tableaux in \(\Gar^A\) save \(\gggg^A\) are standard. By Lemma \ref{ordershape}, \(\gggg^A\) is the unique maximal element of \(\Gar^A\), and there exists a unique minimal element \(\ttt^A\), which has the bricks \(B_1^A, \ldots, B_{f^A}^A\) in order from left to right in row \(a\), and the remaining bricks in order from left to right in row \(a+1\). By definition, if \(\ttt \in \Gar^A\), then \(\bi(\ttt) = \bi(\gggg^A)\). Define \(\bi^A\) as this common residue sequence.

Let \(\mathscr{D}^A\) be the set of minimal length left coset representatives of \(\mathfrak{S}_{f^A} \times \mathfrak{S}_{k^A-f^A}\) in \(\mathfrak{S}^A\). We have
\begin{align*}
\Gar^A = \{w\ttt^A \mid w \in \mathscr{D}^A\}.
\end{align*}

Continuing the example above, the elements of \(\Gar^A\) are \(\gggg^A=w_1^Aw_2^A\ttt^A\) (shown above), and the following two tableaux:
\begin{align*}
\ttt^A= 
\begin{tikzpicture}[scale=0.42]
\tikzset{baseline=-10mm}
\draw [ fill=gray!30] (3,-2)--(3,-3)--(4,-3)--(4,-2)--cycle;
\draw [ fill=gray!30] (10,-2)--(10,-1)--(9,-1)--(9,-2)--cycle;
\draw [fill=green!50, opacity=0.5] (4,-2)--(4,-3)--(6,-3)--(6,-2)--cycle;
\draw [fill=red!50,  opacity=0.5] (5,-1)--(5,-2)--(7,-2)--(7,-1)--cycle;
\draw [fill=blue!50, opacity=0.5] (7,-1)--(7,-2)--(9,-2)--(9,-1)--cycle;
\draw(0,-4)--(0,-5);
\draw(1,-3)--(1,-5);
\draw(2,-3)--(2,-5);
\draw(3,-2)--(3,-3);
\draw(4,-1)--(4,-3);
\draw(5,-1)--(5,-3);
\draw(6,-1)--(6,-3);
\draw(7,0)--(7,-3);
\draw(8,0)--(8,-2);
\draw(9,0)--(9,-2);
\draw(10,0)--(10,-2);
\draw(11,0)--(11,-1);
\draw(7,0)--(11,0);
\draw(4,-1)--(11,-1);
\draw(3,-2)--(10,-2);
\draw(3,-3)--(7,-3);
\draw(1,-3)--(2,-3);
\draw(0,-4)--(2,-4);
\draw(0,-5)--(2,-5);
\draw(7.5,-0.5) node{$\scriptstyle 1$};
\draw(8.5,-0.5) node{$\scriptstyle 2$};
\draw(9.5,-0.5) node{$\scriptstyle 3$};
\draw(10.5,-0.5) node{$\scriptstyle 4$};
\draw(4.5,-1.5) node{$\scriptstyle 5$};
\draw(5.5,-1.5) node{$\scriptstyle 7$};
\draw(6.5,-1.5) node{$\scriptstyle 8$};
\draw(7.5,-1.5) node{$\scriptstyle 9$};
\draw(8.5,-1.5) node{$\scriptstyle 10$};
\draw(9.5,-1.5) node{$\scriptstyle 13$};
\draw(3.5,-2.5) node{$\scriptstyle 6$};
\draw(4.5,-2.5) node{$\scriptstyle 11$};
\draw(5.5,-2.5) node{$\scriptstyle 12$};
\draw(6.5,-2.5) node{$\scriptstyle 14$};
\draw(1.5,-3.5) node{$\scriptstyle 15$};
\draw(0.5,-4.5) node{$\scriptstyle 16$};
\draw(1.5,-4.5) node{$\scriptstyle 17$};
\draw [semithick, double] (4,-2)--(4,-3);
\draw [semithick,double] (7,-1)--(7,-2);
\draw [semithick,double] (9,-1)--(9,-2);
\draw [semithick,double] (5,-2)--(6,-2);
\draw [ultra thick] (3,-2)--(3,-3)--(6,-3)--(6,-2)--(10,-2)--(10,-1)--(5,-1)--(5,-2)--cycle;
\end{tikzpicture}
\hspace{12mm}
 w_2^A \ttt^A=
\begin{tikzpicture}[scale=0.42]
\tikzset{baseline=-10mm}
\draw [ fill=gray!30] (3,-2)--(3,-3)--(4,-3)--(4,-2)--cycle;
\draw [ fill=gray!30] (10,-2)--(10,-1)--(9,-1)--(9,-2)--cycle;
\draw [fill=blue!50, opacity=0.5] (4,-2)--(4,-3)--(6,-3)--(6,-2)--cycle;
\draw [fill=red!50,  opacity=0.5] (5,-1)--(5,-2)--(7,-2)--(7,-1)--cycle;
\draw [fill=green!50, opacity=0.5] (7,-1)--(7,-2)--(9,-2)--(9,-1)--cycle;
\draw(0,-4)--(0,-5);
\draw(1,-3)--(1,-5);
\draw(2,-3)--(2,-5);
\draw(3,-2)--(3,-3);
\draw(4,-1)--(4,-3);
\draw(5,-1)--(5,-3);
\draw(6,-1)--(6,-3);
\draw(7,0)--(7,-3);
\draw(8,0)--(8,-2);
\draw(9,0)--(9,-2);
\draw(10,0)--(10,-2);
\draw(11,0)--(11,-1);
\draw(7,0)--(11,0);
\draw(4,-1)--(11,-1);
\draw(3,-2)--(10,-2);
\draw(3,-3)--(7,-3);
\draw(1,-3)--(2,-3);
\draw(0,-4)--(2,-4);
\draw(0,-5)--(2,-5);
\draw(7.5,-0.5) node{$\scriptstyle 1$};
\draw(8.5,-0.5) node{$\scriptstyle 2$};
\draw(9.5,-0.5) node{$\scriptstyle 3$};
\draw(10.5,-0.5) node{$\scriptstyle 4$};
\draw(4.5,-1.5) node{$\scriptstyle 5$};
\draw(5.5,-1.5) node{$\scriptstyle 7$};
\draw(6.5,-1.5) node{$\scriptstyle 8$};
\draw(7.5,-1.5) node{$\scriptstyle 11$};
\draw(8.5,-1.5) node{$\scriptstyle 12$};
\draw(9.5,-1.5) node{$\scriptstyle 13$};
\draw(3.5,-2.5) node{$\scriptstyle 6$};
\draw(4.5,-2.5) node{$\scriptstyle 9$};
\draw(5.5,-2.5) node{$\scriptstyle 10$};
\draw(6.5,-2.5) node{$\scriptstyle 14$};
\draw(1.5,-3.5) node{$\scriptstyle 15$};
\draw(0.5,-4.5) node{$\scriptstyle 16$};
\draw(1.5,-4.5) node{$\scriptstyle 17$};
\draw [semithick, double] (4,-2)--(4,-3);
\draw [semithick,double] (7,-1)--(7,-2);
\draw [semithick,double] (9,-1)--(9,-2);
\draw [semithick,double] (5,-2)--(6,-2);
\draw [ultra thick] (3,-2)--(3,-3)--(6,-3)--(6,-2)--(10,-2)--(10,-1)--(5,-1)--(5,-2)--cycle;
\end{tikzpicture}
\end{align*}
Taking \(\kappa=(1)\), the residues of nodes in \(\blam/\bmu\) are as follows:
\begin{align*}
\begin{tikzpicture}[scale=0.42]
\tikzset{baseline=-10mm}
\draw(0,-4)--(0,-5);
\draw(1,-3)--(1,-5);
\draw(2,-3)--(2,-5);
\draw(3,-2)--(3,-3);
\draw(4,-1)--(4,-3);
\draw(5,-1)--(5,-3);
\draw(6,-1)--(6,-3);
\draw(7,0)--(7,-3);
\draw(8,0)--(8,-2);
\draw(9,0)--(9,-2);
\draw(10,0)--(10,-2);
\draw(11,0)--(11,-1);
\draw(7,0)--(11,0);
\draw(4,-1)--(11,-1);
\draw(3,-2)--(10,-2);
\draw(3,-3)--(7,-3);
\draw(1,-3)--(2,-3);
\draw(0,-4)--(2,-4);
\draw(0,-5)--(2,-5);
\draw(7.5,-0.5) node{$\scriptstyle 0$};
\draw(8.5,-0.5) node{$\scriptstyle 1$};
\draw(9.5,-0.5) node{$\scriptstyle 0$};
\draw(10.5,-0.5) node{$\scriptstyle 1$};
\draw(4.5,-1.5) node{$\scriptstyle 0$};
\draw(5.5,-1.5) node{$\scriptstyle 1$};
\draw(6.5,-1.5) node{$\scriptstyle 0$};
\draw(7.5,-1.5) node{$\scriptstyle 1$};
\draw(8.5,-1.5) node{$\scriptstyle 0$};
\draw(9.5,-1.5) node{$\scriptstyle 1$};
\draw(3.5,-2.5) node{$\scriptstyle 0$};
\draw(4.5,-2.5) node{$\scriptstyle 1$};
\draw(5.5,-2.5) node{$\scriptstyle 0$};
\draw(6.5,-2.5) node{$\scriptstyle 1$};
\draw(1.5,-3.5) node{$\scriptstyle 1$};
\draw(0.5,-4.5) node{$\scriptstyle 1$};
\draw(1.5,-4.5) node{$\scriptstyle 0$};
\end{tikzpicture}
\end{align*}
In diagrammatic notation,
\begin{align*}
\psi^{\ttt^A}1_{\bi^{\blam/\bmu}}&=
\begin{braid}\tikzset{baseline=0mm}
\draw [color=red!50, fill=red!50,opacity=0.5] (6,2)--(7,2)--(8,-2)--(7,-2)--cycle;
\draw [color=blue!50, fill=blue!50,opacity=0.5] (8,2)--(9,2)--(10,-2)--(9,-2)--cycle;
\draw [color=green!50, fill=green!50,opacity=0.5] (12,2)--(13,2)--(12,-2)--(11,-2)--cycle;
\draw(1,2) node[above]{$0$}--(1,-2);
\draw(2,2) node[above]{$1$}--(2,-2);
\draw(3,2) node[above]{$0$}--(3,-2);
\draw(4,2) node[above]{$1$}--(4,-2);
\draw(5,2) node[above]{$0$}--(5,-2);
\draw(6,2) node[above]{$1$}--(7,-2);
\draw(7,2) node[above]{$0$}--(8,-2);
\draw(8,2) node[above]{$1$}--(9,-2);
\draw(9,2) node[above]{$0$}--(10,-2);
\draw(10,2) node[above]{$1$}--(13,-2);
\draw(11,2) node[above]{$0$}--(6,-2);
\draw(12,2) node[above]{$1$}--(11,-2);
\draw(13,2) node[above]{$0$}--(12,-2);
\draw(14,2) node[above]{$1$}--(14,-2);
\draw(15,2) node[above]{$1$}--(15,-2);
\draw(16,2) node[above]{$1$}--(16,-2);
\draw(17,2) node[above]{$0$}--(17,-2);
\end{braid},
\\
\psi^{w_2^A\ttt^A}1_{\bi^{\blam/\bmu}}&=
\begin{braid}\tikzset{baseline=0mm}
\draw [color=red!50, fill=red!50,opacity=0.5] (6,2)--(7,2)--(8,-2)--(7,-2)--cycle;
\draw [color=green!50, fill=green!50,opacity=0.5] (12,2)--(13,2)--(10,-2)--(9,-2)--cycle;
\draw [color=blue!50, fill=blue!50,opacity=0.5] (8,2)--(9,2)--(12,-2)--(11,-2)--cycle;
\draw(1,2) node[above]{$0$}--(1,-2);
\draw(2,2) node[above]{$1$}--(2,-2);
\draw(3,2) node[above]{$0$}--(3,-2);
\draw(4,2) node[above]{$1$}--(4,-2);
\draw(5,2) node[above]{$0$}--(5,-2);
\draw(6,2) node[above]{$1$}--(7,-2);
\draw(7,2) node[above]{$0$}--(8,-2);
\draw(8,2) node[above]{$1$}--(11,-2);
\draw(9,2) node[above]{$0$}--(12,-2);
\draw(10,2) node[above]{$1$}--(13,-2);
\draw(11,2) node[above]{$0$}--(6,-2);
\draw(12,2) node[above]{$1$}--(9,-2);
\draw(13,2) node[above]{$0$}--(10,-2);
\draw(14,2) node[above]{$1$}--(14,-2);
\draw(15,2) node[above]{$1$}--(15,-2);
\draw(16,2) node[above]{$1$}--(16,-2);
\draw(17,2) node[above]{$0$}--(17,-2);
\end{braid},
\\
\psi^{\gggg^A}1_{\bi^{\blam/\bmu}}&=
\begin{braid}\tikzset{baseline=0mm}
\draw [color=green!50, fill=green!50,opacity=0.5] (12,2)--(13,2)--(8,-2)--(7,-2)--cycle;
\draw [color=red!50, fill=red!50,opacity=0.5] (6,2)--(7,2)--(10,-2)--(9,-2)--cycle;
\draw [color=blue!50, fill=blue!50,opacity=0.5] (8,2)--(9,2)--(12,-2)--(11,-2)--cycle;
\draw(1,2) node[above]{$0$}--(1,-2);
\draw(2,2) node[above]{$1$}--(2,-2);
\draw(3,2) node[above]{$0$}--(3,-2);
\draw(4,2) node[above]{$1$}--(4,-2);
\draw(5,2) node[above]{$0$}--(5,-2);
\draw(6,2) node[above]{$1$}--(9,-2);
\draw(7,2) node[above]{$0$}--(10,-2);
\draw(8,2) node[above]{$1$}--(11,-2);
\draw(9,2) node[above]{$0$}--(12,-2);
\draw(10,2) node[above]{$1$}--(13,-2);
\draw(11,2) node[above]{$0$}--(6,-2);
\draw(12,2) node[above]{$1$}--(7,-2);
\draw(13,2) node[above]{$0$}--(8,-2);
\draw(14,2) node[above]{$1$}--(14,-2);
\draw(15,2) node[above]{$1$}--(15,-2);
\draw(16,2) node[above]{$1$}--(16,-2);
\draw(17,2) node[above]{$0$}--(17,-2);
\end{braid}.
\end{align*}
The colors in the diagram are merely intended to highlight the action on bricks in \({\bf B}^A\).
\begin{lem}\label{GarA}
Suppose that \(\blam/\bmu \in \mathscr{S}^\kappa_{\bmu}\) and \(A \in \blam/\bmu\) is a Garnir node. Then
\begin{align*}
\Gar^A \backslash\{\gggg^A\} = \{ \ttt \in \St(\blam/\bmu) \mid \ttt \trianglelefteq \gggg^A \textup{ and } \bi(\ttt)=\bi^A\}.
\end{align*}
Moreover, \(\deg \ttt=\deg s_r\gggg^A-a_{i_r,i_{r+1}}\) for all \(\ttt \in \Gar^A\backslash\{\gggg^A\}\), where \(r=\gggg^A(A)-1\).
\end{lem}
\begin{proof}
The first statement is clear from the preceding discussion and Lemma \ref{sameoutside}. 

For the second statement, we instead prove that \(\codeg(s_r\gggg^A)-\codeg(\ttt)=-a_{i_r,i_{r+1}} \), which is equivalent by \cite[Lemma 3.12]{bkw}. By the definition of codegree, and the fact that \(\YYY(\ttt)\) and \(\YYY(s_r\gggg^A)\) agree outside of the bricks of the Garnir belt, it is enough to consider the case where \(\blam/\bmu\) is a two-row Young diagram, with \(\blam = \blam/\bmu = ((k^Ae-1,(k^A-f^A)e))\), and \(A=(1,(k^A-f^A)e,1)\). Pictorially, 
\begin{align*}
\blam=
\begin{tikzpicture}[scale=0.42]
\tikzset{baseline=-5mm}
\draw [ fill=gray!30] (5,0)--(5,-1)--(-3,-1)--(-3,-2)--(6,-2)--(6,-1)--(14,-1)--(14,0)--cycle;
\draw(-3,0)--(14,0)--(14,-1)--(6,-1)--(6,-2)--(-3,-2)--cycle;
\draw(5.5,-0.5) node{$\scriptstyle A$};
\draw [semithick,double] (0,-1)--(0,-2);
\draw [semithick,double] (3,-1)--(3,-2);
\draw [semithick,double] (5,-1)--(6,-1);
\draw [semithick,double] (8,-0)--(8,-1);
\draw [semithick,double] (11,-0)--(11,-1);
\draw [ultra thick] (5,0)--(5,-1)--(-3,-1)--(-3,-2)--(6,-2)--(6,-1)--(14,-1)--(14,0)--cycle;
\draw [color=white,fill=white] (1,0.2)--(2,0.2)--(2,-2.2)--(1,-2.2)--cycle;
\draw [color=white,fill=white] (9,0.2)--(10,0.2)--(10,-2.2)--(9,-2.2)--cycle;
\draw(1.5,-0.5) node{$\scriptstyle \cdots$};
\draw(1.5,-1.5) node{$\scriptstyle \cdots$};
\draw(9.5,-0.5) node{$\scriptstyle \cdots$};
\end{tikzpicture}
\end{align*}
Each brick contributes 0 to \(\codeg(\ttt)\), and every brick contributes 0 to \(\codeg(s_r\gggg^A)\), except for \(B_{k^A-f^A}^A\) (the rightmost brick in the bottom row), which contributes \(2\) if \(e=2\) and \(1\) if \(e>2\). Thus \(\codeg(\ttt)=0\) and \(\codeg(s_r\gggg^A)=-a_{i_r,i_{r+1}}\), and the result follows. 
\end{proof}
\subsection{The skew Specht module \texorpdfstring{$S^{\blam/\bmu}$}{}}
Fix a Garnir node \(A \in \blam/\bmu\). Define
\begin{align*}
\sigma_r^A := \psi_{w_r^A}1_{\bi^A} \hspace{5mm} \textup{and} \hspace{5mm} \tau_r^A:=(\sigma_r^A +1)1_{\bi^A}.
\end{align*}
Write \(u \in \mathfrak{S}^A\) as a reduced product \(w_{r_1}^A \cdots w_{r_a}^A\) of simple generators in \(\mathfrak{S}^A\). If \(u \in \mathscr{D}^A\), then \(u\) is fully commutative, and thus we have well-defined elements
\begin{align*}
\{\tau_u^A := \tau_{r_1}^A \cdots \tau_{r_a}^A \mid u \in \mathscr{D}^A\}.
\end{align*}
For any \(\sss \in \Gar^A\), we may can write \(w^\sss=u^\sss w^{\ttt^A}\) so that \(\ell(w^\sss)=\ell(u^\sss)+\ell (w^{\ttt^A})\) and \(u^\sss \in \mathscr{D}^A\), and the elements \(\psi_{u^\sss}\), \(\psi^{\ttt^A}\) and \(\psi^\sss=\psi_{u^\sss}\psi^{\ttt^A}\) are all independent of the choice of reduced decomposition. 
\begin{definition}
Let \(\blam/\bmu \in \mathscr{S}^\kappa_{\bmu,\alpha}\), and \(A \in \blam/\bmu\) be a Garnir node. The {\it Garnir element} is 
\begin{align}\label{gA}
g^A := \sum_{u \in \mathscr{D}^A} \tau_u^A \psi^{\ttt^A} \in R_\alpha.
\end{align}
\end{definition}
By Lemma \ref{GarA}, all summands on the right side of (\ref{gA}) are of the same degree.
\begin{definition}\label{Mdef}
Let \(\alpha \in Q_+\), \(d=\textup{ht}(\alpha)\), and \(\blam/\bmu \in \mathscr{S}^\kappa_{\bmu,\alpha}\). Define the {\it graded skew row permutation module} \(M^{\blam/\bmu} = M^{\blam/\bmu}(\mathcal{O})\)
to be the graded \(R_\alpha\)-module generated by the vector \(m^{\blam/\bmu}\) in degree \(\deg \ttt^{\blam/\bmu}\) and subject only to the following relations:
\begin{enumerate}
\item[(i)] \(1_{\bj} m^{\blam/\bmu}=\delta_{\bj,\bi^{\blam/\bmu}}m^{\blam/\bmu}\) for all \(\bj \in \langle I \rangle_\alpha\);
\item[(ii)] \(y_rm^{\blam/\bmu} = 0\) for all \(r=1, \ldots, d\);
\item[(iii)] \(\psi_r m^{\blam/\bmu} = 0\) for all \(r=1, \ldots, d-1\) such that \(r \rightarrow_{\ttt^{\blam/\bmu}} r+1\).
\end{enumerate}
\end{definition}
\begin{definition}\label{Sdef}
Let \(\alpha \in Q_+\), \(d=\textup{ht}(\alpha)\), and \(\blam/\bmu \in \mathscr{S}^\kappa_{\bmu,\alpha}\). We define the {\it graded skew Specht module} \(S^{\blam/\bmu} = S^{\blam/\bmu}(\mathcal{O})\) to be the graded \(R_\alpha\)-module generated by the vector \(z^{\blam/\bmu}\) in degree \(\deg \ttt^{\blam/\bmu}\) and subject only to the following relations:
\begin{enumerate}
\item[(i)] \(1_{\bj} z^{\blam/\bmu}=\delta_{\bj,\bi^{\blam/\bmu}}z^{\blam/\bmu}\) for all \(\bj \in \langle I \rangle_\alpha\);
\item[(ii)] \(y_rz^{\blam/\bmu} = 0\) for all \(r=1, \ldots, d\);
\item[(iii)] \(\psi_r z^{\blam/\bmu} = 0\) for all \(r=1, \ldots, d-1\) such that \(r \rightarrow_{\ttt^{\blam/\bmu}} r+1\);
\item[(iv)] \(g^Az^{\blam/\bmu}=0\) for all Garnir nodes \(A \in \blam/\bmu\).
\end{enumerate}
\end{definition}
In other words, \(S^{\blam/\bmu} = (R_\alpha/J_\alpha^{\blam/\bmu})\langle \deg(\ttt^{\blam/\bmu})\rangle\), where \(J_\alpha^{\blam/\bmu}\) is the homogeneous left ideal of \(R_\alpha\) generated by the elements
\begin{enumerate}\label{Jalpha}
\item[(i*)] \(1_{\bj} -\delta_{\bj,\bi^{\blam/\bmu}}\) for all \(\bj \in \langle I \rangle_\alpha\);
\item[(ii*)] \(y_r\) for all \(r=1, \ldots, d\);
\item[(iii*)] \(\psi_r \) for all \(r=1, \ldots, d-1\) such that \(r \rightarrow_{\ttt^{\blam/\bmu}} r+1\);
\item[(iv*)] \(g^A\) for all Garnir nodes \(A \in \blam/\bmu\).
\end{enumerate}
The elements (i*)-(iii*) generate a left ideal \(K^{\blam/\bmu}\) such that \(R_\alpha/K^{\blam/\bmu} \cong M^{\blam/\bmu}\). So we have a natural surjection \(M^{\blam/\bmu} \twoheadrightarrow S^{\blam/\bmu}\) with kernel \(J^{\blam/\bmu}\) generated by the Garnir relations \(g^Am^{\blam/\bmu} = 0\). This surjection maps \(m^{\blam/\bmu}\) to \(z^{\blam/\bmu}\) and \(J^{\blam/\bmu} = J_\alpha^{\blam/\bmu} m^{\blam/\bmu}\).

For \(\ttt \in \Tab(\blam/\bmu)\), we write 
\begin{align*}
m^\ttt := \psi^\ttt m^{\blam/\bmu} \in M^{\blam/\bmu} \hspace{15mm} \textup{and} \hspace{15mm} v^\ttt := \psi^\ttt z^{\blam/\bmu} \in S^{\blam/\bmu}.
\end{align*}

\subsection{A basis for \texorpdfstring{$M^{\blam/\bmu}$}{} and a spanning set for \texorpdfstring{$S^{\blam/\bmu}$}{}}

\begin{thm}
The elements of the set
\begin{align*}
\{m^\ttt \mid \ttt \in \Tab(\blam/\bmu) \textup{ is row-strict}\}
\end{align*}
form an \(\mathcal{O}\)-basis for \(M^{\blam/\bmu}\).
\end{thm}
\begin{proof}
This is \cite[Theorem 5.6]{kmr} in the Young diagram case. But since \(M^{\blam/\bmu}\) is a permutation module in the sense of \cite[\S 3.6]{kmr}, the proof in the skew case also follows immediately from \cite[Theorem 3.23]{kmr}.
\end{proof}

\begin{prop}
The elements of the set
\begin{align}\label{spanset}
\{v^\ttt \mid \ttt \in \St(\blam/\bmu)\}
\end{align}
span \(S^{\blam/\bmu}\) over \(\mathcal{O}\). Moreover, we have \(\deg(v^\ttt) = \deg(\ttt)\) for all \(\ttt \in \St(\blam/\bmu)\).
\end{prop}
\begin{proof}
Using Lemma \ref{degmatch}, 
\begin{align*}
\deg(v^\ttt) = \deg(\psi^\ttt1_{\bi^{\blam/\bmu}}z^{\blam/\bmu}) = \deg(\psi^\ttt1_{\bi^{\blam/\bmu}})+\deg(z^{\blam/\bmu})= \deg(\ttt)-\deg(\ttt^{\blam/\bmu}) + \deg(z^{\blam/\bmu}),
\end{align*}
which proves the second statement, as \(\deg(z^{\blam/\bmu})=\deg(\ttt^{\blam/\bmu})\) by definition. 

The proof of the first statement follows exactly as it does in the Young diagram case provided in \cite[Proposition 5.14]{kmr}---there are clear skew analogues of the results in \cite[\S5.5--5.6]{kmr}---the only caveat is that our preferred partial order on standard tableaux is opposite that of \cite{kmr}, so one must swap the direction of `\(\triangleleft\)' signs when necessary, and make use of the analogous skew dominance results in Lemmas \ref{skeworder}--\ref{srt}.
\end{proof}


\section{Connecting skew Specht modules with restrictions of Specht modules}
\label{linind}
\noindent In this section we show that for \(\blam \in \mathscr{P}^\kappa_{\alpha+\beta}\), the \(R_{\alpha,\beta}\)-module \(\Res_{\alpha,\beta}S^{\blam}\) has a filtration with subquotients isomorphic to \(S^{\bmu} \boxtimes S^{\blam/\bmu}\), with \(\bmu \in \mathscr{P}^\kappa_{\alpha}\) and \(\blam/\bmu \in \mathscr{S}^\kappa_{\bmu,\beta}\). As a consequence, we get that (\ref{spanset}) is an \(\mathcal{O}\)-basis for \(S^{\blam/\bmu}\). For the case of Young diagrams, this was shown in \cite[Corollary 6.24]{kmr}:
\begin{thm}\label{youngbasis}
Let \(\blam \in \mathscr{P}_\alpha^\kappa\). Then \(S^{\blam}\) has \(\mathcal{O}\)-basis
\(\{v^\TTT \mid \TTT \in \St(\blam)\}.
\)
\end{thm}
\subsection{Submodules of \texorpdfstring{$\Res_{\alpha,\beta} S^{\blam}$}{}}
Let \(\alpha, \beta \in Q_+\) and \(\textup{ht}(\alpha)=a, \textup{ht}(\beta)=b\). Let \(\blam \in \mathscr{P}^\kappa_{\alpha+\beta}\), \(\bmu \in \mathscr{P}^\kappa_\alpha\). By Theorem \ref{youngbasis}, \(S^{\blam}_{\alpha,\beta} := \Res_{\alpha,\beta}(S^{\blam})\) has \(\mathcal{O}\)-basis \(\{v^\TTT \mid \TTT \in B\},\)
where
\begin{align*}
B=\{\TTT \in \St(\lambda)\mid  \; \textup{cont}(\textup{sh}(\TTT_{\leq a}))=\alpha\},
\end{align*}
since \(1_{\alpha,\beta}v^\TTT = v^\TTT\) if and only if \(\bi(\TTT) = i_1 \cdots i_{a+b}\) has \(\alpha_{i_1} + \cdots +\alpha_{i_a} = \alpha\), and is zero otherwise. Define
\begin{align*}
B_{\bmu} = \{\TTT \in B\mid  \textup{sh}(\TTT_{\leq a}) \trianglerighteq \bmu\}
\hspace{8mm}
\textup{and}
\hspace{8mm}
C_{\bmu} = \mathcal{O}\{v^\TTT \in S^{\blam} \mid \TTT \in B_{\bmu}\}.
\end{align*}
\begin{lem}\label{lowerin}
If \(\UUU, \TTT \in B\), \(\UUU \trianglelefteq \TTT\), and \(\TTT \in B_{\bmu}\), then \(\UUU \in B_{\bmu}\).
\end{lem}
\begin{proof}
If \(\UUU \trianglelefteq \TTT\), then by Lemma \ref{ordershape}, \(\textup{sh}(\UUU_{\leq a}) \trianglerighteq \textup{sh}(\TTT_{\leq a}) \trianglerighteq \bmu\).\end{proof}

\begin{lem}\label{submodules}
\(C_{\bmu}\) is an \(R_{\alpha,\beta}\)-submodule of \(S^{\blam}_{\alpha,\beta}\). 
\end{lem}
\begin{proof}
We show that \(C_{\bmu}\) is invariant under the action of generators of \(R_{\alpha,\beta}\). For idempotents \(1_{\bi\bj}\) this is clear. Let \(\TTT \in B_{\bmu}\). 
\begin{enumerate}
\item For \(1 \leq j \leq a+b\), \(y_j v^\TTT\) is an \(\mathcal{O}\)-linear combination of \(v^\UUU \in B\) for \(\UUU \triangleleft \TTT\), by \cite[Lemma 4.8]{bkw}. By Lemma \ref{lowerin}, each \(v^\UUU\) is in \(C_{\bmu}\).
\item For \(j \in \{1, \ldots, a-1,a+1,\ldots,a+b-1\}\), where \(j \rightarrow_\TTT j+1\) or \(j \downarrow_\TTT j+1\), then \(\psi_j v^\TTT\) is a linear combination of \(v^\UUU \in B\) for \(\UUU \triangleleft \TTT\), by \cite[Lemma 4.9]{bkw}, and the result follows by Lemma \ref{lowerin}.
\item For \(j \in \{1, \ldots, a-1,a+1,\ldots,a+b-1\}\), where \(j \rotatebox[origin=c]{45}{$\Rightarrow$}_{\hspace{-1mm}\TTT}\, j+1\), then \(\psi_j v^\TTT\) is a linear combination of \(v^\UUU \in B\) for \(\UUU \triangleleft \TTT\), by \cite[Lemma 2.14]{fs}, and the result follows by Lemma \ref{lowerin}.
\item Assume \(j \in \{1, \ldots, a-1,a+1,\ldots,a+b-1\}\), and \(j+1 \rotatebox[origin=c]{45}{$\Rightarrow$}_{\hspace{-1mm}\TTT}\, j\). Then \(s_j\TTT \triangleright \TTT\), and \(s_jw^\TTT = w^{s_j\TTT}\), with \(\ell(w^{s_j\TTT})=\ell(w^\TTT)+1\). Then by Lemma \ref{rewrite}, \(\psi_j v^\TTT = v^{s_j\TTT} + \sum_{\UUU \triangleleft s_j\TTT} c_\UUU v^\UUU\) for some constants \(c_\UUU \in \mathcal{O}\). But \((s_j\TTT)_{\leq a} = \TTT_{\leq a}\), so \(s_j\TTT \in B_{\bmu}\) and the result follows by Lemma \ref{lowerin}.
\end{enumerate}
This exhausts the possibilities for \(\TTT\) and completes the proof. 
\end{proof}
Now define 
\(
B_{\triangleright \bmu} = \bigcup_{\bnu \triangleright \bmu} B_{\bnu} = \{\TTT \in B \mid \textup{sh}(\TTT_{\leq a}) \triangleright \bmu\}.
\)
Then 
\(
C_{\triangleright \bmu} := \sum_{\bnu \triangleright \bmu} C_{\bnu} = \mathcal{O} \{v^\TTT \in S^{\blam} \mid \textup{sh}(\TTT_{\leq a}) \triangleright \bmu\}
\)
is an \(R_{\alpha,\beta}\)-submodule of \(S^{\blam}_{\alpha,\beta}\). Define \(N_{\bmu} =C_{\bmu}/C_{\triangleright \bmu}\), and write 
\begin{align*}
x^\TTT = v^\TTT + C_{\triangleright\bmu} \in S^{\blam}_{\alpha,\beta}/C_{\triangleright \bmu}
\end{align*}
for \(\TTT \in B\). To cut down on notational clutter in what follows, write \(\bxi\) for \(\blam / \bmu\), \(\xi^{(i)}\) for the components \(\lambda^{(i)}/\mu^{(i)}\) of \(\blam/\bmu\), and \(\xi^{(i)}_j\) for the \(j\)th row of nodes in \(\xi^{(i)}\). Then for \(\TTT \in \Tab(\bmu)\), \(\ttt \in \Tab(\bxi)\), define \(\TTT\ttt \in \Tab(\blam)\) such that \((\TTT\ttt)_{\leq a} = \TTT\) and \(\TTT\ttt(A) = \YYY(\ttt)(A)\) for nodes \(A \in \bxi\). From the definition it is clear that \(N_{\bmu}\) has homogeneous \(\mathcal{O}\)-basis
\begin{align*}
\{x^{\TTT} \mid \TTT \in \St(\blam), \textup{sh}(\textup{cont}(\TTT_{\leq a}))= \bmu \} =
\{x^{\TTT\ttt} \mid \TTT \in \St(\bmu), \ttt \in \St(\bxi)\}.
\end{align*}
Write \(\TTT^{\bmu\bxi}:=\TTT^{\bmu}\ttt^{\bxi} = \YYY(\ttt^{\bxi})\), and write \(x^{\bmu\bxi}\) for \(x^{\TTT^{\bmu\bxi}}\). 

\subsection{Constructing a morphism \texorpdfstring{$S^{\bmu} \boxtimes S^{\blam/\bmu} \to N_{\bmu}$}{}}\label{fmorphism}
Define a (graded) morphism \(f\) from the free module \(R_{\alpha,\beta}\langle \deg\TTT^{\bmu} + \deg\ttt^{\bxi}\rangle\) to \(N_{\bmu}\) by \(f:1_{\alpha,\beta} \mapsto x^{\bmu\bxi}\). 
\begin{prop}\label{kerf}
The kernel of \( f\) contains the left ideal \(K_\alpha^{\bmu} \otimes R_\beta + R_{\alpha} \otimes K_\beta^{\bxi} \).
\end{prop}
\begin{proof}
We show that the relevant generators of \(K_\alpha^{\bmu} \otimes R_\beta\), given by (i*)--(iii*) in Definition \ref{Sdef} are sent to zero by \(f\). The proof for \(R_{\alpha} \otimes K_\beta^{\bxi}\) is similar.
\begin{enumerate}
\item[(i*)] First we consider idempotents.
\begin{align*}
f[(1_{\bj} - \delta_{\bj,\bi^{\bmu}}) \otimes 1_\beta ] &= (1_{\bj,\beta} - \delta_{\bj,\bi^{\bmu}})x^{\bmu \bxi} = \sum_{\bk \in I^{\beta}}1_{\bj\bk}x^{\bmu\bxi} - \delta_{\bj,\bi^{\bmu}}x^{\bxi}\\
&=\sum_{\bk \in I^\beta} \delta_{\bj\bk,\bi^{\bmu}\bi^{\bxi}} x^{\bmu\bxi} - \delta_{\bj,\bi^{\bmu}}x^{\bmu\bxi} = \delta_{\bj\bi^{\bxi},\bi^{\bmu}\bi^{\bxi}}x^{\bmu\bxi}- \delta_{\bj,\bi^{\bmu}}x^{\bmu\bxi}=0.
\end{align*}
\item[(ii*)] For \(1 \leq r \leq a\), we have by \cite[Lemma 4.8]{bkw} that \(f(y_r)=y_r \cdot x^{\bmu\bxi}\) is an \(\mathcal{O}\)-linear combination of \(x^\UUU\), where \(\UUU \in B\) and \(\UUU \triangleleft \TTT^{\bmu\bxi}\). But \(\TTT^{\bmu\bxi}\) is minimal such that \(\textup{sh}(\TTT_{\leq a}) = \bmu\), so each \(\UUU \in B_{\triangleright\bmu}\), and thus \(f(y_r)=0\).
\item[(iii*)] Note that \(r \rightarrow_{\TTT^{\bmu}} r+1\) implies \(r \rightarrow_{\TTT^{\bmu\bxi}} r+1\), so by \cite[Lemma 4.9]{bkw} it follows that for \(1 \leq r \leq a-1\), \(f(\psi_r) = \psi_r  x^{\bmu\bxi} \) is an \(\mathcal{O}\)-linear combination of \(x^\UUU\), where \(\UUU \in B\) and \(\UUU \triangleleft \TTT^{\bmu\bxi}\). But then as in (2) this implies that \(f(\psi_r)=0\).
\end{enumerate}
\end{proof}

The goal in the rest of this section is to show that in fact, the kernel of \(f\) contains the the left ideal \(J^{\bmu}_\alpha \otimes R_\beta + R_\alpha \otimes J^{\bxi}_\beta\), i.e., \(g^{A_{\bmu}}\otimes 1_\beta\) (resp.  \(1_\alpha \otimes g^{A_{\bxi}}\)) are sent to zero by \(f\), for Garnir nodes \(A_{\bmu} \in \bmu\) (resp. \(A_{\bxi}\in \bxi\)). As the proofs for \(\bmu\) and \(\bxi\) are similar (see Remark \ref{fxi}), we focus on the former and leave the latter for the reader to verify. We will occasionally need to make use of the following lemma, proved in \cite[Lemma 2.16]{fs}:
\begin{lem}\label{subword}
Suppose \(\blam \in \mathscr{P}^\kappa_\alpha\), \(\TTT \in \St(\blam)\), \(j_1,\ldots,j_r \in \{1,\ldots,d-1\}\), and that when \(\psi_{j_1} \cdots \psi_{j_r} z^{\blam}\) is expressed as a linear combination of standard basis elements, \(v^\TTT\) appears with non-zero coefficient. Then the expression \(s_{j_1} \cdots s_{j_r}\) has a reduced expression for \(w^\TTT\) as a subexpression.
\end{lem}

Note that \(w^{\bmu\bxi}:=w^{\TTT^{\bmu\bxi}}\) is in \(\mathscr{D}^{\mu^{(1)}_1,\xi^{(1)}_1,\ldots,\mu^{(l)}_{n(\blam,l)},\xi^{(l)}_{n(\blam,l)}}_{a,b}\), the set of minimal length double coset representatives for 
\begin{align*}
\mathfrak{S}_a \times \mathfrak{S}_b \backslash \mathfrak{S}_{a+b} / \mathfrak{S}_{\mu^{(1)}_1} \times \mathfrak{S}_{\xi^{(1)}_1} \times \cdots \times \mathfrak{S}_{\mu_{n(\blam,l)}^{(l)}} \times \mathfrak{S}_{\xi^{(l)}_{n(\blam,l)}},
\end{align*}
and as such is fully commutative. Writing \(n:=n(\blam,l)\), in diagrammatic form we have
\begin{align*}
w^{\bmu\bxi}=
\begin{braid}\tikzset{baseline=0mm}
\draw [color=blue!50, fill=blue!50, opacity=0.5] (1,2)--(2,2)--(2,-2)--(1,-2)--cycle;
\draw [color=blue!50, fill=blue!50, opacity=0.5] (5,2)--(6,2)--(4,-2)--(3,-2)--cycle;
\draw [color=blue!50, fill=blue!50, opacity=0.5] (11,2)--(12,2)--(7,-2)--(6,-2)--cycle;
\draw [color=blue!50, fill=blue!50, opacity=0.5] (15,2)--(16,2)--(9,-2)--(8,-2)--cycle;
\draw [color=red!50, fill=red!50, opacity=0.5] (3,2)--(4,2)--(11,-2)--(10,-2)--cycle;
\draw [color=red!50, fill=red!50, opacity=0.5] (7,2)--(8,2)--(13,-2)--(12,-2)--cycle;
\draw [color=red!50, fill=red!50, opacity=0.5] (13,2)--(14,2)--(16,-2)--(15,-2)--cycle;
\draw [color=red!50, fill=red!50, opacity=0.5] (17,2)--(18,2)--(18,-2)--(17,-2)--cycle;
\draw(1,2)--(1,-2);
\draw(1.5,2) node[above]{$\mu^{(1)}_1$};
\draw(2,2)--(2,-2);
\draw(3,2)--(10,-2);
\draw(3.5,2) node[above]{$\xi^{(1)}_1$};
\draw(4,2)--(11,-2);
\draw(5,2)--(3,-2);
\draw(5.5,2) node[above]{$\mu^{(1)}_2$};
\draw(6,2)--(4,-2);
\draw(7,2)--(12,-2);
\draw(7.5,2) node[above]{$\xi^{(1)}_2$};
\draw(8,2)--(13,-2);
\draw(9.5,2) node[above]{$\cdots$};
\draw(11,2)--(6,-2);
\draw(11.5,2) node[above]{$\mu^{(l)}_{n-1}$};
\draw(12,2)--(7,-2);
\draw(13,2)--(15,-2);
\draw(13.5,2) node[above]{$\xi^{(l)}_{n-1}$};
\draw(14,2)--(16,-2);
\draw(15,2)--(8,-2);
\draw(15.5,2) node[above]{$\mu^{(l)}_n$};
\draw(16,2)--(9,-2);
\draw(17,2)--(17,-2);
\draw(17.5,2) node[above]{$\xi^{(l)}_n$};
\draw(18,2)--(18,-2);
\draw(5,-2) node{$\cdots$};
\draw(14,-2) node{$\cdots$};
\end{braid}.
\end{align*}
Here we are letting \(\mu^{(j)}_i\) in the diagram stand for \((a_1, \ldots, a_k)\), where \(a_1, \ldots, a_k\) are the entries (in order) in \(\TTT^{\blam}\) of the nodes contained in the \(i\)th row of \(\mu^{(j)}\), and similarly for \(\xi^{(j)}_i\).

Let \(1 \leq i \leq l\), \(1 \leq j \leq n(\blam,i)\). It will be useful to write \(w^{\bmu\bxi} = w^D_{i,j}w^R_{i,j}w^L_{i,j}\), the decomposition into fully commutative elements of \(\mathfrak{S}_{a+b}\) given as follows:
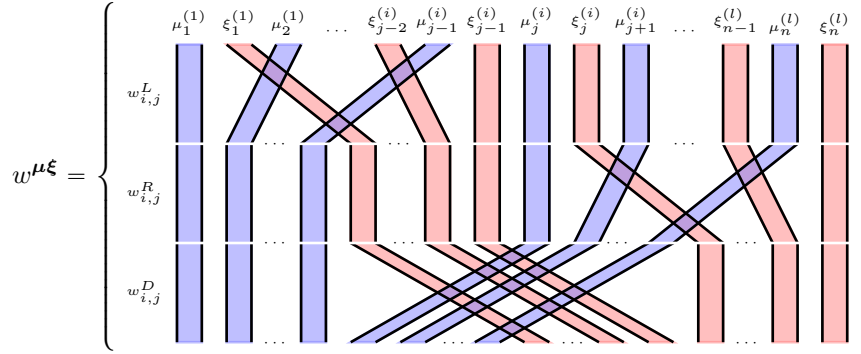
\begin{align*}
w^{\bmu\bxi}=
\begin{cases}
\begin{braid}\tikzset{baseline=-15mm}
\draw [color=blue!50, fill=blue!50, opacity=0.5] (1,2)--(2,2)--(2,-10)--(1,-10)--cycle;
\draw [color=red!50, fill=red!50, opacity=0.5] (3,2)--(4,2)--(9,-2)--(9,-6)--(16,-10)--(15,-10)--(8,-6)--(8,-2)--cycle;
\draw [color=blue!50, fill=blue!50, opacity=0.5] (5,2)--(6,2)--(4,-2)--(4,-10)--(3,-10)--(3,-2)--cycle;
\draw [color=red!50, fill=red!50, opacity=0.5] (9,2)--(10,2)--(12,-2)--(12,-6)--(19,-10)--(18,-10)--(11,-6)--(11,-2)--cycle;
\draw [color=blue!50, fill=blue!50, opacity=0.5] (11,2)--(12,2)--(7,-2)--(7,-6)--(7,-10)--(6,-10)--(6,-6)--(6,-2)--cycle;
\draw [color=red!50, fill=red!50, opacity=0.5] (13,2)--(14,2)--(14,-6)--(21,-10)--(20,-10)--(13,-6)--(13,-2)--cycle;
\draw [color=blue!50, fill=blue!50, opacity=0.5] (15,2)--(16,2)--(16,-6)--(9,-10)--(8,-10)--(15,-6)--(15,-2)--cycle;
\draw [color=red!50, fill=red!50, opacity=0.5] (17,2)--(18,2)--(18,-2)--(23,-6)--(23,-10)--(22,-10)--(22,-6)--(17,-2)--(17,-2)--cycle;
\draw [color=blue!50, fill=blue!50, opacity=0.5] (19,2)--(20,2)--(20,-2)--(18,-6)--(11,-10)--(10,-10)--(17,-6)--(19,-2)--cycle;
\draw [color=red!50, fill=red!50, opacity=0.5] (23,2)--(24,2)--(24,-2)--(26,-6)--(26,-10)--(25,-10)--(25,-6)--(23,-2)--cycle;
\draw [color=blue!50, fill=blue!50, opacity=0.5] (25,2)--(26,2)--(26,-2)--(21,-6)--(14,-10)--(13,-10)--(20,-6)--(25,-2)--cycle;
\draw [color=red!50, fill=red!50, opacity=0.5] (27,2)--(28,2)--(28,-10)--(27,-10)--cycle;
\draw(1,2)--(1,-2)--(1,-6)--(1,-10);
\draw(1.5,2) node[above]{$\mu^{(1)}_1$};
\draw(2,2)--(2,-2)--(2,-6)--(2,-10);
\draw(3,2)--(8,-2)--(8,-6)--(15,-10);
\draw(3.5,2) node[above]{$\xi^{(1)}_1$};
\draw(4,2)--(9,-2)--(9,-6)--(16,-10);
\draw(5,2)--(3,-2)--(3,-6)--(3,-10);
\draw(5.5,2) node[above]{$\mu^{(1)}_2$};
\draw(6,2)--(4,-2)--(4,-6)--(4,-10);
\draw(7.5,2) node[above]{$\cdots$};
\draw(9,2)--(11,-2)--(11,-6)--(18,-10);
\draw(9.5,2) node[above]{$\xi_{j-2}^{(i)}$};
\draw(10,2)--(12,-2)--(12,-6)--(19,-10);
\draw(11,2)--(6,-2)--(6,-6)--(6,-10);
\draw(11.5,2) node[above]{$\mu_{j-1}^{(i)}$};
\draw(12,2)--(7,-2)--(7,-6)--(7,-10);
\draw(13,2)--(13,-2)--(13,-6)--(20,-10);
\draw(13.5,2) node[above]{$\xi_{j-1}^{(i)}$};
\draw(14,2)--(14,-2)--(14,-6)--(21,-10);
\draw(15,2)--(15,-2)--(15,-6)--(8,-10);
\draw(15.5,2) node[above]{$\mu_{j}^{(i)}$};
\draw(16,2)--(16,-2)--(16,-6)--(9,-10);
\draw(17,2)--(17,-2)--(22,-6)--(22,-10);
\draw(17.5,2) node[above]{$\xi_{j}^{(i)}$};
\draw(18,2)--(18,-2)--(23,-6)--(23,-10);
\draw(19,2)--(19,-2)--(17,-6)--(10,-10);
\draw(19.5,2) node[above]{$\mu_{j+1}^{(i)}$};
\draw(20,2)--(20,-2)--(18,-6)--(11,-10);
\draw(21.5,2) node[above]{$\cdots$};
\draw(23,2)--(23,-2)--(25,-6)--(25,-10);
\draw(23.5,2) node[above]{$\xi^{(l)}_{n-1}$};
\draw(24,2)--(24,-2)--(26,-6)--(26,-10);
\draw(25,2)--(25,-2)--(20,-6)--(13,-10);
\draw(25.5,2) node[above]{$\mu^{(l)}_{n}$};
\draw(26,2)--(26,-2)--(21,-6)--(14,-10);
\draw(27,2)--(27,-2)--(27,-6)--(27,-10);
\draw(27.5,2) node[above]{$\xi^{(l)}_{n}$};
\draw(28,2)--(28,-2)--(28,-6)--(28,-10);
\draw[color=white, opacity=1] (0,-2)--(28.5,-2);
\draw[color=white, opacity=1] (0,-6)--(28.5,-6);
\draw(5,-2) node{$\cdots$};
\draw(10,-2) node{$\cdots$};
\draw(5,-6) node{$\cdots$};
\draw(5,-10) node{$\cdots$};
\draw(10.2,-6) node{$\cdots$};
\draw(21.5,-2) node{$\cdots$};
\draw(18.8,-6) node{$\cdots$};
\draw(24,-6) node{$\cdots$};
\draw(12.2,-10) node{$\cdots$};
\draw(17,-10) node{$\cdots$};
\draw(24,-10) node{$\cdots$};
\draw(-0.3,0) node{$w^L_{i,j}$};
\draw(-0.3,-4) node{$w^R_{i,j}$};
\draw(-0.3,-8) node{$w^D_{i,j}$};
\end{braid}.
\end{cases}
\end{align*}
Define \(\psi^X_{i,j}:= \psi_{w^X_{i,j}}\) for \(X \in \{L,R,D\}\), and set
\begin{align*}
c_{i,j}&=\sum_{\substack{1\leq h \leq i-1\\1 \leq k \leq n(\blam,h)}} \mu^{(h)}_k + \sum_{1 \leq k \leq j-1} \mu^{(i)}_k,\\
d_{i,j}&= \sum_{\substack{1 \leq h \leq i-1\\1 \leq k \leq n(\blam,h)}} \xi^{(h)}_k + \sum_{1 \leq k \leq j-1} \xi^{(i)}_k.
\end{align*}
If \(\Psi:= \psi_{r_1} \cdots \psi_{r_s}\) for some \(r_1, \ldots, r_s\), then we will write \(\Psi[c]:= \psi_{r_1 + c} \cdots \psi_{r_s+c}\) for admissible \(c \in \ZZ\). 
The following lemma will aid us in translating between Garnir relations defining \(S^{\blam}\) and those defining \(S^{\bmu}\).
\begin{lem}\label{movepast}
Assume \(r_1, \ldots, r_s\) are such that \(c_{i,j}+1 \leq r_1, \ldots, r_s \leq a-1\), and \(\Psi=\psi_{r_1} \cdots \psi_{r_s}\). Then 
\begin{align*}
\Psi  x^{\bmu\bxi} = \psi^D_{i,j} \psi^L_{i,j} \Psi[d_{i,j}] \psi^R_{i,j} x^{\blam}.
\end{align*}
\end{lem}
\begin{proof}
We go by induction on \(s\), the base case \(s=0\) being trivial. By assumption we have
\begin{align*}
\Psi x^{\bmu\bxi}= \psi_{r_1} \cdots \psi_{r_s}  x^{\bmu \bxi} = \psi_{r_1}\psi^D_{i.j} \psi^L_{i.j} \psi_{r_2+d_{i,j}} \cdots \psi_{r_s+d_{i,j}}\psi^R_{i,j} x^{\blam}.
\end{align*}
Write \(\bi^{\mu^{(h)}_k}\) for the residue sequence associated with the nodes in \(\mu^{(h)}_k\) in \(\TTT^{\blam}\), and similarly for \(\bi^{\xi^{(h)}_k}\). In terms of Khovanov-Lauda diagrams, with the vector \(x^{\blam}\) pictured as being at the top of the diagram, we must show that
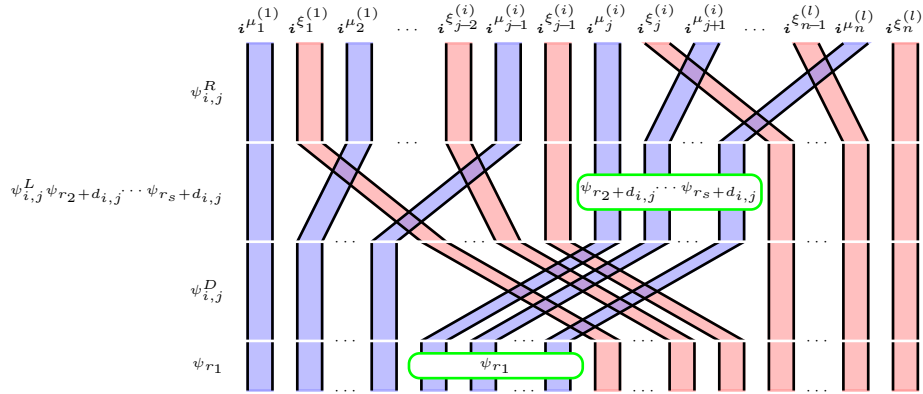
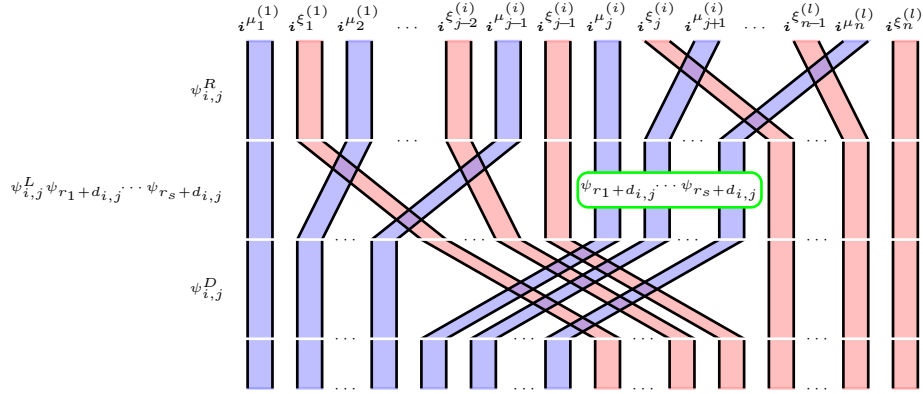
\begin{align*}
\begin{braid}\tikzset{baseline=0mm}
\draw [color=blue!50, fill=blue!50, opacity=0.5] (1,2)--(2,2)--(2,-12)--(1,-12)--cycle;
\draw [color=red!50, fill=red!50, opacity=0.5] (3,2)--(4,2)--(4,-2)--(9,-6)--(16,-10)--(16,-12)--(15,-12)--(15,-10)--(8,-6)--(3,-2)--cycle;
\draw [color=blue!50, fill=blue!50, opacity=0.5] (5,2)--(6,2)--(6,-2)--(4,-6)--(4,-10)--(4,-12)--(3,-12)--(3,-6)--(5,-2)--cycle;
\draw [color=red!50, fill=red!50, opacity=0.5] (9,2)--(10,2)--(10,-2)--(12,-6)--(19,-10)--(19,-12)--(18,-12)--(18,-10)--(11,-6)--(9,-2)--cycle;
\draw [color=blue!50, fill=blue!50, opacity=0.5] (11,2)--(12,2)--(12,-2)--(7,-6)--(7,-10)--(7,-12)--(6,-12)--(6,-10)--(6,-6)--(11,-2)--cycle;
\draw [color=red!50, fill=red!50, opacity=0.5] (13,2)--(14,2)--(14,-6)--(21,-10)--(21,-12)--(20,-12)--(20,-10)--(13,-6)--(13,-2)--cycle;
\draw [color=blue!50, fill=blue!50, opacity=0.5] (15,2)--(16,2)--(16,-6)--(9,-10)--(9,-12)--(8,-12)--(8,-10)--(15,-6)--(15,-2)--cycle;
\draw [color=red!50, fill=red!50, opacity=0.5] (17,2)--(18,2)--(23,-2)--(23,-6)--(23,-10)--(23,-12)--(22,-12)--(22,-10)--(22,-6)--(22,-2)--cycle;
\draw [color=blue!50, fill=blue!50, opacity=0.5] (19,2)--(20,2)--(18,-2)--(18,-6)--(11,-10)--(11,-12)--(10,-12)--(10,-10)--(17,-6)--(17,-2)--cycle;
\draw [color=red!50, fill=red!50, opacity=0.5] (23,2)--(24,2)--(26,-2)--(26,-6)--(26,-10)--(26,-12)--(25,-12)--(25,-10)--(25,-6)--(25,-2)--cycle;
\draw [color=blue!50, fill=blue!50, opacity=0.5] (25,2)--(26,2)--(21,-2)--(21,-6)--(14,-10)--(14,-12)--(13,-12)--(13,-10)--(20,-6)--(20,-2)--cycle;
\draw [color=red!50, fill=red!50, opacity=0.5] (27,2)--(28,2)--(28,-12)--(27,-12)--cycle;
\draw(1,2)--(1,-2)--(1,-6)--(1,-10)--(1,-12);
\draw(1.5,2) node[above]{$\bi^{\mu^{(1)}_1}$};
\draw(2,2)--(2,-2)--(2,-6)--(2,-10)--(2,-12);
\draw(3,2)--(3,-2)--(8,-6)--(15,-10)--(15,-12);
\draw(3.5,2) node[above]{$\bi^{\xi^{(1)}_1}$};
\draw(4,2)--(4,-2)--(9,-6)--(16,-10)--(16,-12);
\draw(5,2)--(5,-2)--(3,-6)--(3,-10)--(3,-12);
\draw(5.5,2) node[above]{$\bi^{\mu^{(1)}_2}$};
\draw(6,2)--(6,-2)--(4,-6)--(4,-10)--(4,-12);
\draw(7.5,2) node[above]{$\cdots$};
\draw(9,2)--(9,-2)--(11,-6)--(18,-10)--(18,-12);
\draw(9.5,2) node[above]{$\bi^{\xi^{(i)}_{j\hspace{-0.5mm}-\hspace{-0.5mm}2}}$};
\draw(10,2)--(10,-2)--(12,-6)--(19,-10)--(19,-12);
\draw(11,2)--(11,-2)--(6,-6)--(6,-10)--(6,-12);
\draw(11.5,2) node[above]{$\bi^{\mu^{(i)}_{j\hspace{-0.5mm}-\hspace{-0.5mm}1}}$};
\draw(12,2)--(12,-2)--(7,-6)--(7,-10)--(7,-12);
\draw(13,2)--(13,-2)--(13,-6)--(20,-10)--(20,-12);
\draw(13.5,2) node[above]{$\bi^{\xi^{(i)}_{j\hspace{-0.5mm}-\hspace{-0.5mm}1}}$};
\draw(14,2)--(14,-2)--(14,-6)--(21,-10)--(21,-12);
\draw(15,2)--(15,-2)--(15,-6)--(8,-10)--(8,-12);
\draw(15.5,2) node[above]{$\bi^{\mu^{(i)}_{j}}$};
\draw(16,2)--(16,-2)--(16,-6)--(9,-10)--(9,-12);
\draw(17,2)--(22,-2)--(22,-6)--(22,-10)--(22,-12);
\draw(17.5,2) node[above]{$\bi^{\xi^{(i)}_{j}}$};
\draw(18,2)--(23,-2)--(23,-6)--(23,-10)--(23,-12);
\draw(19,2)--(17,-2)--(17,-6)--(10,-10)--(10,-12);
\draw(19.5,2) node[above]{$\bi^{\mu^{(i)}_{j\hspace{-0.5mm}+\hspace{-0.5mm}1}}$};
\draw(20,2)--(18,-2)--(18,-6)--(11,-10)--(11,-12);
\draw(21.5,2) node[above]{$\cdots$};
\draw(23,2)--(25,-2)--(25,-6)--(25,-10)--(25,-12);
\draw(23.5,2) node[above]{$\bi^{\xi^{(l)}_{n\hspace{-0.5mm}-\hspace{-0.5mm}1}}$};
\draw(24,2)--(26,-2)--(26,-6)--(26,-10)--(26,-12);
\draw(25,2)--(20,-2)--(20,-6)--(13,-10)--(13,-12);
\draw(25.5,2) node[above]{$\bi^{\mu^{(l)}_{n}}$};
\draw(26,2)--(21,-2)--(21,-6)--(14,-10)--(14,-12);
\draw(27,2)--(27,-2)--(27,-6)--(27,-10)--(27,-12);
\draw(27.5,2) node[above]{$\bi^{\xi^{(l)}_{n}}$};
\draw(28,2)--(28,-2)--(28,-6)--(28,-10)--(28,-12);
\draw[color=white, opacity=1] (0,-2)--(28.5,-2);
\draw[color=white, opacity=1] (0,-6)--(28.5,-6);
\draw[color=white, opacity=1] (0,-10)--(28.5,-10);
\draw(7.5,-2) node{$\cdots$};
\draw(5,-6) node{$\cdots$};
\draw(5,-10) node{$\cdots$};
\draw(10.2,-6) node{$\cdots$};
\draw(19,-2) node{$\cdots$};
\draw(18.8,-6) node{$\cdots$};
\draw(24,-2) node{$\cdots$};
\draw(24,-6) node{$\cdots$};
\draw(12.2,-10) node{$\cdots$};
\draw(5,-12) node{$\cdots$};
\draw(17,-10) node{$\cdots$};
\draw(24,-10) node{$\cdots$};
\draw(12.2,-12) node{$\cdots$};
\draw(17,-12) node{$\cdots$};
\draw(24,-12) node{$\cdots$};
\draw(0.5,0) node[left]{$\psi^R_{i,j}$};
\draw(0.5,-4) node[left]{$\psi^L_{i,j}\psi_{r_2+d_{i,j}}\hspace{-1mm} \cdots \psi_{r_s+d_{i,j}}$};
\draw(0.5,-8) node[left]{$\psi^D_{i,j}$};
\draw(0.5,-11) node[left]{$\psi_{r_1}$};
\draw [rounded corners, color=green, fill=white!50] (7.5,-10.5)--(14.5,-10.5)--(14.5,-11.5)--(7.5,-11.5)--cycle;
\draw [rounded corners, color=green, fill=white!50] (14.3,-3.3)--(21.7,-3.3)--(21.7,-4.7)--(14.3,-4.7)--cycle;
\draw(18,-4) node{$\psi_{r_2+d_{i,j}}\hspace{-1mm} \cdots \psi_{r_s+d_{i,j}}$};
\draw(11,-11) node{$\psi_{r_1} $};
\end{braid}
\end{align*}
is equal to 
\begin{align*}
\begin{braid}\tikzset{baseline=0mm}
\draw [color=blue!50, fill=blue!50, opacity=0.5] (1,2)--(2,2)--(2,-12)--(1,-12)--cycle;
\draw [color=red!50, fill=red!50, opacity=0.5] (3,2)--(4,2)--(4,-2)--(9,-6)--(16,-10)--(16,-12)--(15,-12)--(15,-10)--(8,-6)--(3,-2)--cycle;
\draw [color=blue!50, fill=blue!50, opacity=0.5] (5,2)--(6,2)--(6,-2)--(4,-6)--(4,-10)--(4,-12)--(3,-12)--(3,-6)--(5,-2)--cycle;
\draw [color=red!50, fill=red!50, opacity=0.5] (9,2)--(10,2)--(10,-2)--(12,-6)--(19,-10)--(19,-12)--(18,-12)--(18,-10)--(11,-6)--(9,-2)--cycle;
\draw [color=blue!50, fill=blue!50, opacity=0.5] (11,2)--(12,2)--(12,-2)--(7,-6)--(7,-10)--(7,-12)--(6,-12)--(6,-10)--(6,-6)--(11,-2)--cycle;
\draw [color=red!50, fill=red!50, opacity=0.5] (13,2)--(14,2)--(14,-6)--(21,-10)--(21,-12)--(20,-12)--(20,-10)--(13,-6)--(13,-2)--cycle;
\draw [color=blue!50, fill=blue!50, opacity=0.5] (15,2)--(16,2)--(16,-6)--(9,-10)--(9,-12)--(8,-12)--(8,-10)--(15,-6)--(15,-2)--cycle;
\draw [color=red!50, fill=red!50, opacity=0.5] (17,2)--(18,2)--(23,-2)--(23,-6)--(23,-10)--(23,-12)--(22,-12)--(22,-10)--(22,-6)--(22,-2)--cycle;
\draw [color=blue!50, fill=blue!50, opacity=0.5] (19,2)--(20,2)--(18,-2)--(18,-6)--(11,-10)--(11,-12)--(10,-12)--(10,-10)--(17,-6)--(17,-2)--cycle;
\draw [color=red!50, fill=red!50, opacity=0.5] (23,2)--(24,2)--(26,-2)--(26,-6)--(26,-10)--(26,-12)--(25,-12)--(25,-10)--(25,-6)--(25,-2)--cycle;
\draw [color=blue!50, fill=blue!50, opacity=0.5] (25,2)--(26,2)--(21,-2)--(21,-6)--(14,-10)--(14,-12)--(13,-12)--(13,-10)--(20,-6)--(20,-2)--cycle;
\draw [color=red!50, fill=red!50, opacity=0.5] (27,2)--(28,2)--(28,-12)--(27,-12)--cycle;
\draw(1,2)--(1,-2)--(1,-6)--(1,-10)--(1,-12);
\draw(1.5,2) node[above]{$\bi^{\mu^{(1)}_1}$};
\draw(2,2)--(2,-2)--(2,-6)--(2,-10)--(2,-12);
\draw(3,2)--(3,-2)--(8,-6)--(15,-10)--(15,-12);
\draw(3.5,2) node[above]{$\bi^{\xi^{(1)}_1}$};
\draw(4,2)--(4,-2)--(9,-6)--(16,-10)--(16,-12);
\draw(5,2)--(5,-2)--(3,-6)--(3,-10)--(3,-12);
\draw(5.5,2) node[above]{$\bi^{\mu^{(1)}_2}$};
\draw(6,2)--(6,-2)--(4,-6)--(4,-10)--(4,-12);
\draw(7.5,2) node[above]{$\cdots$};
\draw(9,2)--(9,-2)--(11,-6)--(18,-10)--(18,-12);
\draw(9.5,2) node[above]{$\bi^{\xi^{(i)}_{j\hspace{-0.5mm}-\hspace{-0.5mm}2}}$};
\draw(10,2)--(10,-2)--(12,-6)--(19,-10)--(19,-12);
\draw(11,2)--(11,-2)--(6,-6)--(6,-10)--(6,-12);
\draw(11.5,2) node[above]{$\bi^{\mu^{(i)}_{j\hspace{-0.5mm}-\hspace{-0.5mm}1}}$};
\draw(12,2)--(12,-2)--(7,-6)--(7,-10)--(7,-12);
\draw(13,2)--(13,-2)--(13,-6)--(20,-10)--(20,-12);
\draw(13.5,2) node[above]{$\bi^{\xi^{(i)}_{j\hspace{-0.5mm}-\hspace{-0.5mm}1}}$};
\draw(14,2)--(14,-2)--(14,-6)--(21,-10)--(21,-12);
\draw(15,2)--(15,-2)--(15,-6)--(8,-10)--(8,-12);
\draw(15.5,2) node[above]{$\bi^{\mu^{(i)}_{j}}$};
\draw(16,2)--(16,-2)--(16,-6)--(9,-10)--(9,-12);
\draw(17,2)--(22,-2)--(22,-6)--(22,-10)--(22,-12);
\draw(17.5,2) node[above]{$\bi^{\xi^{(i)}_{j}}$};
\draw(18,2)--(23,-2)--(23,-6)--(23,-10)--(23,-12);
\draw(19,2)--(17,-2)--(17,-6)--(10,-10)--(10,-12);
\draw(19.5,2) node[above]{$\bi^{\mu^{(i)}_{j\hspace{-0.5mm}+\hspace{-0.5mm}1}}$};
\draw(20,2)--(18,-2)--(18,-6)--(11,-10)--(11,-12);
\draw(21.5,2) node[above]{$\cdots$};
\draw(23,2)--(25,-2)--(25,-6)--(25,-10)--(25,-12);
\draw(23.5,2) node[above]{$\bi^{\xi^{(l)}_{n\hspace{-0.5mm}-\hspace{-0.5mm}1}}$};
\draw(24,2)--(26,-2)--(26,-6)--(26,-10)--(26,-12);
\draw(25,2)--(20,-2)--(20,-6)--(13,-10)--(13,-12);
\draw(25.5,2) node[above]{$\bi^{\mu^{(l)}_{n}}$};
\draw(26,2)--(21,-2)--(21,-6)--(14,-10)--(14,-12);
\draw(27,2)--(27,-2)--(27,-6)--(27,-10)--(27,-12);
\draw(27.5,2) node[above]{$\bi^{\xi^{(l)}_{n}}$};
\draw(28,2)--(28,-2)--(28,-6)--(28,-10)--(28,-12);
\draw[color=white, opacity=1] (0,-2)--(28.5,-2);
\draw[color=white, opacity=1] (0,-6)--(28.5,-6);
\draw[color=white, opacity=1] (0,-10)--(28.5,-10);
\draw(7.5,-2) node{$\cdots$};
\draw(5,-6) node{$\cdots$};
\draw(5,-10) node{$\cdots$};
\draw(10.2,-6) node{$\cdots$};
\draw(19,-2) node{$\cdots$};
\draw(18.8,-6) node{$\cdots$};
\draw(24,-2) node{$\cdots$};
\draw(24,-6) node{$\cdots$};
\draw(12.2,-10) node{$\cdots$};
\draw(5,-12) node{$\cdots$};
\draw(17,-10) node{$\cdots$};
\draw(24,-10) node{$\cdots$};
\draw(12.2,-12) node{$\cdots$};
\draw(17,-12) node{$\cdots$};
\draw(24,-12) node{$\cdots$};
\draw(0.5,0) node[left]{$\psi^R_{i,j}$};
\draw(0.5,-4) node[left]{$\psi^L_{i,j}\psi_{r_1+d_{i,j}}\hspace{-1mm} \cdots \psi_{r_s+d_{i,j}}$};
\draw(0.5,-8) node[left]{$\psi^D_{i,j}$};
\draw [rounded corners, color=green, fill=white!50] (14.3,-3.3)--(21.7,-3.3)--(21.7,-4.7)--(14.3,-4.7)--cycle;
\draw(18,-4) node{$\psi_{r_1+d_{i,j}}\hspace{-1mm} \cdots \psi_{r_s+d_{i,j}}$};
\end{braid}
\end{align*}
Let \(\bj = w_{i,j}^L w_{r_2+d_{i,j}} \cdots w_{r_s + d_{i,j}} w_{i,j}^R \bi^{\blam}\). Since \(s_{r_1}w^D_{i,j} = w_{i,j}^Ds_{r_1 + d_{i,j}}\) and \(\ell(s_{r_1}w^D_{i,j}) = \ell(w^D_{i,j}) + \ell(s_{r_1})\), it follows from Lemma \ref{rewrite} that
\begin{align}\label{psiD}
\psi_{r_1} \psi^D_{i,j} 1_{\bj} = \psi^D_{i,j} \psi_{r_1 + d_{i,j}}1_{\bj} + \sum_{u \triangleleft w_{i,j}^D} c_u \psi_u \psi_{r_1 + d_{i,j}}^{\epsilon_u} f_u(y)1_{\bj}
\end{align}
for some constants \(c_u\in \mathcal{O}\), polynomials \(f_u(y_1, \ldots, y_{b+d})\), and \(\epsilon_u \in \{0,1\}\). Thus it remains to show that 
\begin{align*}
\psi_u \psi_{r_1 + d_{i,j}}^{\epsilon_u} f_u(y) \psi_{r_2 + d_{i,j}} \cdots \psi_{r_s+ d_{i,j}} \psi^R_{i,j}\psi^L_{i,j} x^{\blam}=0 \in S^{\blam}/C_{\triangleright \bmu}
\end{align*}
for all \(u\) in the sum in (\ref{psiD}). Let \(s_{t_1^R} \cdots s_{t_{N_R}^R}\) be the preferred reduced expression for \(w_{i,j}^R\), and similarly for \(w_{i,j}^L\). Pushing the \(y\)'s to the right to act (as zero) on \(x^{\blam}\), this is by lemma \ref{rewrite} an \(\mathcal{O}\)-linear combination of terms of the form 
\begin{align}\label{lotsofpsis}
\psi_u \psi^{\epsilon_u}_{r_1 + d_{i,j}} \psi_{r_2 + d_{i,j}}^{\epsilon_2} \cdots \psi_{r_s + d_{i,j}}^{\epsilon_s} \psi_{t_1^R}^{\epsilon_{s+1}} \cdots \psi_{t_{N_R}^R}^{\epsilon_{s+N_R}} \psi_{t_1^L}^{\epsilon_{s+N_R + 1}} \cdots \psi_{t_{N^L}^L}^{\epsilon_s + N_R + N_L} x^{\blam}
\end{align}
for some \(\epsilon_i \in \{0,1\}\). Write \(\Theta\) for the sequence of \(\psi\)'s in (\ref{lotsofpsis}). Assume \(v^\UUU\) appears with nonzero coefficient when \(\Theta v^{\blam}\) is expanded in terms of basis elements. Then it follows from Lemma \ref{subword} that one can write \(w^\UUU\) diagrammatically by removing crossings from the diagram
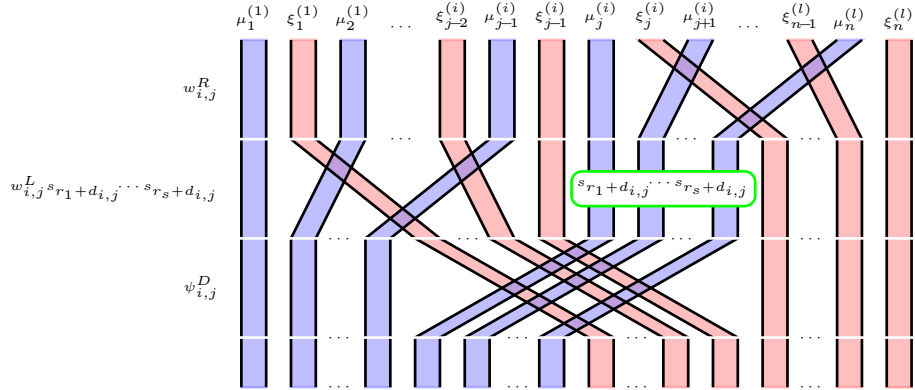
\begin{align*}
\begin{braid}\tikzset{baseline=0mm}
\draw [color=blue!50, fill=blue!50, opacity=0.5] (1,2)--(2,2)--(2,-12)--(1,-12)--cycle;
\draw [color=red!50, fill=red!50, opacity=0.5] (3,2)--(4,2)--(4,-2)--(9,-6)--(16,-10)--(16,-12)--(15,-12)--(15,-10)--(8,-6)--(3,-2)--cycle;
\draw [color=blue!50, fill=blue!50, opacity=0.5] (5,2)--(6,2)--(6,-2)--(4,-6)--(4,-10)--(4,-12)--(3,-12)--(3,-6)--(5,-2)--cycle;
\draw [color=red!50, fill=red!50, opacity=0.5] (9,2)--(10,2)--(10,-2)--(12,-6)--(19,-10)--(19,-12)--(18,-12)--(18,-10)--(11,-6)--(9,-2)--cycle;
\draw [color=blue!50, fill=blue!50, opacity=0.5] (11,2)--(12,2)--(12,-2)--(7,-6)--(7,-10)--(7,-12)--(6,-12)--(6,-10)--(6,-6)--(11,-2)--cycle;
\draw [color=red!50, fill=red!50, opacity=0.5] (13,2)--(14,2)--(14,-6)--(21,-10)--(21,-12)--(20,-12)--(20,-10)--(13,-6)--(13,-2)--cycle;
\draw [color=blue!50, fill=blue!50, opacity=0.5] (15,2)--(16,2)--(16,-6)--(9,-10)--(9,-12)--(8,-12)--(8,-10)--(15,-6)--(15,-2)--cycle;
\draw [color=red!50, fill=red!50, opacity=0.5] (17,2)--(18,2)--(23,-2)--(23,-6)--(23,-10)--(23,-12)--(22,-12)--(22,-10)--(22,-6)--(22,-2)--cycle;
\draw [color=blue!50, fill=blue!50, opacity=0.5] (19,2)--(20,2)--(18,-2)--(18,-6)--(11,-10)--(11,-12)--(10,-12)--(10,-10)--(17,-6)--(17,-2)--cycle;
\draw [color=red!50, fill=red!50, opacity=0.5] (23,2)--(24,2)--(26,-2)--(26,-6)--(26,-10)--(26,-12)--(25,-12)--(25,-10)--(25,-6)--(25,-2)--cycle;
\draw [color=blue!50, fill=blue!50, opacity=0.5] (25,2)--(26,2)--(21,-2)--(21,-6)--(14,-10)--(14,-12)--(13,-12)--(13,-10)--(20,-6)--(20,-2)--cycle;
\draw [color=red!50, fill=red!50, opacity=0.5] (27,2)--(28,2)--(28,-12)--(27,-12)--cycle;
\draw(1,2)--(1,-2)--(1,-6)--(1,-10)--(1,-12);
\draw(1.5,2) node[above]{$\mu^{(1)}_1$};
\draw(2,2)--(2,-2)--(2,-6)--(2,-10)--(2,-12);
\draw(3,2)--(3,-2)--(8,-6)--(15,-10)--(15,-12);
\draw(3.5,2) node[above]{$\xi^{(1)}_1$};
\draw(4,2)--(4,-2)--(9,-6)--(16,-10)--(16,-12);
\draw(5,2)--(5,-2)--(3,-6)--(3,-10)--(3,-12);
\draw(5.5,2) node[above]{$\mu^{(1)}_2$};
\draw(6,2)--(6,-2)--(4,-6)--(4,-10)--(4,-12);
\draw(7.5,2) node[above]{$\cdots$};
\draw(9,2)--(9,-2)--(11,-6)--(18,-10)--(18,-12);
\draw(9.5,2) node[above]{$\xi^{(i)}_{j\hspace{-0.5mm}-\hspace{-0.5mm}2}$};
\draw(10,2)--(10,-2)--(12,-6)--(19,-10)--(19,-12);
\draw(11,2)--(11,-2)--(6,-6)--(6,-10)--(6,-12);
\draw(11.5,2) node[above]{$\mu^{(i)}_{j\hspace{-0.5mm}-\hspace{-0.5mm}1}$};
\draw(12,2)--(12,-2)--(7,-6)--(7,-10)--(7,-12);
\draw(13,2)--(13,-2)--(13,-6)--(20,-10)--(20,-12);
\draw(13.5,2) node[above]{$\xi^{(i)}_{j\hspace{-0.5mm}-\hspace{-0.5mm}1}$};
\draw(14,2)--(14,-2)--(14,-6)--(21,-10)--(21,-12);
\draw(15,2)--(15,-2)--(15,-6)--(8,-10)--(8,-12);
\draw(15.5,2) node[above]{$\mu^{(i)}_{j}$};
\draw(16,2)--(16,-2)--(16,-6)--(9,-10)--(9,-12);
\draw(17,2)--(22,-2)--(22,-6)--(22,-10)--(22,-12);
\draw(17.5,2) node[above]{$\xi^{(i)}_{j}$};
\draw(18,2)--(23,-2)--(23,-6)--(23,-10)--(23,-12);
\draw(19,2)--(17,-2)--(17,-6)--(10,-10)--(10,-12);
\draw(19.5,2) node[above]{$\mu^{(i)}_{j\hspace{-0.5mm}+\hspace{-0.5mm}1}$};
\draw(20,2)--(18,-2)--(18,-6)--(11,-10)--(11,-12);
\draw(21.5,2) node[above]{$\cdots$};
\draw(23,2)--(25,-2)--(25,-6)--(25,-10)--(25,-12);
\draw(23.5,2) node[above]{$\xi^{(l)}_{n\hspace{-0.5mm}-\hspace{-0.5mm}1}$};
\draw(24,2)--(26,-2)--(26,-6)--(26,-10)--(26,-12);
\draw(25,2)--(20,-2)--(20,-6)--(13,-10)--(13,-12);
\draw(25.5,2) node[above]{$\mu^{(l)}_{n}$};
\draw(26,2)--(21,-2)--(21,-6)--(14,-10)--(14,-12);
\draw(27,2)--(27,-2)--(27,-6)--(27,-10)--(27,-12);
\draw(27.5,2) node[above]{$\xi^{(l)}_{n}$};
\draw(28,2)--(28,-2)--(28,-6)--(28,-10)--(28,-12);
\draw[color=white, opacity=1] (0,-2)--(28.5,-2);
\draw[color=white, opacity=1] (0,-6)--(28.5,-6);
\draw[color=white, opacity=1] (0,-10)--(28.5,-10);
\draw(7.5,-2) node{$\cdots$};
\draw(5,-6) node{$\cdots$};
\draw(5,-10) node{$\cdots$};
\draw(10.2,-6) node{$\cdots$};
\draw(19,-2) node{$\cdots$};
\draw(18.8,-6) node{$\cdots$};
\draw(24,-2) node{$\cdots$};
\draw(24,-6) node{$\cdots$};
\draw(12.2,-10) node{$\cdots$};
\draw(5,-12) node{$\cdots$};
\draw(17,-10) node{$\cdots$};
\draw(24,-10) node{$\cdots$};
\draw(12.2,-12) node{$\cdots$};
\draw(17,-12) node{$\cdots$};
\draw(24,-12) node{$\cdots$};
\draw(0.5,0) node[left]{$w^R_{i,j}$};
\draw(0.5,-4) node[left]{$w^L_{i,j}s_{r_1+d_{i,j}}\hspace{-1mm} \cdots s_{r_s+d_{i,j}}$};
\draw(0.5,-8) node[left]{$\psi^D_{i,j}$};
\draw [rounded corners, color=green, fill=white!50] (14.3,-3.3)--(21.7,-3.3)--(21.7,-4.7)--(14.3,-4.7)--cycle;
\draw(18,-4) node{$s_{r_1+d_{i,j}}\hspace{-1mm} \cdots s_{r_s+d_{i,j}}$};
\end{braid}
\end{align*}
and in particular, removing at least one crossing from \(w_{i,j}^D\), the third row of the diagram, since \(u \triangleleft w_{i,j}^D\). But in any case, this implies that there is a pink strand that ends to the left of a blue strand, i.e., some \(t\leq a\) such that \((w^\UUU)^{-1}(t)\) is in \(\xi^{(h)}_k\) for some \(h,k\). Then \(\sh(\UUU_{\leq a}) \neq \bmu\). But since \(N_{\bmu}\) is an \(R_{\alpha,\beta}\)-submodule, we must have \(\UUU \in B_{\bmu}\). This implies that \(\UUU \in B_{\triangleright\bmu}\), and hence \( x^\UUU=0 \in S^\lambda/C_{\bmu}\). 
\end{proof}
Let \(A_{\bmu}\) be a Garnir node in \(\bmu\). This is also a Garnir node of \(\blam\), and when we consider it as such, we will label it with \(A_{\blam}\). Let \({\bf B}^{A_{\blam}}\) be the Garnir belt associated with \(A_{\blam}\), and let \({\bf B}^{A_{\bmu}}\) be the Garnir belt of nodes in \(\bmu\). Assume \(A_{\blam}\) is in row \(j\) of the \(i\)th component of \(\blam\). We subdivide the sets of nodes of \(\mu^{(i)}_{j}\), \(\mu^{(i)}_{j+1}\) and \(\xi^{(i)}_{j}\) in the following fashion:
\begin{enumerate}
\item We subdivide \(\mu^{(i)}_j\) into three sets:
\begin{enumerate}
\item Let \(\mu^{A,1}\) be the nodes of \(\mu^{(i)}_j\) not contained in \({\bf B}^{A_{\bmu}}\).
\item Let \(\mu^{A,2}\) be the nodes of \(\mu^{(i)}_j\) contained in bricks in \({\bf B}^{A_{\bmu}}\).
\item Let \(\mu^{A,3}\) be the nodes of \(\mu^{(i)}_j\) contained in \({\bf B}^{A_{\bmu}}\), but not contained in any brick.
\end{enumerate}
\item We subdivide \(\xi^{(i)}_j\) into three sets:
\begin{enumerate}
\item Let \(\xi^{A,1}\) be the nodes of \(\xi^{(i)}_j\) contained in a brick in \({\bf B}^{A_{\blam}}\) which contains nodes of \(\bmu\).
\item Let \(\xi^{A,2}\) be the nodes of \(\xi^{(i)}_j\) contained in a brick in \({\bf B}^{A_{\blam}}\) which is entirely contained in \(\bxi\).
\item Let \(\xi^{A,3}\) be the nodes of \(\xi^{(i)}_j\) contained in \({\bf B}^{A_{\blam}}\), but not contained in any brick.
\end{enumerate}
\item We subdivide \(\mu^{(i)}_{j+1}\) into three sets:
\begin{enumerate}
\item Let \(\mu_{A,1}\) be the nodes of \(\mu^{(i)}_{j+1}\) contained in \({\bf B}^{A_{\bmu}}\) but not contained in any brick.
\item Let \(\mu_{A,2}\) be the nodes of \(\mu^{(i)}_{j+1}\) contained in bricks in \({\bf B}^{A_{\bmu}}\).
\item Let \(\mu_{A,3}\) be the nodes of \(\mu^{(i)}_{j+1}\) not contained in \({\bf B}^{A_{\bmu}}\).
\end{enumerate}
\end{enumerate}
For example, if \(l=1\), \(\mu=(14,10)\), \(\lambda=(23,18)\), \(A=(1,8,1)\) and \(e=3\), then the subdivisions are as follows:
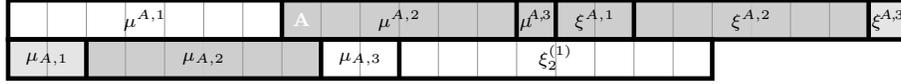
\begin{align*}
\begin{tikzpicture}[scale=0.52]
\draw [color=gray!50, fill=gray!50, opacity=0.4] (1,0)--(17,0)--(17,-1)--(1,-1)--cycle;
\draw [color=gray!50, fill=gray!50, opacity=0.4] (-6,-1)--(2,-1)--(2,-2)--(-6,-2)--cycle;
\draw [color=gray!50, fill=gray!50, opacity=0.6] (1,0)--(16,0)--(16,-1)--(1,-1)--cycle;
\draw [color=gray!50, fill=gray!50, opacity=0.6] (-4,-1)--(2,-1)--(2,-2)--(-4,-2)--cycle;
\draw[color=black!40] (-5,0)--(-5,-2);
\draw[color=black!40] (-4,0)--(-4,-2);
\draw[color=black!40] (-3,0)--(-3,-2);
\draw[color=black!40] (-2,0)--(-2,-2);
\draw[color=black!40] (-1,0)--(-1,-2);
\draw[color=black!40] (0,0)--(0,-2);
\draw[color=black!40] (1,0)--(1,-2);
\draw[color=black!40] (2,0)--(2,-2);
\draw[color=black!40] (3,0)--(3,-2);
\draw[color=black!40] (4,0)--(4,-2);
\draw[color=black!40] (5,0)--(5,-2);
\draw[color=black!40] (6,0)--(6,-2);
\draw[color=black!40] (7,0)--(7,-2);
\draw[color=black!40] (8,0)--(8,-2);
\draw[color=black!40] (9,0)--(9,-2);
\draw[color=black!40] (10,0)--(10,-2);
\draw[color=black!40] (11,0)--(11,-2);
\draw[color=black!40] (12,0)--(12,-1);
\draw[color=black!40] (13,0)--(13,-1);
\draw[color=black!40] (14,0)--(14,-1);
\draw[color=black!40] (15,0)--(15,-1);
\draw [color=white, opacity=0.5] (10,0)--(10,-1);
\draw [color=white, opacity=0.5] (13,0)--(13,-1);
\draw [color=white, opacity=0.5] (-1,-1)--(-1,-2);
\draw[line width=0.6mm](-6,0)--(17,0)--(17,-1)--(-6,-1)--cycle;
\draw[line width=0.6mm](-6,-1)--(12,-1)--(12,-2)--(-6,-2)--cycle;
\draw[line width=0.6mm](8,0)--(8,-1);
\draw[line width=0.6mm](10,0)--(10,-1);
\draw[line width=0.6mm](7,0)--(7,-1);
\draw[line width=0.6mm](16,0)--(16,-1);
\draw[line width=0.6mm](2,-1)--(2,-2);
\draw[line width=0.6mm](4,-1)--(4,-2);
\draw[line width=0.6mm](-4,-1)--(-4,-2);
\draw[line width=0.6mm](1,0)--(1,-1);
\draw(-2.5,0) node[below]{$\scriptstyle \mu^{A,1}$};
\draw(4,0) node[below]{$\scriptstyle \mu^{A,2}$};
\draw(7.5,0) node[below]{$\scriptstyle \mu^{\hspace{-0.8mm}A\hspace{-0.2mm},\hspace{-0.3mm}3}$};
\draw(9,0) node[below]{$\scriptstyle \xi^{A,1}$};
\draw(13,0) node[below]{$\scriptstyle \xi^{A,2}$};
\draw(16.5,0) node[below]{$\scriptstyle \xi^{\hspace{-0.4mm}A\hspace{-0.2mm},\hspace{-0.3mm}3}$};
\draw(-5,-1.1) node[below]{$\scriptstyle \mu_{A,1}$};
\draw(-1,-1.1) node[below]{$\scriptstyle \mu_{A,2}$};
\draw(3,-1.1) node[below]{$\scriptstyle \mu_{A,3}$};
\draw(8,-0.9) node[below]{$\scriptstyle \xi_{2}^{(1)}$};
\draw [color=white] (1.5,-0.1) node[below]{$\scriptstyle {\mathbf A}$};
\end{tikzpicture}
\end{align*}
Now write \(w^R_{i,j}=w^{R'}_{i,j} w^{R''}_{i,j}\), where \(w^{R'}_{i,j}\), \(w^{R''}_{i,j}\) are given as follows:
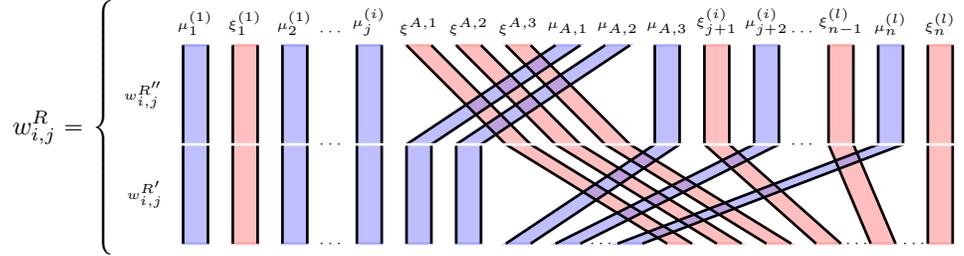
\begin{align*}
w^R_{i,j} = \begin{cases}
\begin{braid}\tikzset{baseline=-7mm}
\draw [color=blue!50, fill=blue!50, opacity=0.5] (1,2)--(1,-2)--(1,-6)--(2,-6)--(2,2)--cycle;
\draw [color=red!50, fill=red!50, opacity=0.5] (3,2)--(3,-2)--(3,-6)--(4,-6)--(4,2)--cycle;
\draw [color=blue!50, fill=blue!50, opacity=0.5] (5,2)--(5,-2)--(5,-6)--(6,-6)--(6,2)--cycle;\draw [color=blue!50, fill=blue!50, opacity=0.5] (8,2)--(8,-2)--(8,-6)--(9,-6)--(9,2)--cycle;
\draw [color=red!50, fill=red!50, opacity=0.5] (10,2)--(14,-2)--(20.5,-6)--(21.5,-6)--(15,-2)--(11,2)--cycle;
\draw [color=red!50, fill=red!50, opacity=0.5] (12,2)--(16,-2)--(22.5,-6)--(23.5,-6)--(17,-2)--(13,2)--cycle;
\draw [color=red!50, fill=red!50, opacity=0.5] (14,2)--(18,-2)--(24.5,-6)--(25.5,-6)--(19,-2)--(15,2)--cycle;
\draw [color=blue!50, fill=blue!50, opacity=0.5] (16,2)--(10,-2)--(10,-6)--(11,-6)--(11,-2)--(17,2)--cycle;
\draw [color=blue!50, fill=blue!50, opacity=0.5] (18,2)--(12,-2)--(12,-6)--(13,-6)--(13,-2)--(19,2)--cycle;
\draw [color=blue!50, fill=blue!50, opacity=0.5] (20,2)--(20,-2)--(14,-6)--(15,-6)--(21,-2)--(21,2)--cycle;
\draw [color=red!50, fill=red!50, opacity=0.5] (22,2)--(22,-2)--(26.5,-6)--(27.5,-6)--(23,-2)--(23,2)--cycle;
\draw [color=blue!50, fill=blue!50, opacity=0.5] (24,2)--(24,-2)--(16,-6)--(17,-6)--(25,-2)--(25,2)--cycle;
\draw [color=red!50, fill=red!50, opacity=0.5] (27,2)--(27,-2)--(28.7,-6)--(29.7,-6)--(28,-2)--(28,2)--cycle;
\draw [color=blue!50, fill=blue!50, opacity=0.5] (29,2)--(29,-2)--(18.5,-6)--(19.5,-6)--(30,-2)--(30,2)--cycle;
\draw [color=red!50, fill=red!50, opacity=0.5] (31,2)--(31,-6)--(32,-6)--(32,2)--cycle;
\draw(1,2)--(1,-2)--(1,-6);
\draw(1.5,2) node[above]{$\mu_1^{(1)}$};
\draw(2,2)--(2,-2)--(2,-6);
\draw(3,2)--(3,-6);
\draw(3.5,2) node[above]{$\xi_1^{(1)}$};
\draw(4,2)--(4,-6);
\draw(5,2)--(5,-6);
\draw(5.5,2) node[above]{$\mu_2^{(1)}$};
\draw(6,2)--(6,-6);
\draw(7,2) node[above]{$\cdots$};
\draw(8,2)--(8,-6);
\draw(8.5,2) node[above]{$\mu^{(i)}_j$};
\draw(9,2)--(9,-6);
\draw(10,2)--(14,-2)--(20.5,-6);
\draw(10.5,2) node[above]{$\xi^{A,1}$};
\draw(11,2)--(15,-2)--(21.5,-6);
\draw(12,2)--(16,-2)--(22.5,-6);
\draw(12.5,2) node[above]{$\xi^{A,2}$};
\draw(13,2)--(17,-2)--(23.5,-6);
\draw(14,2)--(18,-2)--(24.5,-6);
\draw(14.5,2) node[above]{$\xi^{A,3}$};
\draw(15,2)--(19,-2)--(25.5,-6);
\draw(16,2)--(10,-2)--(10,-6);
\draw(16.5,2) node[above]{$\mu_{A,1}$};
\draw(17,2)--(11,-2)--(11,-6);
\draw(18,2)--(12,-2)--(12,-6);
\draw(18.5,2) node[above]{$\mu_{A,2}$};
\draw(19,2)--(13,-2)--(13,-6);
\draw(20,2)--(20,-2)--(14,-6);
\draw(20.5,2) node[above]{$\mu_{A,3}$};
\draw(21,2)--(21,-2)--(15,-6);
\draw(22,2)--(22,-2)--(26.5,-6);
\draw(22.5,2) node[above]{$\xi^{(i)}_{j+1}$};
\draw(23,2)--(23,-2)--(27.5,-6);
\draw(24,2)--(24,-2)--(16,-6);
\draw(24.5,2) node[above]{$\mu^{(i)}_{j+2}$};
\draw(25,2)--(25,-2)--(17,-6);
\draw(26,2) node[above]{$\cdots$};
\draw(27,2)--(27,-2)--(28.7,-6);
\draw(27.5,2) node[above]{$\xi^{(l)}_{n-1}$};
\draw(28,2)--(28,-2)--(29.7,-6);
\draw(29,2)--(29,-2)--(18.5,-6);
\draw(29.5,2) node[above]{$\mu^{(l)}_n$};
\draw(30,2)--(30,-2)--(19.5,-6);
\draw(31,2)--(31,-6);
\draw(31.5,2) node[above]{$\xi^{(l)}_n$};
\draw(32,2)--(32,-6);
\draw[color=white, opacity=1] (0.8,-2)--(32.2,-2);
\draw(7,-2) node{$\cdots$};
\draw(7,-6) node{$\cdots$};
\draw(17.9,-6) node{$\cdots$};
\draw(30.5,-6) node{$\cdots$};
\draw(28.2,-6) node{$\cdots$};
\draw(26,-2) node{$\cdots$};
\draw(0.6,0) node[left]{$w^{R''}_{i,j}$};
\draw(0.6,-4) node[left]{$w^{R'}_{i,j}$};
\end{braid}.
\end{cases}
\end{align*}
Let \(\GGG^{A_{\blam}} = \omega \TTT^{A_{\blam}}\) and \(\GGG^{A_{\bmu}} = \zeta \TTT^{A_{\bmu}}\), where \(\omega \in \mathscr{D}^{A_{\blam}}\) and \(\zeta \in \mathscr{D}^{A_{\bmu}}\). Then \(\omega=\omega_2 \omega_1\), where \(\omega_1, \omega_2 \in \mathfrak{S}^{A_{\blam}}\) are given as follows:
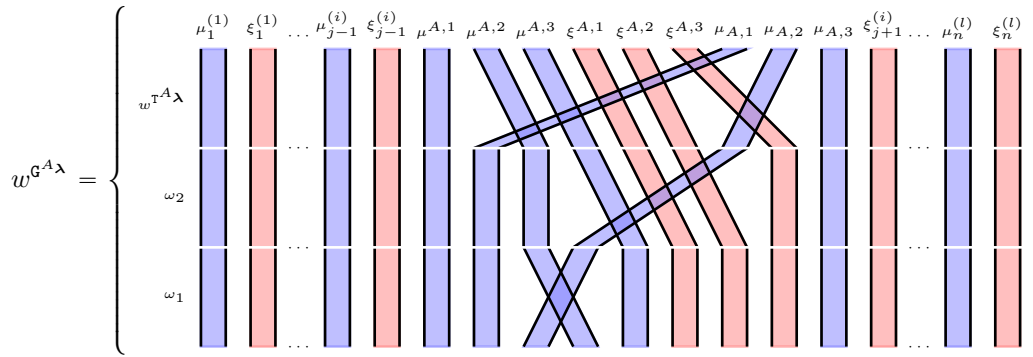
\begin{align*}
w^{\GGG^{A_{\blam}}} = \begin{cases}
\begin{braid}\tikzset{baseline=-15mm}
\draw [color=blue!50, fill=blue!50, opacity=0.5] (-5,2)--(-5,-10)--(-4,-10)--(-4,2)--cycle;
\draw [color=red!50, fill=red!50, opacity=0.5] (-3,2)--(-3,-10)--(-2,-10)--(-2,2)--cycle;
\draw [color=blue!50, fill=blue!50, opacity=0.5] (0,2)--(0,-10)--(1,-10)--(1,2)--cycle;
\draw [color=red!50, fill=red!50, opacity=0.5] (2,2)--(2,-10)--(3,-10)--(3,2)--cycle;
\draw [color=blue!50, fill=blue!50, opacity=0.5] (4,2)--(4,-10)--(5,-10)--(5,2)--cycle;
\draw [color=blue!50, fill=blue!50, opacity=0.5] (6,2)--(8,-2)--(8,-6)--(10,-10)--(11,-10)--(9,-6)--(9,-2)--(7,2)--cycle;
\draw [color=blue!50, fill=blue!50, opacity=0.5] (8,2)--(10,-2)--(12,-6)--(12,-10)--(13,-10)--(13,-6)--(11,-2)--(9,2)--cycle;
\draw [color=red!50, fill=red!50, opacity=0.5] (10,2)--(12,-2)--(14,-6)--(14,-10)--(15,-10)--(15,-6)--(13,-2)--(11,2)--cycle;
\draw [color=red!50, fill=red!50, opacity=0.5] (12,2)--(14,-2)--(16,-6)--(16,-10)--(17,-10)--(17,-6)--(15,-2)--(13,2)--cycle;
\draw [color=red!50, fill=red!50, opacity=0.5] (14,2)--(18,-2)--(18,-10)--(19,-10)--(19,-2)--(15,2)--cycle;
\draw [color=blue!50, fill=blue!50, opacity=0.5] (16,2)--(6,-2)--(6,-10)--(7,-10)--(7,-2)--(17,2)--cycle;
\draw [color=blue!50, fill=blue!50, opacity=0.5] (18,2)--(16,-2)--(10,-6)--(8,-10)--(9,-10)--(11,-6)--(17,-2)--(19,2)--cycle;
\draw [color=blue!50, fill=blue!50, opacity=0.5] (20,2)--(20,-10)--(21,-10)--(21,2)--cycle;
\draw [color=red!50, fill=red!50, opacity=0.5] (22,2)--(22,-10)--(23,-10)--(23,2)--cycle;
\draw [color=blue!50, fill=blue!50, opacity=0.5] (25,2)--(25,-10)--(26,-10)--(26,2)--cycle;
\draw [color=red!50, fill=red!50, opacity=0.5] (27,2)--(27,-10)--(28,-10)--(28,2)--cycle;
\draw(-5,2)--(-5,-10);
\draw(-4.5,2) node[above]{$\mu^{(1)}_1$};
\draw(-4,2)--(-4,-10);
\draw(-3,2)--(-3,-10);
\draw(-2.5,2) node[above]{$\xi^{(1)}_1$};
\draw(-2,2)--(-2,-10);
\draw(-1,2) node[above]{$\cdots$};
\draw(0,2)--(0,-10);
\draw(0.5,2) node[above]{$\mu^{(i)}_{j-1}$};
\draw(1,2)--(1,-10);
\draw(2,2)--(2,-10);
\draw(2.5,2) node[above]{$\xi^{(i)}_{j-1}$};
\draw(3,2)--(3,-10);
\draw(4,2)--(4,-10);
\draw(4.5,2) node[above]{$\mu^{A,1}$};
\draw(5,2)--(5,-10);
\draw(6,2)--(8,-2)--(8,-6)--(10,-10);
\draw(6.5,2) node[above]{$\mu^{A,2}$};
\draw(7,2)--(9,-2)--(9,-6)--(11,-10);
\draw(8,2)--(10,-2)--(12,-6)--(12,-10);
\draw(8.5,2) node[above]{$\mu^{A,3}$};
\draw(9,2)--(11,-2)--(13,-6)--(13,-10);
\draw(10,2)--(12,-2)--(14,-6)--(14,-10);
\draw(10.5,2) node[above]{$\xi^{A,1}$};
\draw(11,2)--(13,-2)--(15,-6)--(15,-10);
\draw(12,2)--(14,-2)--(16,-6)--(16,-10);
\draw(12.5,2) node[above]{$\xi^{A,2}$};
\draw(13,2)--(15,-2)--(17,-6)--(17,-10);
\draw(14,2)--(18,-2)--(18,-10);
\draw(14.5,2) node[above]{$\xi^{A,3}$};
\draw(15,2)--(19,-2)--(19,-10);
\draw(16,2)--(6,-2)--(6,-10);
\draw(16.5,2) node[above]{$\mu_{A,1}$};
\draw(17,2)--(7,-2)--(7,-10);
\draw(18,2)--(16,-2)--(10,-6)--(8,-10);
\draw(18.5,2) node[above]{$\mu_{A,2}$};
\draw(19,2)--(17,-2)--(11,-6)--(9,-10);
\draw(20,2)--(20,-10);
\draw(20.5,2) node[above]{$\mu_{A,3}$};
\draw(21,2)--(21,-10);
\draw(22,2)--(22,-10);
\draw(22.5,2) node[above]{$\xi^{(i)}_{j+1}$};
\draw(23,2)--(23,-10);
\draw(24,2) node[above]{$\cdots$};
\draw(25,2)--(25,-10);
\draw(25.5,2) node[above]{$\mu^{(l)}_n$};
\draw(26,2)--(26,-10);
\draw(27,2)--(27,-10);
\draw(27.5,2) node[above]{$\xi^{(l)}_n$};
\draw(28,2)--(28,-10);
\draw[color=white, opacity=1] (-6.2,-2)--(28.2,-2);
\draw[color=white, opacity=1] (-6.2,-6)--(28.2,-6);
\draw(-1,-2) node{$\cdots$};
\draw(-1,-6) node{$\cdots$};
\draw(-1,-10) node{$\cdots$};
\draw(24,-2) node{$\cdots$};
\draw(24,-6) node{$\cdots$};
\draw(24,-10) node{$\cdots$};
\draw(-5.1,0) node[left]{$w^{\TTT^{A_{\blam}}}$};
\draw(-5.2,-4) node[left]{$\omega_2$};
\draw(-5.2,-8) node[left]{$\omega_1$};
\end{braid}.
\end{cases}
\end{align*}
Then
\begin{align}\label{movega}
\zeta w^{\TTT^{A_{\bmu}}} w^{\bmu \bxi} = w^{(\GGG^{A_{\bmu}})\ttt^{\bxi}} = w_{i,j}^D w_{i,j}^L w_{i,j}^{R'} w^{\GGG^{A_{\blam}}} = w_{i,j}^Dw_{i,j}^L w_{i,j}^{R'}\omega_1 \omega_2 w^{\TTT^{A_{\blam}}}.
\end{align}
This is best seen diagrammatically. On the right side of (\ref{movega}) we have
\begin{align}\label{bigdiagram}
\begin{braid}\tikzset{baseline=-20mm}
\draw [color=blue!50, fill=blue!50, opacity=0.5] (-5,2)--(-5,-18)--(-4,-18)--(-4,2)--cycle;
\draw [color=red!50, fill=red!50, opacity=0.5] (-3,2)--(-3,-10)--(-0.5,-14)--(14,-18)--(15,-18)--(0.5,-14)--(-2,-10)--(-2,2)--cycle;
\draw [color=blue!50, fill=blue!50, opacity=0.5] (0,2)--(0,-10)--(-2.5,-14)--(-2.5,-18)--(-1.5,-18)--(-1.5,-14)--(1,-10)--(1,2)--cycle;
\draw [color=red!50, fill=red!50, opacity=0.5] (2,2)--(2,-14)--(16.5,-18)--(17.5,-18)--(3,-14)--(3,2)--cycle;
\draw [color=blue!50, fill=blue!50, opacity=0.5] (4,2)--(4,-14)--(0,-18)--(1,-18)--(5,-14)--(5,2)--cycle;
\draw [color=blue!50, fill=blue!50, opacity=0.5] (6,2)--(8,-2)--(8,-6)--(10,-10)--(10,-14)--(6,-18)--(7,-18)--(11,-14)--(11,-10)--(9,-6)--(9,-2)--(7,2)--cycle;
\draw [color=blue!50, fill=blue!50, opacity=0.5] (8,2)--(10,-2)--(12,-6)--(12,-10)--(12,-14)--(8,-18)--(9,-18)--(13,-14)--(13,-10)--(13,-6)--(11,-2)--(9,2)--cycle;
\draw [color=red!50, fill=red!50, opacity=0.5] (10,2)--(12,-2)--(14,-6)--(14,-10)--(18.5,-14)--(18.5,-18)--(19.5,-18)--(19.5,-14)--(15,-10)--(15,-6)--(13,-2)--(11,2)--cycle;
\draw [color=red!50, fill=red!50, opacity=0.5] (12,2)--(14,-2)--(16,-6)--(16,-10)--(20.5,-14)--(20.5,-18)--(21.5,-18)--(21.5,-14)--(17,-10)--(17,-6)--(15,-2)--(13,2)--cycle;
\draw [color=red!50, fill=red!50, opacity=0.5] (14,2)--(18,-2)--(18,-10)--(22.5,-14)--(22.5,-18)--(23.5,-18)--(23.5,-14)--(19,-10)--(19,-2)--(15,2)--cycle;
\draw [color=blue!50, fill=blue!50, opacity=0.5] (16,2)--(6,-2)--(6,-10)--(6,-14)--(2,-18)--(3,-18)--(7,-14)--(7,-10)--(7,-2)--(17,2)--cycle;
\draw [color=blue!50, fill=blue!50, opacity=0.5] (18,2)--(16,-2)--(10,-6)--(8,-10)--(8,-14)--(4,-18)--(5,-18)--(9,-14)--(9,-10)--(11,-6)--(17,-2)--(19,2)--cycle;
\draw [color=blue!50, fill=blue!50, opacity=0.5] (20,2)--(20,-10)--(14,-14)--(10,-18)--(11,-18)--(15,-14)--(21,-10)--(21,2)--cycle;
\draw [color=red!50, fill=red!50, opacity=0.5] (22,2)--(22,-10)--(24.5,-14)--(24.5,-18)--(25.5,-18)--(25.5,-14)--(23,-10)--(23,2)--cycle;
\draw [color=blue!50, fill=blue!50, opacity=0.5] (25,2)--(25,-10)--(16.5,-14)--(12.3,-18)--(13.3,-18)--(17.5,-14)--(26,-10)--(26,2)--cycle;
\draw [color=red!50, fill=red!50, opacity=0.5] (27,2)--(27,-18)--(28,-18)--(28,2)--cycle;
\draw(-5,2)--(-5,-18);
\draw(-4.5,2) node[above]{$\mu^{(1)}_1$};
\draw(-4,2)--(-4,-18);
\draw(-3,2)--(-3,-10)--(-0.5,-14)--(14,-18);
\draw(-2.5,2) node[above]{$\xi^{(1)}_1$};
\draw(-2,2)--(-2,-10)--(0.5,-14)--(15,-18);
\draw(-1,2) node[above]{$\cdots$};
\draw(0,2)--(0,-10)--(-2.5,-14)--(-2.5,-18);
\draw(0.5,2) node[above]{$\mu^{(i)}_{j-1}$};
\draw(1,2)--(1,-10)--(-1.5,-14)--(-1.5,-18);
\draw(2,2)--(2,-14)--(16.5,-18);
\draw(2.5,2) node[above]{$\xi^{(i)}_{j-1}$};
\draw(3,2)--(3,-14)--(17.5,-18);
\draw(4,2)--(4,-14)--(0,-18);
\draw(4.5,2) node[above]{$\mu^{A,1}$};
\draw(5,2)--(5,-14)--(1,-18);
\draw(6,2)--(8,-2)--(8,-6)--(10,-10)--(10,-14)--(6,-18);
\draw(6.5,2) node[above]{$\mu^{A,2}$};
\draw(7,2)--(9,-2)--(9,-6)--(11,-10)--(11,-14)--(7,-18);
\draw(8,2)--(10,-2)--(12,-6)--(12,-10)--(12,-14)--(8,-18);
\draw(8.5,2) node[above]{$\mu^{A,3}$};
\draw(9,2)--(11,-2)--(13,-6)--(13,-10)--(13,-14)--(9,-18);
\draw(10,2)--(12,-2)--(14,-6)--(14,-10)--(18.5,-14)--(18.5,-18);
\draw(10.5,2) node[above]{$\xi^{A,1}$};
\draw(11,2)--(13,-2)--(15,-6)--(15,-10)--(19.5,-14)--(19.5,-18);
\draw(12,2)--(14,-2)--(16,-6)--(16,-10)--(20.5,-14)--(20.5,-18);
\draw(12.5,2) node[above]{$\xi^{A,2}$};
\draw(13,2)--(15,-2)--(17,-6)--(17,-10)--(21.5,-14)--(21.5,-18);
\draw(14,2)--(18,-2)--(18,-10)--(22.5,-14)--(22.5,-18);
\draw(14.5,2) node[above]{$\xi^{A,3}$};
\draw(15,2)--(19,-2)--(19,-10)--(23.5,-14)--(23.5,-18);
\draw(16,2)--(6,-2)--(6,-10)--(6,-14)--(2,-18);
\draw(16.5,2) node[above]{$\mu_{A,1}$};
\draw(17,2)--(7,-2)--(7,-10)--(7,-14)--(3,-18);
\draw(18,2)--(16,-2)--(10,-6)--(8,-10)--(8,-14)--(4,-18);
\draw(18.5,2) node[above]{$\mu_{A,2}$};
\draw(19,2)--(17,-2)--(11,-6)--(9,-10)--(9,-14)--(5,-18);
\draw(20,2)--(20,-10)--(14,-14)--(10,-18);
\draw(20.5,2) node[above]{$\mu_{A,3}$};
\draw(21,2)--(21,-10)--(15,-14)--(11,-18);
\draw(22,2)--(22,-10)--(24.5,-14)--(24.5,-18);
\draw(22.5,2) node[above]{$\xi^{(i)}_{j+1}$};
\draw(23,2)--(23,-10)--(25.5,-14)--(25.5,-18);
\draw(24,2) node[above]{$\cdots$};
\draw(25,2)--(25,-10)--(16.5,-14)--(12.3,-18);
\draw(25.5,2) node[above]{$\mu^{(l)}_n$};
\draw(26,2)--(26,-10)--(17.5,-14)--(13.3,-18);
\draw(27,2)--(27,-18);
\draw(27.5,2) node[above]{$\xi^{(l)}_n$};
\draw(28,2)--(28,-18);
\draw[color=white, opacity=1] (-6.2,-2)--(28.2,-2);
\draw[color=white, opacity=1] (-6.2,-6)--(28.2,-6);
\draw[color=white, opacity=1] (-6.2,-10)--(28.2,-10);
\draw[color=white, opacity=1] (-6.2,-14)--(28.2,-14);
\draw(-1,-2) node{$\cdots$};
\draw(-1,-6) node{$\cdots$};
\draw(-1,-10) node{$\cdots$};
\draw(24,-2) node{$\cdots$};
\draw(24,-6) node{$\cdots$};
\draw(24,-10) node{$\cdots$};
\draw(-3.2,-14) node{$\cdots$};
\draw(-3.2,-18) node{$\cdots$};
\draw(26.3,-14) node{$\cdots$};
\draw(26.3,-18) node{$\cdots$};
\draw(-5.1,0) node[left]{$w^{\TTT^{A_{\blam}}}$};
\draw(-5.2,-4) node[left]{$\omega_2$};
\draw(-5.2,-8) node[left]{$\omega_1$};
\draw(-5.2,-12) node[left]{$w^L_{i,j}w^{R'}_{i,j}$};
\draw(-5.2,-16) node[left]{$w^D_{i,j}$};
\end{braid}.
\end{align}
and pulling the \(\mu^{A,2}\), \(\mu^{A,3}\) strands to the left gives us the left side of (\ref{movega}):
\begin{align*}
\begin{braid}\tikzset{baseline=-15mm}
\draw [color=blue!50, fill=blue!50, opacity=0.5] (-5,2)--(-5,-10)--(-4,-10)--(-4,2)--cycle;
\draw [color=red!50, fill=red!50, opacity=0.5] (-3,2)--(14,-2)--(14,-10)--(15,-10)--(15,-2)--(-2,2)--cycle;
\draw [color=blue!50, fill=blue!50, opacity=0.5] (0,2)--(-2.5,-2)--(-2.5,-10)--(-1.5,-10)--(-1.5,-2)--(1,2)--cycle;
\draw [color=red!50, fill=red!50, opacity=0.5] (2,2)--(16.5,-2)--(16.5,-10)--(17.5,-10)--(17.5,-2)--(3,2)--cycle;
\draw [color=blue!50, fill=blue!50, opacity=0.5] (4,2)--(-0.5,-2)--(-0.5,-10)--(0.5,-10)--(0.5,-2)--(5,2)--cycle;
\draw [color=blue!50, fill=blue!50, opacity=0.5] (6,2)--(1.5,-2)--(3.5,-6)--(5.5,-10)--(6.5,-10)--(4.5,-6)--(2.5,-2)--(7,2)--cycle;
\draw [color=blue!50, fill=blue!50, opacity=0.5] (8,2)--(3.5,-2)--(7.5,-6)--(7.5,-10)--(8.5,-10)--(8.5,-6)--(4.5,-2)--(9,2)--cycle;
\draw [color=red!50, fill=red!50, opacity=0.5] (10,2)--(18.5,-2)--(18.5,-10)--(19.5,-10)--(19.5,-2)--(11,2)--cycle;
\draw [color=red!50, fill=red!50, opacity=0.5] (12,2)--(20.5,-2)--(20.5,-10)--(21.5,-10)--(21.5,-2)--(13,2)--cycle;
\draw [color=red!50, fill=red!50, opacity=0.5] (14,2)--(22.5,-2)--(22.5,-10)--(23.5,-10)--(23.5,-2)--(15,2)--cycle;
\draw [color=blue!50, fill=blue!50, opacity=0.5] (16,2)--(5.5,-2)--(1.5,-6)--(1.5,-10)--(2.5,-10)--(2.5,-6)--(6.5,-2)--(17,2)--cycle;
\draw [color=blue!50, fill=blue!50, opacity=0.5] (18,2)--(7.5,-2)--(5.5,-6)--(3.5,-10)--(4.5,-10)--(6.5,-6)--(8.5,-2)--(19,2)--cycle;
\draw [color=blue!50, fill=blue!50, opacity=0.5] (20,2)--(9.5,-2)--(9.5,-10)--(10.5,-10)--(10.5,-2)--(21,2)--cycle;
\draw [color=red!50, fill=red!50, opacity=0.5] (22,2)--(24.5,-2)--(24.5,-10)--(25.5,-10)--(25.5,-2)--(23,2)--cycle;
\draw [color=blue!50, fill=blue!50, opacity=0.5] (25,2)--(12,-2)--(12,-10)--(13,-10)--(13,-2)--(26,2)--cycle;
\draw [color=red!50, fill=red!50, opacity=0.5] (27,2)--(27,-10)--(28,-10)--(28,2)--cycle;
\draw(-5,2)--(-5,-10);
\draw(-4.5,2) node[above]{$\mu^{(1)}_1$};
\draw(-4,2)--(-4,-10);
\draw(-3,2)--(14,-2)--(14,-10);
\draw(-2.5,2) node[above]{$\xi^{(1)}_1$};
\draw(-2,2)--(15,-2)--(15,-10);
\draw(-1,2) node[above]{$\cdots$};
\draw(0,2)--(-2.5,-2)--(-2.5,-10);
\draw(0.5,2) node[above]{$\mu^{(i)}_{j-1}$};
\draw(1,2)--(-1.5,-2)--(-1.5,-10);
\draw(2,2)--(16.5,-2)--(16.5,-10);
\draw(2.5,2) node[above]{$\xi^{(i)}_{j-1}$};
\draw(3,2)--(17.5,-2)--(17.5,-10);
\draw(4,2)--(-0.5,-2)--(-0.5,-10);
\draw(4.5,2) node[above]{$\mu^{A,1}$};
\draw(5,2)--(0.5,-2)--(0.5,-10);
\draw(6,2)--(1.5,-2)--(3.5,-6)--(5.5,-10);
\draw(6.5,2) node[above]{$\mu^{A,2}$};
\draw(7,2)--(2.5,-2)--(4.5,-6)--(6.5,-10);
\draw(8,2)--(3.5,-2)--(7.5,-6)--(7.5,-10);
\draw(8.5,2) node[above]{$\mu^{A,3}$};
\draw(9,2)--(4.5,-2)--(8.5,-6)--(8.5,-10);
\draw(10,2)--(18.5,-2)--(18.5,-10);
\draw(10.5,2) node[above]{$\xi^{A,1}$};
\draw(11,2)--(19.5,-2)--(19.5,-10);
\draw(12,2)--(20.5,-2)--(20.5,-10);
\draw(12.5,2) node[above]{$\xi^{A,2}$};
\draw(13,2)--(21.5,-2)--(21.5,-10);
\draw(14,2)--(22.5,-2)--(22.5,-10);
\draw(14.5,2) node[above]{$\xi^{A,3}$};
\draw(15,2)--(23.5,-2)--(23.5,-10);
\draw(16,2)--(5.5,-2)--(1.5,-6)--(1.5,-10);
\draw(16.5,2) node[above]{$\mu_{A,1}$};
\draw(17,2)--(6.5,-2)--(2.5,-6)--(2.5,-10);
\draw(18,2)--(7.5,-2)--(5.5,-6)--(3.5,-10);
\draw(18.5,2) node[above]{$\mu_{A,2}$};
\draw(19,2)--(8.5,-2)--(6.5,-6)--(4.5,-10);
\draw(20,2)--(9.5,-2)--(9.5,-10);
\draw(20.5,2) node[above]{$\mu_{A,3}$};
\draw(21,2)--(10.5,-2)--(10.5,-10);
\draw(22,2)--(24.5,-2)--(24.5,-10);
\draw(22.5,2) node[above]{$\xi^{(i)}_{j+1}$};
\draw(23,2)--(25.5,-2)--(25.5,-10);
\draw(24,2) node[above]{$\cdots$};
\draw(25,2)--(12,-2)--(12,-10);
\draw(25.5,2) node[above]{$\mu^{(l)}_n$};
\draw(26,2)--(13,-2)--(13,-10);
\draw(27,2)--(27,-10);
\draw(27.5,2) node[above]{$\xi^{(l)}_n$};
\draw(28,2)--(28,-10);
\draw[color=white, opacity=1] (-6.2,-2)--(28.2,-2);
\draw[color=white, opacity=1] (-6.2,-6)--(28.2,-6);
\draw[color=white, opacity=1] (-6.2,-10)--(28.2,-10);
\draw(-3.2,-2) node{$\cdots$};
\draw(-3.2,-6) node{$\cdots$};
\draw(-3.2,-10) node{$\cdots$};
\draw(11.2,-6) node{$\cdots$};
\draw(11.2,-10) node{$\cdots$};
\draw(26.3,-2) node{$\cdots$};
\draw(26.3,-6) node{$\cdots$};
\draw(26.3,-10) node{$\cdots$};
\draw(15.8,-6) node{$\cdots$};
\draw(15.8,-10) node{$\cdots$};
\draw(-5.1,0) node[left]{$w^{\bmu\bxi}$};
\draw(-5.2,-4) node[left]{$w^{\TTT^{A_{\bmu}}}$};
\draw(-5.2,-8) node[left]{$\zeta$};
\end{braid}.
\end{align*}
Let \(\omega_1=w^{A_{\blam}}_{r^1_1} \cdots w^{A_{\blam}}_{r^1_{n_1}}\) and \(\omega_2=w^{A_{\blam}}_{r^2_1} \cdots w^{A_{\blam}}_{r^2_{n_2}}\) be reduced words for \(\omega_1\) and \(\omega_2\) in \(\mathfrak{S}^{A_{\blam}}\).
Now consider
\begin{align}\label{wform}
w=(w_{r_1^1}^{A_{\blam}})^{\epsilon_1^1} \cdots (w_{r_{n_1}^1}^{A_{\blam}})^{\epsilon_{n_1}^1} (w_{r_1^2}^{A_{\blam}})^{\epsilon_1^2} \cdots (w_{r_{n_2}^2}^{A_{\blam}})^{\epsilon_{n_2}^2} \in \mathfrak{S}^{A_{\blam}},
\end{align}
where each \(\epsilon^h_k \in \{0,1\}\). In other words, \(w\) is achieved by deleting simple transpositions in \(\mathfrak{S}^{A_\lambda}\) from \(\omega\). 
\begin{lem}\label{tausurvive}
If \(\epsilon^2_k=0\) for some \(1 \leq k \leq n_2\), then
\begin{align*}
\psi^D_{i,j} \psi^L_{i,j} \psi^{R'}_{i,j} (\sigma_{r_1^1}^{A_{\blam}})^{\epsilon_1^1} \cdots (\sigma_{r_{n_1}^1}^{A_{\blam}})^{\epsilon_{n_1}^1}(\sigma_{r_1^2}^{A_{\blam}})^{\epsilon_1^2} \cdots (\sigma_{r_{n_2}^2}^{A_{\blam}})^{\epsilon_{n_2}^2}\psi^{\TTT^{A_{\blam}}}x^{\blam}=0.
\end{align*}
\end{lem}
\begin{proof}
By Lemma \ref{rewrite}, 
\begin{align*}
\psi^D_{i,j} \psi^L_{i,j} \psi^{R'}_{i,j} (\sigma_{r_1^1}^{A_{\blam}})^{\epsilon_1^1} \cdots (\sigma_{r_{n_1}^1}^{A_{\blam}})^{\epsilon_{n_1}^1}(\sigma_{r_1^2}^{A_{\blam}})^{\epsilon_1^2} \cdots (\sigma_{r_{n_2}^2}^{A_{\blam}})^{\epsilon_{n_2}^2}\psi^{\TTT^{A_{\blam}}}v^{\blam}
\end{align*}
is an \(\mathcal{O}\)-linear combination of elements of the form \(v^\TTT\), where a reduced expression for \(w^\TTT\) appears as a subexpression in the (not necessarily reduced) concatenation of reduced expressions associated with
\begin{align*}
w_{i,j}^Dw_{i,j}^Lw_{i,j}^{R'}(w_{r_1^1}^{A_{\blam}})^{\epsilon_1^1} \cdots (w_{r_{n_1}^1}^{A_{\blam}})^{\epsilon_{n_1}^1} (w_{r_1^2}^{A_{\blam}})^{\epsilon_1^2} \cdots (w_{r_{n_2}^2}^{A_{\blam}})^{\epsilon_{n_2}^2}w^{\TTT^{A_{\blam}}}.
\end{align*}
In other words, one can write \(w^\TTT\) by removing crossings in (\ref{bigdiagram}), and in particular (since \(\epsilon^2_k = 0\) for some \(k\)), removing at least one of the pink/blue crossings in the second row. In any case then, there is some pink strand that \(w^\TTT\) sends to the left side, i.e., some \(c\leq a\) such that \((w^\TTT)^{-1}(c) \in \xi^{(h)}_k\) for some \(h,k\). Then \(\textup{sh}(\TTT_{\leq a}) \neq \bmu\). But since \(w^\TTT\) is obtained by removing crossings in \(w^{(\GGG^{A_{\bmu}})\ttt^{\bxi}}\), we have \(\TTT \trianglelefteq \GGG^{A_{\bmu}}\ttt^{\bxi}\). If \(s_r\) is the transposition such that \(s_r\GGG^{A_{\bmu}}\ttt^{\bxi} \in \St(\blam)\), then Lemma \ref{srt} implies that \(\TTT \trianglelefteq s_r\GGG^{A_{\bmu}}\ttt^{\bxi} \in B_{\bmu}\), which in turn implies by Lemma \ref{lowerin} that \(\TTT \in B_{\bmu}\). But then \(\TTT \in B_{\triangleright\bmu}\), and thus \(x^\TTT = 0 \in S^\lambda/C_{\triangleright \bmu}\).
\end{proof}
Every \(w  \in \mathscr{D}^{A_{\blam}}\) can be written as a reduced expression of the form (\ref{wform}) for some \(\epsilon_k^h \in \{0,1\}\). If \(\epsilon^2_k = 0\) for some \(1 \leq k \leq n_2\), or equivalently, if there is some node \((a,b,m)\) in \(\bxi\) such that \(w\TTT^{A_{\blam}}(a,b,m) \neq \GGG^{A_{\blam}}(a,b,m)\), then the above lemma implies that
\begin{align*}
\psi^D_{i,j}\psi^L_{i,j}\psi^{R'}_{i,j}\tau_w^{A_{\blam}}\psi^{\TTT^{A_{\blam}}}x^{\blam} = 0,
\end{align*}
and 
\begin{align*}
\psi^D_{i,j}\psi^L_{i,j}\psi^{R'}_{i,j}\tau_w^{A_{\blam}}\psi^{\TTT^{A_{\blam}}}x^{\blam}=
\psi^D_{i,j}\psi^L_{i,j}\psi^{R'}_{i,j}(\sigma_{r_1^1}^{A_{\blam}})^{\epsilon^1_1} \cdots (\sigma_{r^1_{n_1}}^{A_{\blam}})^{\epsilon^1_{n_1}} \psi_{\omega_2} \psi^{\TTT^{A_{\blam}}}x^{\blam}
\end{align*}
otherwise. Let \(f^{A_{\blam}}\) and \(f^{A_{\bmu}}\) denote the number of bricks in the top row of \({\bf B}^{A_{\blam}}\) and \({\bf B}^{A_{\bmu}}\) respectively. Note that \(w=(w_{r_1^1}^{A_{\blam}})^{\epsilon_1^1} \cdots (w_{r_{n_1}^1}^{A_{\blam}})^{\epsilon_{n_1}^1}\omega_2\) is a reduced expression for an element in \(\mathscr{D}^{A_{\blam}}\) if and only if \((w_{r_1^1}^{A_{\blam}})^{\epsilon_1^1} \cdots (w_{r_{n_1}^1}^{A_{\blam}})^{\epsilon_{n_1}^1}\) is a reduced expression for an element in \(\mathscr{D}^{f^{A_{\bmu}},k^{A_{\blam}}-f^{A_{\blam}}}\). Since \(k^{A_{\bmu}} = k^{A_{\blam}}-(f^{A_{\blam}}-f^{A_{\bmu}})\), this allows us to associate \(\mathscr{D}^{A_{\bmu}}\) with \(\mathscr{D}^{A_{\blam}}\) in the following way. Let \(\widehat{\mathscr{D}}^{A_{\blam}}\) be the set of all \(w \in \mathscr{D}^{A_{\blam}}\) such that \(\epsilon_k^2 \neq 0\) for all \(k\). Then there is a bijection between \(\mathscr{D}^{A_{\bmu}}\) and \(\widehat{\mathscr{D}}^{A_{\blam}}\) given by
\(
u \mapsto u[d_{i,j}]\omega_2.
\)
\begin{lem}\label{tauspast}
For all \(u\in \mathscr{D}^{A_{\bmu}}\), 
\begin{align*}
\tau_u \psi^{\TTT^{A_{\bmu}}} \psi^{\bmu\bxi} x^{\blam} = \psi^D_{i,j}\psi^L_{i,j}\psi^{L'}_{i,j} \tau_u[d_{i,j}] \psi_{\omega_2} \psi^{\TTT^{A_{\blam}}} x^{\blam}.
\end{align*}
\end{lem}
\begin{proof}
This is easily seen in terms of Khovanov-Lauda diagrams, with \(x^{\blam}\) pictured as being at the top of the diagram. The left side:
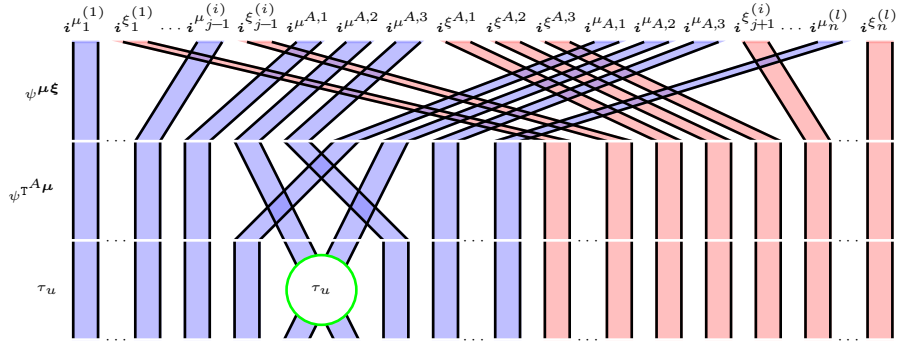
\begin{align*}
\begin{braid}\tikzset{baseline=-15mm}
\draw [color=blue!50, fill=blue!50, opacity=0.5] (-5,2)--(-5,-10)--(-4,-10)--(-4,2)--cycle;
\draw [color=red!50, fill=red!50, opacity=0.5] (-3,2)--(14,-2)--(14,-10)--(15,-10)--(15,-2)--(-2,2)--cycle;
\draw [color=blue!50, fill=blue!50, opacity=0.5] (0,2)--(-2.5,-2)--(-2.5,-10)--(-1.5,-10)--(-1.5,-2)--(1,2)--cycle;
\draw [color=red!50, fill=red!50, opacity=0.5] (2,2)--(16.5,-2)--(16.5,-10)--(17.5,-10)--(17.5,-2)--(3,2)--cycle;
\draw [color=blue!50, fill=blue!50, opacity=0.5] (4,2)--(-0.5,-2)--(-0.5,-10)--(0.5,-10)--(0.5,-2)--(5,2)--cycle;
\draw [color=blue!50, fill=blue!50, opacity=0.5] (6,2)--(1.5,-2)--(3.5,-6)--(5.5,-10)--(6.5,-10)--(4.5,-6)--(2.5,-2)--(7,2)--cycle;
\draw [color=blue!50, fill=blue!50, opacity=0.5] (8,2)--(3.5,-2)--(7.5,-6)--(7.5,-10)--(8.5,-10)--(8.5,-6)--(4.5,-2)--(9,2)--cycle;
\draw [color=red!50, fill=red!50, opacity=0.5] (10,2)--(18.5,-2)--(18.5,-10)--(19.5,-10)--(19.5,-2)--(11,2)--cycle;
\draw [color=red!50, fill=red!50, opacity=0.5] (12,2)--(20.5,-2)--(20.5,-10)--(21.5,-10)--(21.5,-2)--(13,2)--cycle;
\draw [color=red!50, fill=red!50, opacity=0.5] (14,2)--(22.5,-2)--(22.5,-10)--(23.5,-10)--(23.5,-2)--(15,2)--cycle;
\draw [color=blue!50, fill=blue!50, opacity=0.5] (16,2)--(5.5,-2)--(1.5,-6)--(1.5,-10)--(2.5,-10)--(2.5,-6)--(6.5,-2)--(17,2)--cycle;
\draw [color=blue!50, fill=blue!50, opacity=0.5] (18,2)--(7.5,-2)--(5.5,-6)--(3.5,-10)--(4.5,-10)--(6.5,-6)--(8.5,-2)--(19,2)--cycle;
\draw [color=blue!50, fill=blue!50, opacity=0.5] (20,2)--(9.5,-2)--(9.5,-10)--(10.5,-10)--(10.5,-2)--(21,2)--cycle;
\draw [color=red!50, fill=red!50, opacity=0.5] (22,2)--(24.5,-2)--(24.5,-10)--(25.5,-10)--(25.5,-2)--(23,2)--cycle;
\draw [color=blue!50, fill=blue!50, opacity=0.5] (25,2)--(12,-2)--(12,-10)--(13,-10)--(13,-2)--(26,2)--cycle;
\draw [color=red!50, fill=red!50, opacity=0.5] (27,2)--(27,-10)--(28,-10)--(28,2)--cycle;
\draw(-5,2)--(-5,-10);
\draw(-4.5,2) node[above]{$\bi^{\mu^{(1)}_1}$};
\draw(-4,2)--(-4,-10);
\draw(-3,2)--(14,-2)--(14,-10);
\draw(-2.5,2) node[above]{$\bi^{\xi^{(1)}_1}$};
\draw(-2,2)--(15,-2)--(15,-10);
\draw(-1,2) node[above]{$\cdots$};
\draw(0,2)--(-2.5,-2)--(-2.5,-10);
\draw(0.5,2) node[above]{$\bi^{\mu^{(i)}_{j\hspace{-0.4mm}-\hspace{-0.4mm}1}}$};
\draw(1,2)--(-1.5,-2)--(-1.5,-10);
\draw(2,2)--(16.5,-2)--(16.5,-10);
\draw(2.5,2) node[above]{$\bi^{\xi^{(i)}_{j\hspace{-0.4mm}-\hspace{-0.4mm}1}}$};
\draw(3,2)--(17.5,-2)--(17.5,-10);
\draw(4,2)--(-0.5,-2)--(-0.5,-10);
\draw(4.5,2) node[above]{$\bi^{\mu^{\hspace{-0.2mm}A\hspace{-0.2mm},\hspace{-0.2mm}1}}$};
\draw(5,2)--(0.5,-2)--(0.5,-10);
\draw(6,2)--(1.5,-2)--(3.5,-6)--(5.5,-10);
\draw(6.5,2) node[above]{$\bi^{\mu^{\hspace{-0.2mm}A\hspace{-0.2mm},\hspace{-0.2mm}2}}$};
\draw(7,2)--(2.5,-2)--(4.5,-6)--(6.5,-10);
\draw(8,2)--(3.5,-2)--(7.5,-6)--(7.5,-10);
\draw(8.5,2) node[above]{$\bi^{\mu^{\hspace{-0.2mm}A\hspace{-0.2mm},\hspace{-0.2mm}3}}$};
\draw(9,2)--(4.5,-2)--(8.5,-6)--(8.5,-10);
\draw(10,2)--(18.5,-2)--(18.5,-10);
\draw(10.5,2) node[above]{$\bi^{\xi^{\hspace{-0.2mm}A\hspace{-0.2mm},\hspace{-0.2mm}1}}$};
\draw(11,2)--(19.5,-2)--(19.5,-10);
\draw(12,2)--(20.5,-2)--(20.5,-10);
\draw(12.5,2) node[above]{$\bi^{\xi^{\hspace{-0.2mm}A\hspace{-0.2mm},\hspace{-0.2mm}2}}$};
\draw(13,2)--(21.5,-2)--(21.5,-10);
\draw(14,2)--(22.5,-2)--(22.5,-10);
\draw(14.5,2) node[above]{$\bi^{\xi^{\hspace{-0.2mm}A\hspace{-0.2mm},\hspace{-0.2mm}3}}$};
\draw(15,2)--(23.5,-2)--(23.5,-10);
\draw(16,2)--(5.5,-2)--(1.5,-6)--(1.5,-10);
\draw(16.5,2) node[above]{$\bi^{\mu_{\hspace{-0.2mm}A\hspace{-0.2mm},\hspace{-0.2mm}1}}$};
\draw(17,2)--(6.5,-2)--(2.5,-6)--(2.5,-10);
\draw(18,2)--(7.5,-2)--(5.5,-6)--(3.5,-10);
\draw(18.5,2) node[above]{$\bi^{\mu_{\hspace{-0.2mm}A\hspace{-0.2mm},\hspace{-0.2mm}2}}$};
\draw(19,2)--(8.5,-2)--(6.5,-6)--(4.5,-10);
\draw(20,2)--(9.5,-2)--(9.5,-10);
\draw(20.5,2) node[above]{$\bi^{\mu_{\hspace{-0.2mm}A\hspace{-0.2mm},\hspace{-0.2mm}3}}$};
\draw(21,2)--(10.5,-2)--(10.5,-10);
\draw(22,2)--(24.5,-2)--(24.5,-10);
\draw(22.5,2) node[above]{$\bi^{\xi^{(i)}_{j\hspace{-0.4mm}+\hspace{-0.4mm}1}}$};
\draw(23,2)--(25.5,-2)--(25.5,-10);
\draw(24,2) node[above]{$\cdots$};
\draw(25,2)--(12,-2)--(12,-10);
\draw(25.5,2) node[above]{$\bi^{\mu^{(l)}_n}$};
\draw(26,2)--(13,-2)--(13,-10);
\draw(27,2)--(27,-10);
\draw(27.5,2) node[above]{$\bi^{\xi^{(l)}_n}$};
\draw(28,2)--(28,-10);
\draw[color=white, opacity=1] (-6.2,-2)--(28.2,-2);
\draw[color=white, opacity=1] (-6.2,-6)--(28.2,-6);
\draw[color=white, opacity=1] (-6.2,-10)--(28.2,-10);
\draw(-3.2,-2) node{$\cdots$};
\draw(-3.2,-6) node{$\cdots$};
\draw(-3.2,-10) node{$\cdots$};
\draw(11.2,-6) node{$\cdots$};
\draw(11.2,-10) node{$\cdots$};
\draw(26.3,-2) node{$\cdots$};
\draw(26.3,-6) node{$\cdots$};
\draw(26.3,-10) node{$\cdots$};
\draw(15.8,-6) node{$\cdots$};
\draw(15.8,-10) node{$\cdots$};
\draw(-5.1,0) node[left]{$\psi^{\bmu\bxi}$};
\draw(-5.2,-4) node[left]{$\psi^{\TTT^{A_{\bmu}}}$};
\draw(-5.2,-8) node[left]{$\tau_u$};
\filldraw[color=green, fill=white]  (5,-8) circle (40pt);
\draw(5,-8) node{$\tau_u$};
\end{braid}
\end{align*}
is by Lemma \ref{movepast} equal to
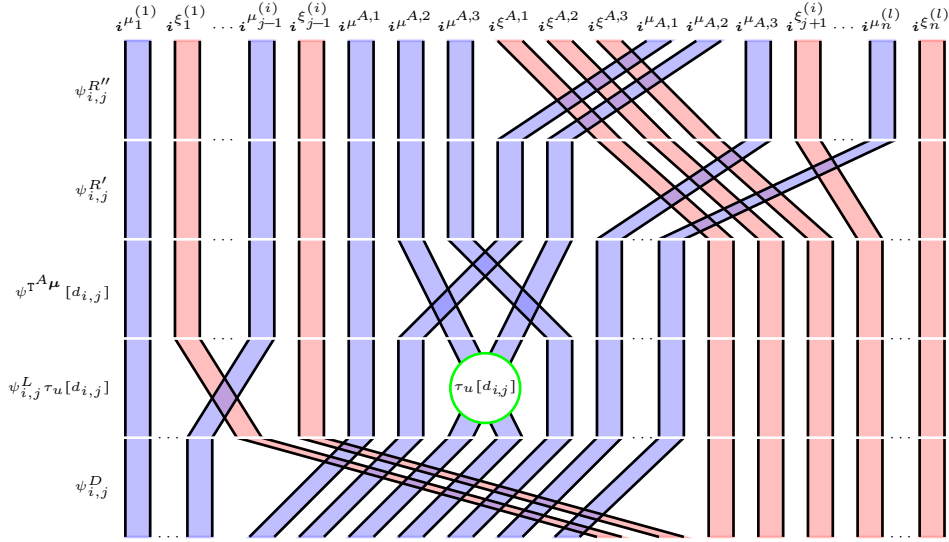
\begin{align*}
\begin{braid}\tikzset{baseline=-25mm}
\draw [color=blue!50, fill=blue!50, opacity=0.5] (-5,2)--(-5,-18)--(-4,-18)--(-4,2)--cycle;
\draw [color=red!50, fill=red!50, opacity=0.5] (-3,2)--(-3,-10)--(-0.5,-14)--(14,-18)--(15,-18)--(0.5,-14)--(-2,-10)--(-2,2)--cycle;
\draw [color=blue!50, fill=blue!50, opacity=0.5] (0,2)--(0,-10)--(-2.5,-14)--(-2.5,-18)--(-1.5,-18)--(-1.5,-14)--(1,-10)--(1,2)--cycle;
\draw [color=red!50, fill=red!50, opacity=0.5] (2,2)--(2,-14)--(16.5,-18)--(17.5,-18)--(3,-14)--(3,2)--cycle;
\draw [color=blue!50, fill=blue!50, opacity=0.5] (4,2)--(4,-14)--(0,-18)--(1,-18)--(5,-14)--(5,2)--cycle;
\draw [color=blue!50, fill=blue!50, opacity=0.5] (6,2)--(6,-2)--(6,-6)--(8,-10)--(10,-14)--(6,-18)--(7,-18)--(11,-14)--(9,-10)--(7,-6)--(7,-2)--(7,2)--cycle;
\draw [color=blue!50, fill=blue!50, opacity=0.5] (8,2)--(8,-2)--(8,-6)--(12,-10)--(12,-14)--(8,-18)--(9,-18)--(13,-14)--(13,-10)--(9,-6)--(9,2)--cycle;
\draw [color=red!50, fill=red!50, opacity=0.5] (10,2)--(14,-2)--(18.5,-6)--(18.5,-18)--(19.5,-18)--(19.5,-6)--(15,-2)--(11,2)--cycle;
\draw [color=red!50, fill=red!50, opacity=0.5] (12,2)--(16,-2)--(20.5,-6)--(20.5,-18)--(21.5,-18)--(21.5,-6)--(17,-2)--(13,2)--cycle;
\draw [color=red!50, fill=red!50, opacity=0.5] (14,2)--(18,-2)--(22.5,-6)--(22.5,-18)--(23.5,-18)--(23.5,-6)--(19,-2)--(15,2)--cycle;
\draw [color=blue!50, fill=blue!50, opacity=0.5] (16,2)--(10,-2)--(10,-6)--(6,-10)--(6,-14)--(2,-18)--(3,-18)--(7,-14)--(7,-10)--(11,-6)--(11,-2)--(17,2)--cycle;
\draw [color=blue!50, fill=blue!50, opacity=0.5] (18,2)--(12,-2)--(12,-6)--(10,-10)--(8,-14)--(4,-18)--(5,-18)--(9,-14)--(11,-10)--(13,-6)--(13,-2)--(19,2)--cycle;
\draw [color=blue!50, fill=blue!50, opacity=0.5] (20,2)--(20,-2)--(14,-6)--(14,-14)--(10,-18)--(11,-18)--(15,-14)--(15,-6)--(21,-2)--(21,2)--cycle;
\draw [color=red!50, fill=red!50, opacity=0.5] (22,2)--(22,-2)--(24.5,-6)--(24.5,-18)--(25.5,-18)--(25.5,-6)--(23,-2)--(23,2)--cycle;
\draw [color=blue!50, fill=blue!50, opacity=0.5] (25,2)--(25,-2)--(16.5,-6)--(16.5,-14)--(12.3,-18)--(13.3,-18)--(17.5,-14)--(17.5,-6)--(26,-2)--(26,2)--cycle;
\draw [color=red!50, fill=red!50, opacity=0.5] (27,2)--(27,-18)--(28,-18)--(28,2)--cycle;
\draw(-5,2)--(-5,-18);
\draw(-4.5,2) node[above]{$\bi^{\mu^{(1)}_1}$};
\draw(-4,2)--(-4,-18);
\draw(-3,2)--(-3,-10)--(-0.5,-14)--(14,-18);
\draw(-2.5,2) node[above]{$\bi^{\xi^{(1)}_1}$};
\draw(-2,2)--(-2,-10)--(0.5,-14)--(15,-18);
\draw(-1,2) node[above]{$\cdots$};
\draw(0,2)--(0,-10)--(-2.5,-14)--(-2.5,-18);
\draw(0.5,2) node[above]{$\bi^{\mu^{(i)}_{j\hspace{-0.4mm}-\hspace{-0.4mm}1}}$};
\draw(1,2)--(1,-10)--(-1.5,-14)--(-1.5,-18);
\draw(2,2)--(2,-14)--(16.5,-18);
\draw(2.5,2) node[above]{$\bi^{\xi^{(i)}_{j\hspace{-0.4mm}-\hspace{-0.4mm}1}}$};
\draw(3,2)--(3,-14)--(17.5,-18);
\draw(4,2)--(4,-14)--(0,-18);
\draw(4.5,2) node[above]{$\bi^{\mu^{\hspace{-0.2mm}A\hspace{-0.2mm},\hspace{-0.2mm}1}}$};
\draw(5,2)--(5,-14)--(1,-18);
\draw(6,2)--(6,-2)--(6,-6)--(8,-10)--(10,-14)--(6,-18);
\draw(6.5,2) node[above]{$\bi^{\mu^{\hspace{-0.2mm}A\hspace{-0.2mm},\hspace{-0.2mm}2}}$};
\draw(7,2)--(7,-2)--(7,-6)--(9,-10)--(11,-14)--(7,-18);
\draw(8,2)--(8,-2)--(8,-6)--(12,-10)--(12,-14)--(8,-18);
\draw(8.5,2) node[above]{$\bi^{\mu^{\hspace{-0.2mm}A\hspace{-0.2mm},\hspace{-0.2mm}3}}$};
\draw(9,2)--(9,-2)--(9,-6)--(13,-10)--(13,-14)--(9,-18);
\draw(10,2)--(14,-2)--(18.5,-6)--(18.5,-18);
\draw(10.5,2) node[above]{$\bi^{\xi^{\hspace{-0.2mm}A\hspace{-0.2mm},\hspace{-0.2mm}1}}$};
\draw(11,2)--(15,-2)--(19.5,-6)--(19.5,-18);
\draw(12,2)--(16,-2)--(20.5,-6)--(20.5,-18);
\draw(12.5,2) node[above]{$\bi^{\xi^{\hspace{-0.2mm}A\hspace{-0.2mm},\hspace{-0.2mm}2}}$};
\draw(13,2)--(17,-2)--(21.5,-6)--(21.5,-18);
\draw(14,2)--(18,-2)--(22.5,-6)--(22.5,-18);
\draw(14.5,2) node[above]{$\bi^{\xi^{\hspace{-0.2mm}A\hspace{-0.2mm},\hspace{-0.2mm}3}}$};
\draw(15,2)--(19,-2)--(23.5,-6)--(23.5,-18);
\draw(16,2)--(10,-2)--(10,-6)--(6,-10)--(6,-14)--(2,-18);
\draw(16.5,2) node[above]{$\bi^{\mu_{\hspace{-0.2mm}A\hspace{-0.2mm},\hspace{-0.2mm}1}}$};
\draw(17,2)--(11,-2)--(11,-6)--(7,-10)--(7,-14)--(3,-18);
\draw(18,2)--(12,-2)--(12,-6)--(10,-10)--(8,-14)--(4,-18);
\draw(18.5,2) node[above]{$\bi^{\mu_{\hspace{-0.2mm}A\hspace{-0.2mm},\hspace{-0.2mm}2}}$};
\draw(19,2)--(13,-2)--(13,-6)--(11,-10)--(9,-14)--(5,-18);
\draw(20,2)--(20,-2)--(14,-6)--(14,-14)--(10,-18);
\draw(20.5,2) node[above]{$\bi^{\mu_{\hspace{-0.2mm}A\hspace{-0.2mm},\hspace{-0.2mm}3}}$};
\draw(21,2)--(21,-2)--(15,-6)--(15,-14)--(11,-18);
\draw(22,2)--(22,-2)--(24.5,-6)--(24.5,-18);
\draw(22.5,2) node[above]{$\bi^{\xi^{(i)}_{\hspace{-0.4mm}j+\hspace{-0.4mm}1}}$};
\draw(23,2)--(23,-2)--(25.5,-6)--(25.5,-18);
\draw(24,2) node[above]{$\cdots$};
\draw(25,2)--(25,-2)--(16.5,-6)--(16.5,-14)--(12.3,-18);
\draw(25.5,2) node[above]{$\bi^{\mu^{(l)}_n}$};
\draw(26,2)--(26,-2)--(17.5,-6)--(17.5,-14)--(13.3,-18);
\draw(27,2)--(27,-18);
\draw(27.5,2) node[above]{$\bi^{\xi^{(l)}_n}$};
\draw(28,2)--(28,-18);
\draw[color=white, opacity=1] (-6.2,-2)--(28.2,-2);
\draw[color=white, opacity=1] (-6.2,-6)--(28.2,-6);
\draw[color=white, opacity=1] (-6.2,-10)--(28.2,-10);
\draw[color=white, opacity=1] (-6.2,-14)--(28.2,-14);
\draw(-1,-2) node{$\cdots$};
\draw(-1,-6) node{$\cdots$};
\draw(-1,-10) node{$\cdots$};
\draw(24,-2) node{$\cdots$};
\draw(-3.2,-14) node{$\cdots$};
\draw(-3.2,-18) node{$\cdots$};
\draw(15.9,-6) node{$\cdots$};
\draw(15.9,-10) node{$\cdots$};
\draw(15.9,-14) node{$\cdots$};
\draw(26.3,-6) node{$\cdots$};
\draw(26.3,-10) node{$\cdots$};
\draw(26.3,-14) node{$\cdots$};
\draw(26.3,-18) node{$\cdots$};
\draw(-5.1,0) node[left]{$\psi^{R''}_{i,j}$};
\draw(-5.2,-4) node[left]{$\psi^{R'}_{i,j}$};
\draw(-5.2,-8) node[left]{$\psi^{\TTT^{A_{\bmu}}}[d_{i,j}]$};
\draw(-5.2,-12) node[left]{$\psi^L_{i,j}\tau_u[d_{i,j}]$};
\draw(-5.2,-16) node[left]{$\psi^D_{i,j}$};
\filldraw[color=green, fill=white]  (9.5,-12) circle (40pt);
\draw(9.5,-12) node {$\tau_u[d_{i\hspace{-0.3mm},\hspace{-0.3mm}j}]$};
\end{braid}
\end{align*}
which, after an isotopy of strands, becomes the right side in the lemma statement:
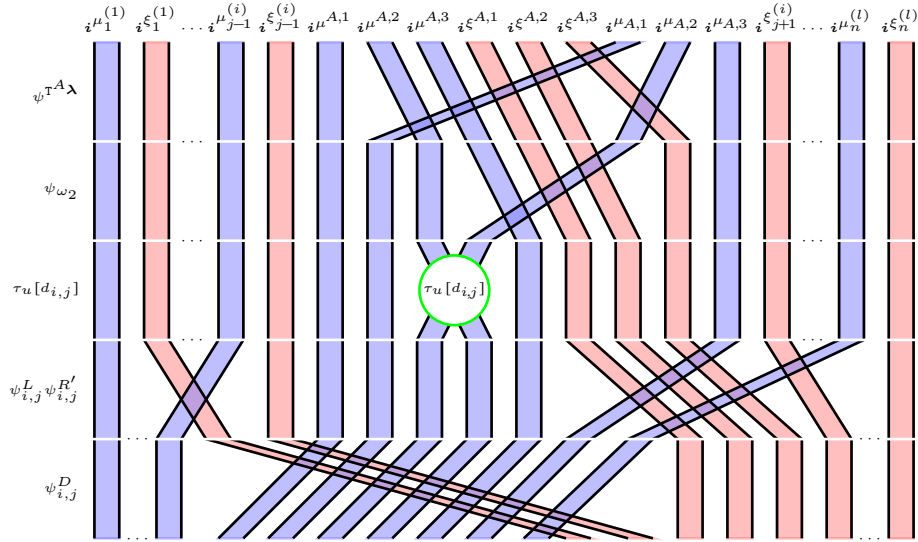
\begin{align*}
\begin{braid}\tikzset{baseline=-20mm}
\draw [color=blue!50, fill=blue!50, opacity=0.5] (-5,2)--(-5,-18)--(-4,-18)--(-4,2)--cycle;
\draw [color=red!50, fill=red!50, opacity=0.5] (-3,2)--(-3,-10)--(-0.5,-14)--(14,-18)--(15,-18)--(0.5,-14)--(-2,-10)--(-2,2)--cycle;
\draw [color=blue!50, fill=blue!50, opacity=0.5] (0,2)--(0,-10)--(-2.5,-14)--(-2.5,-18)--(-1.5,-18)--(-1.5,-14)--(1,-10)--(1,2)--cycle;
\draw [color=red!50, fill=red!50, opacity=0.5] (2,2)--(2,-14)--(16.5,-18)--(17.5,-18)--(3,-14)--(3,2)--cycle;
\draw [color=blue!50, fill=blue!50, opacity=0.5] (4,2)--(4,-14)--(0,-18)--(1,-18)--(5,-14)--(5,2)--cycle;
\draw [color=blue!50, fill=blue!50, opacity=0.5] (6,2)--(8,-2)--(8,-6)--(10,-10)--(10,-14)--(6,-18)--(7,-18)--(11,-14)--(11,-10)--(9,-6)--(9,-2)--(7,2)--cycle;
\draw [color=blue!50, fill=blue!50, opacity=0.5] (8,2)--(10,-2)--(12,-6)--(12,-10)--(12,-14)--(8,-18)--(9,-18)--(13,-14)--(13,-10)--(13,-6)--(11,-2)--(9,2)--cycle;
\draw [color=red!50, fill=red!50, opacity=0.5] (10,2)--(12,-2)--(14,-6)--(14,-10)--(18.5,-14)--(18.5,-18)--(19.5,-18)--(19.5,-14)--(15,-10)--(15,-6)--(13,-2)--(11,2)--cycle;
\draw [color=red!50, fill=red!50, opacity=0.5] (12,2)--(14,-2)--(16,-6)--(16,-10)--(20.5,-14)--(20.5,-18)--(21.5,-18)--(21.5,-14)--(17,-10)--(17,-6)--(15,-2)--(13,2)--cycle;
\draw [color=red!50, fill=red!50, opacity=0.5] (14,2)--(18,-2)--(18,-10)--(22.5,-14)--(22.5,-18)--(23.5,-18)--(23.5,-14)--(19,-10)--(19,-2)--(15,2)--cycle;
\draw [color=blue!50, fill=blue!50, opacity=0.5] (16,2)--(6,-2)--(6,-10)--(6,-14)--(2,-18)--(3,-18)--(7,-14)--(7,-10)--(7,-2)--(17,2)--cycle;
\draw [color=blue!50, fill=blue!50, opacity=0.5] (18,2)--(16,-2)--(10,-6)--(8,-10)--(8,-14)--(4,-18)--(5,-18)--(9,-14)--(9,-10)--(11,-6)--(17,-2)--(19,2)--cycle;
\draw [color=blue!50, fill=blue!50, opacity=0.5] (20,2)--(20,-10)--(14,-14)--(10,-18)--(11,-18)--(15,-14)--(21,-10)--(21,2)--cycle;
\draw [color=red!50, fill=red!50, opacity=0.5] (22,2)--(22,-10)--(24.5,-14)--(24.5,-18)--(25.5,-18)--(25.5,-14)--(23,-10)--(23,2)--cycle;
\draw [color=blue!50, fill=blue!50, opacity=0.5] (25,2)--(25,-10)--(16.5,-14)--(12.3,-18)--(13.3,-18)--(17.5,-14)--(26,-10)--(26,2)--cycle;
\draw [color=red!50, fill=red!50, opacity=0.5] (27,2)--(27,-18)--(28,-18)--(28,2)--cycle;
\draw(-5,2)--(-5,-18);
\draw(-4.5,2) node[above]{$\bi^{\mu^{(1)}_1}$};
\draw(-4,2)--(-4,-18);
\draw(-3,2)--(-3,-10)--(-0.5,-14)--(14,-18);
\draw(-2.5,2) node[above]{$\bi^{\xi^{(1)}_1}$};
\draw(-2,2)--(-2,-10)--(0.5,-14)--(15,-18);
\draw(-1,2) node[above]{$\cdots$};
\draw(0,2)--(0,-10)--(-2.5,-14)--(-2.5,-18);
\draw(0.5,2) node[above]{$\bi^{\mu^{(i)}_{j\hspace{-0.4mm}-\hspace{-0.4mm}1}}$};
\draw(1,2)--(1,-10)--(-1.5,-14)--(-1.5,-18);
\draw(2,2)--(2,-14)--(16.5,-18);
\draw(2.5,2) node[above]{$\bi^{\xi^{(i)}_{j\hspace{-0.4mm}-\hspace{-0.4mm}1}}$};
\draw(3,2)--(3,-14)--(17.5,-18);
\draw(4,2)--(4,-14)--(0,-18);
\draw(4.5,2) node[above]{$\bi^{\mu^{\hspace{-0.2mm}A\hspace{-0.2mm},\hspace{-0.2mm}1}}$};
\draw(5,2)--(5,-14)--(1,-18);
\draw(6,2)--(8,-2)--(8,-6)--(10,-10)--(10,-14)--(6,-18);
\draw(6.5,2) node[above]{$\bi^{\mu^{\hspace{-0.2mm}A\hspace{-0.2mm},\hspace{-0.2mm}2}}$};
\draw(7,2)--(9,-2)--(9,-6)--(11,-10)--(11,-14)--(7,-18);
\draw(8,2)--(10,-2)--(12,-6)--(12,-10)--(12,-14)--(8,-18);
\draw(8.5,2) node[above]{$\bi^{\mu^{\hspace{-0.2mm}A\hspace{-0.2mm},\hspace{-0.2mm}3}}$};
\draw(9,2)--(11,-2)--(13,-6)--(13,-10)--(13,-14)--(9,-18);
\draw(10,2)--(12,-2)--(14,-6)--(14,-10)--(18.5,-14)--(18.5,-18);
\draw(10.5,2) node[above]{$\bi^{\xi^{\hspace{-0.2mm}A\hspace{-0.2mm},\hspace{-0.2mm}1}}$};
\draw(11,2)--(13,-2)--(15,-6)--(15,-10)--(19.5,-14)--(19.5,-18);
\draw(12,2)--(14,-2)--(16,-6)--(16,-10)--(20.5,-14)--(20.5,-18);
\draw(12.5,2) node[above]{$\bi^{\xi^{\hspace{-0.2mm}A\hspace{-0.2mm},\hspace{-0.2mm}2}}$};
\draw(13,2)--(15,-2)--(17,-6)--(17,-10)--(21.5,-14)--(21.5,-18);
\draw(14,2)--(18,-2)--(18,-10)--(22.5,-14)--(22.5,-18);
\draw(14.5,2) node[above]{$\bi^{\xi^{\hspace{-0.2mm}A\hspace{-0.2mm},\hspace{-0.2mm}3}}$};
\draw(15,2)--(19,-2)--(19,-10)--(23.5,-14)--(23.5,-18);
\draw(16,2)--(6,-2)--(6,-10)--(6,-14)--(2,-18);
\draw(16.5,2) node[above]{$\bi^{\mu_{\hspace{-0.2mm}A\hspace{-0.2mm},\hspace{-0.2mm}1}}$};
\draw(17,2)--(7,-2)--(7,-10)--(7,-14)--(3,-18);
\draw(18,2)--(16,-2)--(10,-6)--(8,-10)--(8,-14)--(4,-18);
\draw(18.5,2) node[above]{$\bi^{\mu_{\hspace{-0.2mm}A\hspace{-0.2mm},\hspace{-0.2mm}2}}$};
\draw(19,2)--(17,-2)--(11,-6)--(9,-10)--(9,-14)--(5,-18);
\draw(20,2)--(20,-10)--(14,-14)--(10,-18);
\draw(20.5,2) node[above]{$\bi^{\mu_{\hspace{-0.2mm}A\hspace{-0.2mm},\hspace{-0.2mm}3}}$};
\draw(21,2)--(21,-10)--(15,-14)--(11,-18);
\draw(22,2)--(22,-10)--(24.5,-14)--(24.5,-18);
\draw(22.5,2) node[above]{$\bi^{\xi^{(i)}_{j\hspace{-0.4mm}+\hspace{-0.4mm}1}}$};
\draw(23,2)--(23,-10)--(25.5,-14)--(25.5,-18);
\draw(24,2) node[above]{$\cdots$};
\draw(25,2)--(25,-10)--(16.5,-14)--(12.3,-18);
\draw(25.5,2) node[above]{$\bi^{\mu^{(l)}_n}$};
\draw(26,2)--(26,-10)--(17.5,-14)--(13.3,-18);
\draw(27,2)--(27,-18);
\draw(27.5,2) node[above]{$\bi^{\xi^{(l)}_n}$};
\draw(28,2)--(28,-18);
\draw[color=white, opacity=1] (-6.2,-2)--(28.2,-2);
\draw[color=white, opacity=1] (-6.2,-6)--(28.2,-6);
\draw[color=white, opacity=1] (-6.2,-10)--(28.2,-10);
\draw[color=white, opacity=1] (-6.2,-14)--(28.2,-14);
\draw(-1,-2) node{$\cdots$};
\draw(-1,-6) node{$\cdots$};
\draw(-1,-10) node{$\cdots$};
\draw(24,-2) node{$\cdots$};
\draw(24,-6) node{$\cdots$};
\draw(24,-10) node{$\cdots$};
\draw(-3.2,-14) node{$\cdots$};
\draw(-3.2,-18) node{$\cdots$};
\draw(26.3,-14) node{$\cdots$};
\draw(26.3,-18) node{$\cdots$};
\draw(-5.1,0) node[left]{$\psi^{\TTT^{A_{\blam}}}$};
\draw(-5.2,-4) node[left]{$\psi_{\omega_2}$};
\draw(-5.2,-8) node[left]{$\tau_u[d_{i,j}]$};
\draw(-5.2,-12) node[left]{$\psi^L_{i,j}\psi^{R'}_{i,j}$};
\draw(-5.2,-16) node[left]{$\psi^D_{i,j}$};
\filldraw[color=green, fill=white]  (9.5,-8) circle (40pt);
\draw(9.5,-8) node {$\tau_u[d_{i\hspace{-0.3mm},\hspace{-0.3mm}j}]$};
\end{braid}
\end{align*}
completing the proof.
\end{proof}
\begin{lem}\label{fAzero}
Let \(A_{\bmu}\) be a Garnir node of \(\bmu\). Then \(f(g^{A_{\bmu}}\otimes 1_{\beta}) = 0\).
\end{lem}
\begin{proof}
We make use of Lemma \ref{tauspast} and the bijection between \(\mathscr{D}^{A_{\bmu}}\) and \(\widehat{\mathscr{D}}^{A_{\blam}}\):
\begin{align*}
f(g^{A_{\bmu}}\otimes 1_{\beta})&=g^{A_{\bmu}} \cdot x^{\bmu\bxi} = 
\left( \sum_{u \in \mathscr{D}^{A_{\bmu}}} \tau_u \psi^{\TTT^{A_{\bmu}}}\right)\psi^{\bmu\bxi}x^{\blam}\\
&=\sum_{u \in \mathscr{D}^{A_{\bmu}}} \psi^D_{i,j} \psi^L_{i,j} \psi^{R'}_{i,j} \tau_u [d_{i,j}] \psi_{\omega_2} \psi^{\TTT^{A_{\blam}}} x^{\blam}\\
&= \psi^D_{i,j} \psi^L_{i,j} \psi^{R'}_{i,j} \sum_{u \in \mathscr{D}^{A_{\bmu}}} \tau_u [d_{i,j}] \psi_{\omega_2} \psi^{\TTT^{A_{\blam}}} x^{\blam}\\
&= \psi^D_{i,j} \psi^L_{i,j} \psi^{R'}_{i,j} \sum_{w \in \mathscr{D}^{A_{\blam}}} \tau_w \psi^{\TTT^{A_{\blam}}} x^{\blam}\\
&= \psi^D_{i,j} \psi^L_{i,j} \psi^{R'}_{i,j}  g^{A_{\blam}} x^{\blam}\\
&=0.
\end{align*}
\end{proof}
\begin{remark}\label{fxi}
Although we have focused on Garnir nodes in \(\bmu\), there are obvious analogues (whose proofs are entirely analogous) of Lemmas \ref{movepast}, \ref{tausurvive}, and \ref{tauspast}, which imply the analogue of Proposition \ref{fAzero}: 
\begin{align*}
f(1_{\alpha} \otimes g^{A_{\bxi}})=0
\end{align*}
 for Garnir nodes in \(\bxi\).
\end{remark}
\begin{prop}\label{isom}
The map \(f:R_{\alpha,\beta} \langle \deg \TTT^{\bmu} + \deg \ttt^{\bxi} \rangle \to N_{\bmu}\) induces a graded isomorphism 
\begin{align*}
f:S^{\bmu} \boxtimes S^{\bxi} \xrightarrow{\sim} N_{\bmu}.
\end{align*}
\end{prop}
\begin{proof}
We have that \(f\) factors through to a map 
\begin{align*}
R_{\alpha,\beta}/(J_{\alpha}^{\bmu} \otimes R_{\beta}  + R_{\alpha} \otimes J_{\beta}^{\bxi}) \langle \deg \TTT^{\bmu} + \deg \ttt^{\bxi} \rangle \to \hspace{-1mm} N_{\bmu}
\end{align*}
 by Lemmas \ref{kerf}, \ref{fAzero} and Remark \ref{fxi}. However, we also have
 \begin{align*}
 R_{\alpha,\beta}/(J_{\alpha}^{\bmu} \otimes R_{\beta}  + R_{\alpha} \otimes J_{\beta}^{\bxi}) \langle \deg \TTT^{\bmu} + \deg \ttt^{\bxi} \rangle
 \cong R_{\alpha}/J_{\alpha}^{\bmu} \langle \deg \TTT^{\bmu} \rangle \otimes R_{\beta}/J_{\beta}^{\bxi} \langle  \deg \ttt^{\bxi} \rangle = S^{\bmu} \boxtimes S^{\bxi}.
 \end{align*}
 Moreover, for all \(\TTT \in \St(\bmu)\), \(\ttt \in \St(\bxi)\), 
 \begin{align*}
f(v^\TTT \boxtimes v^\ttt) = f(\psi^\TTT v^{\bmu}\boxtimes \psi^\ttt v^{\bxi}) = \psi^\TTT \psi^\ttt[a] x^{\bmu \bxi} = \psi^\TTT \psi^\ttt[a] \psi^{\bmu \bxi} x^{\blam} = x^{\TTT\ttt} + \sum_{\UUU \triangleleft \TTT\ttt}d_\UUU x^\UUU
\end{align*}
for some constants \(d_\UUU\), by Lemma \ref{rewrite}. Since \(\{v^\TTT \boxtimes v^\ttt \mid \TTT \in \St(\bmu), \ttt \in \St(\bxi)\}\) is a spanning set for \(S^{\bmu} \boxtimes S^{\bxi}\) and \(\{x^{\TTT\ttt} \mid \TTT \in \St(\bmu), \ttt \in \St(\bxi)\}\) is a basis for \(N_{\bmu}\), it follows that \(f\) is an isomorphism.
\end{proof}


\subsection{A basis for \texorpdfstring{$S^{\blam/\bmu}$}{} and a filtration for 
\texorpdfstring{$\Res_{\alpha,\beta}S^{\blam}$}{}}
Proposition \ref{isom} in hand, we may now prove two theorems which complete the analogy with the definition (\ref{skewhom}) in the semisimple case, and justify our use of the term {\it skew Specht module} for \(S^{\blam/\bmu}\).
\begin{thm}\label{basisthm}
Let \(\blam/\bmu \in \mathscr{S}^{\kappa}_{\bmu,\alpha}\). Then \(S^{\blam/\bmu}\) has a homogeneous \(\mathcal{O}\)-basis
\begin{align}\label{Sbasis}
\{v^\ttt \mid \ttt \in \St(\blam/\bmu)\}.
\end{align}
\end{thm}
\begin{proof}
By Proposition \ref{spanset}, the set (\ref{Sbasis}) spans \(S^{\blam/\bmu}\) over \(\mathcal{O}\), and the set is linearly independent by Proposition \ref{isom}.
\end{proof}
\begin{thm}\label{filtration}
Let \(\blam \in \mathscr{P}^\kappa_{\alpha + \beta}\). Let 
\(
\{\bmu_1, \ldots \bmu_k\} = \{ \bmu \in \mathscr{P}^\kappa_{\alpha} \mid \bmu \subset \blam\}
\)
and assume the labels are such that \(\bmu_i \triangleright \bmu_j \implies i <j\). Write 
\begin{align*}
V_i : = \sum_{j=1}^i C_{\bmu_j} = \mathcal{O} \left\{ v^\TTT \in S^{\blam} \mid \TTT \in \St(\blam), \textup{sh}(\TTT_{\leq a}) = \bmu_j \textup{ for some }j<i\right\}
\end{align*} 
for all \(i\). Then
\begin{align*}
0=V_0 \leq V_1 \leq V_2 \leq \cdots \leq V_k =  \Res_{\alpha,\beta} S^{\blam}
\end{align*}
is a graded filtration of \(\Res_{\alpha,\beta}S^{\blam}\) by \(R_{\alpha,\beta}\)-submodules, with subquotients
\begin{align*}
V_i/V_{i-1} \cong S^{\bmu_i} \boxtimes S^{\blam/\bmu_i}.
\end{align*}
\end{thm}
\begin{proof} The fact that \(V_k = \Res_{\alpha,\beta}S^{\blam}\) follows from the fact that \(B = \bigcup_{j=1}^k B_{\bmu_j}\) and \(\{v^\TTT \mid \TTT \in B\}\) is a basis for \(\Res_{\alpha,\beta}S^{\blam}\).
Since \(C_{\bmu_i} \geq C_{\bmu_j}\) if \(\bmu_j \trianglerighteq \bmu_i\), we have
\begin{align*}
V_i = \sum_{j=1}^i C_{\bmu_j} = C_{\bmu_i} \oplus \sum_{\substack{j\leq i-1\\\bmu_j \not \triangleright \bmu_i}}C_{\bmu_j}
\end{align*}
and
\begin{align*}
V_{i-1} = \sum_{j=1}^{i-1} C_{\bmu_j} = \sum_{\bmu_j \triangleright \bmu_i}C_{\bmu_j} \oplus \sum_{\substack{j\leq i-1\\\bmu_j \not \triangleright \bmu_i}}C_{\bmu_j} =
C_{\triangleright \bmu_i}\oplus \sum_{\substack{j\leq i-1\\\bmu_j \not \triangleright \bmu_i}}C_{\bmu_j},
\end{align*}
which implies that \(V_i/V_{i-1} \cong C_{\bmu_i}/C_{\triangleright \bmu_i} = N_{\bmu_i} \cong S^{\bmu_i} \boxtimes S^{\blam/\bmu_i}\).
\end{proof}

\begin{remark}\label{cycconnect}
Theorem \ref{filtration} may be compared with \cite[Theorem 3.1]{JP}, which gives a similar result for restrictions of classical Specht modules over the symmetric group algebra to Young subgroups. However, the connection between our skew Specht \(R_\alpha(\mathbb{F})\)-modules \(S^{\lambda/\mu}\) and the skew Specht \(\mathbb{F}\mathfrak{S}_n\)-module \(S_{\mathbb{F}\mathfrak{S}_n}^{\lambda/\mu}\) defined in \cite{JP} is not as direct as may be expected. Taking \(e=\ch\; \mathbb{F}\), it is shown in \cite{bk} that there exists a surjection \(R_n := \bigoplus_{\textup{ht}(\alpha)=n} R_\alpha \twoheadrightarrow R_n^{\Lambda_i} \cong \mathbb{F}\mathfrak{S}_n\).  Inflating \(S^{\lambda/\mu}_{\mathbb{F}\mathfrak{S}_n}\) along this map, we have an \(R_n\)-module \(\textup{infl}\, S^{\lambda/\mu}_{\mathbb{F}\mathfrak{S}_n}\), and truncating \(S^{\lambda/\mu}\) yields an \(\mathbb{F}\mathfrak{S}_n\)-module \(\textup{pr}\,S^{\lambda/\mu}\). However it is not the case that \(\textup{pr}\, S^{\lambda/\mu} \cong S^{\lambda/\mu}_{\mathbb{F}\mathfrak{S}_n}\) nor \(\textup{infl}\, S^{\lambda/\mu}_{\mathbb{F}\mathfrak{S}_n} \cong S^{\lambda/\mu}\) in general, though the (ungraded) dimensions do agree in the latter case, and both statements hold when \(\mu = \varnothing\) and \(\kappa=(i)\).

For an explicit example of this difference, take \(e=\ch\; \mathbb{F} >0\), and \(n>1\). Let 
\begin{align*}
\lambda &= (ne-e+1, ne-2e+2, ne-3e+3, \ldots, n), \\
\mu&= (ne-e, ne-2e+1, ne-3e+2, \ldots, n-1).
\end{align*}
Then \(\lambda/\mu\) consists of \(n\) disconnected nodes of some residue \(i\) (depending on \(\kappa\)). The \(\mathbb{F}\mathfrak{S}_n\)-module \(S_{\mathbb{F}\mathfrak{S}_n}^{\lambda/\mu}\) is isomorphic to the regular module  \({}_{ \mathbb{F}\mathfrak{S}_n}\mathbb{F}\mathfrak{S}_n\), and thus \(\textup{infl}\, S^{\lambda/\mu}_{\mathbb{F}\mathfrak{S}_n}\) is reducible. However, \(S^{\lambda/\mu}\) as defined in this paper is irreducible---in fact it is the unique irreducible \(R_{n\alpha_i}\)-module (up to grading shift), see \cite[\S2.2]{kl}. Moreover, \(\textup{pr}\,S^{\lambda/\mu}=0\) as \(\mathbb{F}\mathfrak{S}_n\) has no \(n!\)-dimensional irreducible modules.
\end{remark}

\subsection{Induction product of skew Specht modules}
The following theorem was proved in \cite[Theorem 8.2]{kmr} in the context of Young diagrams, but the proof is applicable with no significant alteration to the more general case of skew diagrams.
\begin{thm}\label{circprod}
Suppose that \(\blam/\bmu \in \mathscr{S}^\kappa_\alpha\). Then
\begin{align*}
S^{\blam/\bmu} \cong S^{\lambda^{(1)}/\mu^{(1)}} \circ \cdots \circ S^{\lambda^{(l)}/\mu^{(l)}} \langle d_{\blam/\bmu} \rangle,
\end{align*}
as graded \(R_\alpha\)-modules, where
\begin{align*}
d_{\blam/\bmu}= \deg(\ttt^{\blam/\bmu}) - \deg(\ttt^{\lambda^{(1)}/\mu^{(1)}}) - \cdots - \deg(\ttt^{\lambda^{(l)}/\mu^{(l)}}).
\end{align*}
\end{thm}


\section{Joinable diagrams}\label{joinablesec}
\noindent In this section we present a useful, albeit rather technical, result regarding the graded characters of skew Specht modules whose associated component diagrams jibe with each other in a specific sense. This result, together with Theorem \ref{circprod}, will make it possible for us to identify cuspidal modules in \S\ref{cuspidalmodsec} while operating solely at the level of characters.
\begin{definition} Let \(l=2\), \(\kappa = (k_1, k_2)\), and \(\blam = (\lambda^{(1)}, \lambda^{(2)}) \in \mathscr{P}^\kappa\). Write \(x_i:=n(\blam,i)\), and \(y_i := \lambda^{(i)}_1\). If \((x_1,1,1)\) (the bottom left node in \(\lambda^{(1)}\)) and \((1,y_2,2)\) (the top right node in \(\lambda^{(2)}\)) are such that
\(
\textup{res}(x_1,1,1) = \textup{res}(1,y_2,2) + 1,
\)
we call \(\blam\) {\it joinable}. 
\end{definition}
In this section we will assume that \(\blam\) is joinable. Define the one-part multicharges \(\kappa^*:= (k_2 + x_1)\) and \(\kappa_*:=(k_2+x_1-1)\). We now define \(\lambda^*/\mu^* \in \mathscr{S}^{\kappa^*}\) and \(\lambda_*/\mu_* \in \mathscr{S}^{\kappa_*}\) by setting:
\begin{align*}
\lambda^{*}:=(\lambda^{(1)} + y_2 -1, \ldots, \lambda^{(1)}_{x_1} + y_2-1, \lambda_1^{(2)}, \ldots, \lambda^{(2)}_{x_2}),  \hspace{10mm}\mu^{*}:=((y_2-1)^{x_1}),
\end{align*}
and
\begin{align*}
\lambda_{*}:=(\lambda^{(1)} + y_2 , \ldots, \lambda^{(1)}_{x_1} + y_2, \lambda_2^{(2)}, \ldots, \lambda^{(2)}_{x_2}),  \hspace{10mm}
\mu_{*}:=(y_2^{x_1-1}).
\end{align*}
In other words, \(\lambda^*/\mu^*\) is achieved by shifting the Young diagram associated with \(\lambda^{(1)}\) until its bottom-left node lies directly above the top-right node of \(\lambda^{(2)}\), and then viewing this as a one-part skew diagram. Similarly, \(\lambda_*/\mu_*\) is achieved by shifting the Young diagram associated with \(\lambda^{(1)}\) until its bottom-left node lies directly to the right of the top-right node of \(\lambda^{(2)}\). 

There is an obvious bijection \(\tau^*\) (resp. \(\tau_*\)) between nodes of \(\blam\) and \(\lambda^*/\mu^*\) (resp. \(\lambda_*/\mu_*\)), given by
\begin{center}
\begin{tikzpicture}
  \matrix (a) [matrix of math nodes,row sep=.2em,
  column sep=3em, nodes in empty cells]
  {\lambda^{(1)} \ni (a,b,1)&(a,b+y_2-1) \in \lambda^*/\mu^*\\
\lambda^{(2)} \ni (a,b,2) &(a+x_1,b) \in \lambda^*/\mu^*\\  };
  \path[>=stealth,|->] (a-1-1) edge node[above] {$ \tau^* $}(a-1-2);
\path[>=stealth,|->] (a-2-1) edge node[above] {$ \tau^* $}(a-2-2);
\end{tikzpicture}
\end{center}
and, respectively,
\begin{center}
\begin{tikzpicture}
  \matrix (a) [matrix of math nodes,row sep=.2em,
  column sep=3em, nodes in empty cells]
  {\lambda^{(1)} \ni (a,b,1)&(a,b+y_2) \in \lambda_*/\mu_*\\
\lambda^{(2)} \ni (a,b,2) &(a+x_1-1,b) \in \lambda_*/\mu_*\\  };
  \path[>=stealth,|->] (a-1-1) edge node[above] {$ \tau_* $}(a-1-2);
\path[>=stealth,|->] (a-2-1) edge node[above] {$ \tau_* $}(a-2-2);
\end{tikzpicture}
\end{center}
Note that the charges \(\kappa^*\) and \(\kappa_*\) are chosen so that residues of nodes are preserved under this bijection.
Let \(\TTT \in \St(\blam)\). Viewing the tableau as a function \(\{1, \ldots, d\} \to \blam\), then composing with \(\tau^*\) (resp. \(\tau_*\)) gives a \(\lambda^*/\mu^*\)-tableau (resp. \(\lambda_*/\mu_*\)-tableau). Define
\begin{align*}
\TTT^*:= \tau^* \circ \TTT \hspace{5mm} \textup{and} \hspace{5mm} \TTT_*:= \tau_*\circ \TTT.
\end{align*}
Then we have bijections
\begin{center}
\begin{tikzpicture}
  \matrix (a) [matrix of math nodes,row sep=.2em,
  column sep=3em, nodes in empty cells]
  {\{ \TTT \in \St(\blam) \mid \TTT(x_1,1,1) < \TTT(1,y_2,2)\}&\St(\lambda^*/\mu^*)\\ };
  \path[>=stealth,->] (a-1-1) edge node[above] {$ \tau^* $}(a-1-2);
\end{tikzpicture}
\end{center}
and
\begin{center}
\begin{tikzpicture}
  \matrix (a) [matrix of math nodes,row sep=.2em,
  column sep=3em, nodes in empty cells]
  {\{ \TTT \in \St(\blam) \mid \TTT(x_1,1,1) > \TTT(1,y_2,2)\}&\St(\lambda_*/\mu_*)\\ };
  \path[>=stealth,->] (a-1-1) edge node[above] {$ \tau_* $}(a-1-2);
\end{tikzpicture}
\end{center}
\begin{example}\label{joinexample} Let \(e=3\), \(\kappa=(0,1)\), \(\lambda^{(1)}=(3,2,2)\) and \(\lambda^{(2)}=(2,2)\). Then \(\blam\) is joinable since \(\textup{res}(3,1,1) =1 = 0 + 1 = \textup{res}(1,2,2) +1\). Then, with respect to the row- and column-leading tableaux, we have:
\begin{align*}
\TTT^{\blam}=
\begin{tikzpicture}[scale=0.42]
\tikzset{baseline=-10mm}
\draw(0,0)--(3,0);
\draw(0,-1)--(3,-1);
\draw(0,-2)--(2,-2);
\draw(0,-3)--(2,-3);
\draw(0,0)--(0,-3);
\draw(1,0)--(1,-3);
\draw(2,0)--(2,-3);
\draw(3,0)--(3,-1);
\draw(0,-4)--(2,-4);
\draw(0,-5)--(2,-5);
\draw(0,-6)--(2,-6);
\draw(0,-4)--(0,-6);
\draw(1,-4)--(1,-6);
\draw(2,-4)--(2,-6);
\draw(0.5,-0.5) node{$\scriptstyle 1$};
\draw(1.5,-0.5) node{$\scriptstyle 2$};
\draw(2.5,-0.5) node{$\scriptstyle 3$};
\draw(0.5,-1.5) node{$\scriptstyle 4$};
\draw(1.5,-1.5) node{$\scriptstyle 5$};
\draw(0.5,-2.5) node{$\scriptstyle 6$};
\draw(1.5,-2.5) node{$\scriptstyle 7$};
\draw(0.5,-4.5) node{$\scriptstyle 8$};
\draw(1.5,-4.5) node{$\scriptstyle 9$};
\draw(0.5,-5.5) node{$\scriptstyle 10$};
\draw(1.5,-5.5) node{$\scriptstyle 11$};
\end{tikzpicture}
\hspace{8mm}
(\TTT^{\blam})^*=
\begin{tikzpicture}[scale=0.42]
\tikzset{baseline=-10mm}
\draw(0,0)--(3,0);
\draw(0,-1)--(3,-1);
\draw(0,-2)--(2,-2);
\draw(0,-3)--(2,-3);
\draw(0,0)--(0,-3);
\draw(1,0)--(1,-3);
\draw(2,0)--(2,-3);
\draw(3,0)--(3,-1);
\draw(-1,-3)--(1,-3);
\draw(-1,-4)--(1,-4);
\draw(-1,-5)--(1,-5);
\draw(-1,-3)--(-1,-5);
\draw(0,-3)--(0,-5);
\draw(1,-3)--(1,-5);
\draw(0.5,-0.5) node{$\scriptstyle 1$};
\draw(1.5,-0.5) node{$\scriptstyle 2$};
\draw(2.5,-0.5) node{$\scriptstyle 3$};
\draw(0.5,-1.5) node{$\scriptstyle 4$};
\draw(1.5,-1.5) node{$\scriptstyle 5$};
\draw(0.5,-2.5) node{$\scriptstyle 6$};
\draw(1.5,-2.5) node{$\scriptstyle 7$};
\draw(-0.5,-3.5) node{$\scriptstyle 8$};
\draw(0.5,-3.5) node{$\scriptstyle 9$};
\draw(-0.5,-4.5) node{$\scriptstyle 10$};
\draw(0.5,-4.5) node{$\scriptstyle 11$};
\end{tikzpicture}
\hspace{8mm}
\TTT_{\blam}=
\begin{tikzpicture}[scale=0.42]
\tikzset{baseline=-10mm}
\draw(0,0)--(3,0);
\draw(0,-1)--(3,-1);
\draw(0,-2)--(2,-2);
\draw(0,-3)--(2,-3);
\draw(0,0)--(0,-3);
\draw(1,0)--(1,-3);
\draw(2,0)--(2,-3);
\draw(3,0)--(3,-1);
\draw(0,-4)--(2,-4);
\draw(0,-5)--(2,-5);
\draw(0,-6)--(2,-6);
\draw(0,-4)--(0,-6);
\draw(1,-4)--(1,-6);
\draw(2,-4)--(2,-6);
\draw(0.5,-0.5) node{$\scriptstyle 5$};
\draw(1.5,-0.5) node{$\scriptstyle 8$};
\draw(2.5,-0.5) node{$\scriptstyle 11$};
\draw(0.5,-1.5) node{$\scriptstyle 6$};
\draw(1.5,-1.5) node{$\scriptstyle 9$};
\draw(0.5,-2.5) node{$\scriptstyle 7$};
\draw(1.5,-2.5) node{$\scriptstyle 10$};
\draw(0.5,-4.5) node{$\scriptstyle 1$};
\draw(1.5,-4.5) node{$\scriptstyle 3$};
\draw(0.5,-5.5) node{$\scriptstyle 2$};
\draw(1.5,-5.5) node{$\scriptstyle 4$};
\end{tikzpicture}
\hspace{8mm}
(\TTT_{\blam})_*=
\begin{tikzpicture}[scale=0.42]
\tikzset{baseline=-10mm}
\draw(0,0)--(3,0);
\draw(0,-1)--(3,-1);
\draw(0,-2)--(2,-2);
\draw(0,-3)--(2,-3);
\draw(0,0)--(0,-3);
\draw(1,0)--(1,-3);
\draw(2,0)--(2,-3);
\draw(3,0)--(3,-1);
\draw(-2,-2)--(0,-2);
\draw(-2,-3)--(0,-3);
\draw(-2,-4)--(0,-4);
\draw(-2,-2)--(-2,-4);
\draw(-1,-2)--(-1,-4);
\draw(0,-2)--(0,-4);
\draw(0.5,-0.5) node{$\scriptstyle 5$};
\draw(1.5,-0.5) node{$\scriptstyle 8$};
\draw(2.5,-0.5) node{$\scriptstyle 11$};
\draw(0.5,-1.5) node{$\scriptstyle 6$};
\draw(1.5,-1.5) node{$\scriptstyle 9$};
\draw(0.5,-2.5) node{$\scriptstyle 7$};
\draw(1.5,-2.5) node{$\scriptstyle 10$};
\draw(-1.5,-2.5) node{$\scriptstyle 1$};
\draw(-0.5,-2.5) node{$\scriptstyle 3$};
\draw(-1.5,-3.5) node{$\scriptstyle 2$};
\draw(-0.5,-3.5) node{$\scriptstyle 4$};
\end{tikzpicture}
\end{align*}
\end{example}
\begin{lem}\label{t^*}
Let \(\blam \in \mathscr{P}^{\kappa}\) be joinable, \(\textup{res}(1,y_2,2)=i\), and let \(\TTT \in \St(\blam)\). Then
\begin{align*}
\deg \TTT^*  = \deg \TTT -\left( \Lambda_{i}, \textup{cont}(\lambda^{(1)})\right)
\end{align*}
if  \(\TTT(x_1,1,1) < \TTT(1,y_2,2)\), and 
\begin{align*}
\deg \TTT_* = \deg \TTT - \left( \Lambda_{i+1}, \textup{cont}(\lambda^{(1)})\right)
\end{align*}
if \(\TTT(x_1,1,1) > \TTT(1,y_2,2)\).
\end{lem}
\begin{proof}
We prove the first statement. The second is similar. Let \(\UUU= \TTT_{\leq t}\) for some \(t\). We'll show that the claim holds for \(\UUU\):
\begin{align}\label{ucrit}
\deg(\UUU^*) = \deg(\UUU)  -\left( \Lambda_{i}, \textup{cont}(\sh(\UUU)^{(1)})\right),
\end{align}
going by induction on the size of \(\textup{sh}(\UUU)\).

For the base case we have \(\textup{sh}(\UUU) = (\varnothing, \varnothing)\), so that \(\deg(\UUU) = 0=\deg(\UUU^*)= 0\).

Now we attack the induction step. By the inductive definition of degree for tableaux, we just need to show that for every removable node \(A\) in \(\UUU\), 
\begin{align}\label{dAcrit2}
d_{\tau^*(A)}(\sh(\YYY(\UUU^*)))- d_A(\textup{sh}(\UUU))  = \begin{cases} -1 & A \in \lambda^{(1)} \textup{ and } \textup{res}(A)=i,\\
0 & \textup{otherwise}.\end{cases}
\end{align}
By the construction of \(\UUU^*\), it is clear that for \(1 \leq r \leq x_1-1\) and \(j \in I\), the \(r\)-th row of \(\sh(\UUU)^{(1)}\) has an addable (resp. removable) \(j\)-node if and only if the corresponding \(r\)-th row in \(\sh(\YYY(\UUU^*))\) has an addable (resp. removable) \(j\)-node. Similarly, for \(2 \leq r \leq x_2\), the \(r\)-th row of \(\sh(\UUU)^{(2)}\) has an addable (resp. removable) \(j\)-node if and only if the corresponding \((x_1+r)\)-th row in \(\sh(\YYY(\UUU^*))\) has an addable (resp. removable) \(j\)-node. Thus it remains to compare addable/removable nodes of rows \(x_1,x_1+1\) in \(\sh(\UUU)^{(1)}\) and row \(1\) in \(\sh(\UUU)^{(2)}\) with the rows \( x_1, x_1+1\) in \(\sh(\YYY(\UUU^*))\).

For simplicity, we label
\begin{itemize}
\item \(B:=(x_1-1,1,1)\), the bottom-left node in \(\lambda^{(1)}\). Write \(B^*:= \tau^*(B)\). Both \(B\) and \(B^*\) have residue \(i+1\).
\item \(C:=(1,y_2-1,2)\), the node to the left of the top-right node in \(\lambda^{(2)}\). Write \(C^*:= \tau^*(C)\). Both \(C\) and \(C^*\) have residue \(i-1\).
\item \(D:=(1,y_2,2)\), the top-right node in \(\lambda^{(2)}\). Write \(D^*:= \tau^*(D)\). Both \(D\) and \(D^*\) have residue \(i\).
\end{itemize}
For instance, in \ref{joinexample} these labels correspond to the nodes shown below:
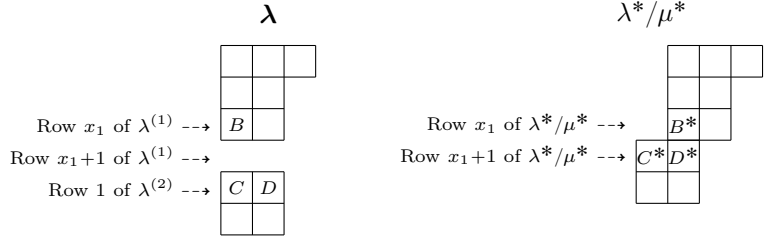
\begin{align*}
\begin{tikzpicture}[scale=0.42]
\tikzset{baseline=0mm}
\draw(0,0)--(3,0);
\draw(0,-1)--(3,-1);
\draw(0,-2)--(2,-2);
\draw(0,-3)--(2,-3);
\draw(0,0)--(0,-3);
\draw(1,0)--(1,-3);
\draw(2,0)--(2,-3);
\draw(3,0)--(3,-1);
\draw(0,-4)--(0,-6);
\draw(1,-4)--(1,-6);
\draw(2,-4)--(2,-6);
\draw(0,-4)--(2,-4);
\draw(0,-5)--(2,-5);
\draw(0,-6)--(2,-6);
\draw(0.5,-2.5) node{$\scriptstyle B$};
\draw(0.5,-4.5) node{$\scriptstyle C$};
\draw(1.5,-4.5) node{$\scriptstyle D$};
\draw(1.5,1) node{$\blam$};
\draw(0,-2.5) node[left]{$\scriptstyle \text{Row } x_1 \text{ of }\lambda^{(1)}\; \dashrightarrow$};
\draw(0,-3.5) node[left]{$\scriptstyle \text{Row } x_1+1 \text{ of } \lambda^{(1)}\; \dashrightarrow$};
\draw(0,-4.5) node[left]{$\scriptstyle \text{Row } 1 \text{ of }\lambda^{(2)}\; \dashrightarrow$};
\end{tikzpicture}
\hspace{10mm}
\begin{tikzpicture}[scale=0.42]
\tikzset{baseline=0mm}
\draw(2,0)--(5,0);
\draw(2,-1)--(5,-1);
\draw(2,-2)--(4,-2);
\draw(2,-3)--(4,-3);
\draw(2,0)--(2,-3);
\draw(3,0)--(3,-3);
\draw(4,0)--(4,-3);
\draw(5,0)--(5,-1);
\draw(1,-3)--(1,-5);
\draw(2,-3)--(2,-5);
\draw(3,-3)--(3,-5);
\draw(1,-3)--(3,-3);
\draw(1,-4)--(3,-4);
\draw(1,-5)--(3,-5);
\draw(1.5,1) node{$\lambda^*/\mu^*$};
\draw(2.5,-2.5) node{$\scriptstyle B^*$};
\draw(1.5,-3.5) node{$\scriptstyle C^*$};
\draw(2.5,-3.5) node{$\scriptstyle D^*$};
\draw(1,-2.5) node[left]{$\scriptstyle \text{Row }x_1 \text{ of }\lambda^*/\mu^*\; \dashrightarrow$};
\draw(1,-3.5) node[left]{$\scriptstyle \text{Row }x_1+1 \text{ of }\lambda^*/\mu^*\; \dashrightarrow$};
\end{tikzpicture}
\end{align*}
There are five cases to consider. 

\begin{enumerate}
\item \(\{B,C,D\} \cap \sh(\UUU) = \varnothing\).
\begin{itemize}
\item Row \(x_1\) of \(\sh(\UUU)^{(1)}\) has addable node \(B\) iff \(B^*\) is addable in \(\sh(\YYY(\UUU^*))\). Row \(x_1\) of \(\sh(\UUU)^{(1)}\) has no removable nodes.
\item Row \(x_1+1\) of \(\sh(\UUU)^{(1)}\) has no addable/removable nodes.
\item Row 1 of \(\sh(\UUU)^{(2)}\) has an addable (resp. removable) \(j\)-node iff row \(x_1+1\) of \(\sh(\YYY(\UUU^*))\) has an addable (resp. removable) \(j\)-node.
\item Row \(x_1\) of \(\sh(\YYY(\UUU^*))\) has a removable \(i\)-node (the bottom-right node of \(\mu_{*}\) to be precise).
\end{itemize}
From this (\ref{dAcrit2}) follows.
\item \(\{B,C,D\} \cap \sh(\UUU) = \{B\}\).
\begin{itemize}
\item Row \(x_1\) of \(\sh(\UUU)^{(1)}\) has an addable (resp. removable) \(j\)-node iff row \(x_1\) of \(\sh(\YYY(\UUU^*))\) has an addable (resp. removable) \(j\)-node.
\item Row \(x_1+1\) of \(\sh(\UUU)^{(1)}\) has an addable \(i\)-node, and no removable nodes. 
\item Row \(1\) of \(\sh(\UUU)^{(2)}\) has an addable (resp. removable) \(j\)-node iff row \(x_1+1\) of \(\sh(\YYY(\UUU^*))\) has an addable (resp. removable) \(j\)-node.
\end{itemize}
From this (\ref{dAcrit2}) follows.
\item \(\{B,C,D\} \cap \sh(\UUU) = \{C\}\).
\begin{itemize}
\item Row \(x_1\) of \(\sh(\UUU)^{(1)}\) has addable node \(B\) iff \(B^*\) is addable in \(\sh(\YYY(\UUU^*))\). Row \(x_1\) of \(\sh(\UUU)^{(1)}\) has no removable nodes.
\item Row \(x_1+1\) of \(\sh(\UUU)^{(1)}\) has no addable/removable nodes.
\item Row 1 of \(\sh(\UUU)^{(2)}\) has an addable \(i\)-node \(D\). Row 1 of \(\sh(\UUU)^{(2)}\) has removable node \(C\) iff \(C^*\) is removable in row \(x_1 +1\) of \(\sh(\YYY(\UUU^*))\).
\item Row \(x_1\) of \(\sh(\YYY(\UUU^*))\) has no removable nodes.
\item Row \(x_1+1\) of \(\sh(\YYY(\UUU^*))\) has no addable nodes.
\end{itemize}
From this (\ref{dAcrit2}) follows.
\item \(\{B,C,D\} \cap \sh(\UUU) = \{B,C\}\).
\begin{itemize}
\item Row \(x_1\) of \(\sh(\UUU)^{(1)}\) has an addable (resp. removable) \(j\)-node iff  row \(x_1\) of \(\sh(\YYY(\UUU^*))\) has an addable (resp. removable) \(j\)-node. 
\item Row \(x_1 +1\) of \(\sh(\UUU)^{(1)}\) has an addable \(i\)-node and no removable node.
\item Row \(1\) of \(\sh(\UUU)^{(2)}\) has an addable (resp. removable) \(j\)-node iff  row \(x_1+1\) of \(\sh(\YYY(\UUU^*))\) has an addable (resp. removable) \(j\)-node. 
\end{itemize}
From this (\ref{dAcrit2}) follows.
\item \(\{B,C,D\} \cap \sh(\UUU) = \{B,C,D\}\).
\begin{itemize}
\item Row \(x_1\) of \(\sh(\UUU)^{(1)}\) has an addable (resp. removable) \(j\)-node to the right of \(B\) iff row \(x_1\) of \(\sh(\YYY(\UUU^*))\) has an addable (resp. removable) \(j\)-node to the right of \(B^*\). The \((i+1)\)-node \(B\) is not removable in row \(x_1\) of \(\sh(\UUU)^{(1)}\) iff row \(x_1+1\) of \(\sh(\YYY(\UUU^*))\) has an addable \((i+1)\)-node to the right of \(D^*\).
\item Row \(x_1 +1\) of \(\sh(\UUU)^{(1)}\) has an addable \(i\)-node and no removable node.
\item Row \(1\) of \(\sh(\UUU)^{(2)}\) has an addable \((i+1)\)-node. Row \(1\) of \(\sh(\UUU)^{(2)}\) has removable node \(D\) iff \(D^*\) is removable in row \(x_1+1\) of \(\sh(\YYY(\UUU^*))\).
\end{itemize}
From this (\ref{dAcrit2}) follows.
\end{enumerate}
Thus in all cases, (\ref{dAcrit2}) is satisfied, and the lemma follows by induction.
\end{proof}
\begin{definition}\label{minimal}
We say that an arbitrary skew diagram \(\blam/\bmu\) is {\it minimal} if \(\mu^{(i)}_1 < \lambda^{(i)}_1\) and \(\mu^{(i)}_{n(\blam,i)}=0\) for all \(i\). Less formally, a skew diagram is minimal if, in each component, it has nodes in the top row and in the leftmost column. 
\end{definition}
\begin{definition} Let \(l\)=2. We say that \(\blam/\bmu \in \mathscr{S}^{\kappa}\) is {\it joinable} if it is minimal and \(\blam\) is joinable. 
\end{definition}
Assuming \(\blam/\bmu\) is joinable, define \(\kappa^*, \kappa_*, x_i, y_i\) as before, with respect to \(\blam\). In the same vein as before we construct a skew tableau \(\lambda^*/\mu^*\) by shifting the skew diagram \(\lambda^{(1)}/\mu^{(1)}\) until the lower left node lies above the upper right node of \(\lambda^{(2)}/\mu^{(2)}\), and we construct a skew tableau \(\lambda^*/\mu^*\) by shifting the skew diagram \(\lambda^{(1)}/\mu^{(1)}\) until the lower left node lies directly to the right of the upper right node of \(\lambda^{(2)}/\mu^{(2)}\). Specifically, define \(\lambda^*/\mu^* \in \mathscr{S}^{\kappa^*}\) and \(\lambda_*/\mu_* \in \mathscr{S}^{\kappa_*}\) by setting:
\begin{align*}
\lambda^*&:=(\lambda_1^{(1)} + y_2 -1, \ldots, \lambda^{(1)}_{x_1} + y_2-1, \lambda_1^{(2)}, \ldots, \lambda^{(2)}_{x_2}),\\
\mu^*&:=(\mu^{(1)}_1+y_2-1, \ldots,\mu^{(1)}_{x_1}+y_2-1, \mu^{(2)}_1, \ldots, \mu^{(2)}_{x_2}),
\end{align*}
and
\begin{align*}
\lambda_*&:=(\lambda_1^{(1)} + y_2 , \ldots, \lambda^{(1)}_{x_1} + y_2, \lambda_2^{(2)}, \ldots, \lambda^{(2)}_{x_2}),\\
\mu_*&:=(\mu^{(1)}_1+y_2, \ldots,\mu^{(1)}_{x_1}+y_2, \mu^{(2)}_2, \ldots,\mu^{(2)}_{x_2}).
\end{align*}
With \(\tau_*\) and \(\tau^*\) defined as before with respect to \(\blam\), we define
\begin{align*}
\textup{Tab}(\lambda^*/\mu^*)\ni \ttt^* := \tau^* \circ \ttt \hspace{10mm} \textup{and} \hspace{10mm}
\textup{Tab}(\lambda_*/\mu_*)\ni \ttt_*:= \tau_* \circ \ttt
\end{align*}
for \(\uuu \in \textup{Tab}(\blam/\bmu)\). We have bijections
\begin{center}
\begin{tikzpicture}
  \matrix (a) [matrix of math nodes,row sep=.2em,
  column sep=3em, nodes in empty cells]
  {\{ \ttt \in \St(\blam/\bmu) \mid \ttt(x_1,1,1) < \ttt(1,y_2,2)\}&\St(\lambda^*/\mu^*)\\ };
  \path[>=stealth,->] (a-1-1) edge node[above] {$ \tau^* $}(a-1-2);
\end{tikzpicture}
\end{center}
and
\begin{center}
\begin{tikzpicture}
  \matrix (a) [matrix of math nodes,row sep=.2em,
  column sep=3em, nodes in empty cells]
  {\{ \ttt \in \St(\blam/\bmu) \mid \ttt(x_1,1,1) > \ttt(1,y_2,2)\}&\St(\lambda^*/\mu^*)\\ };
  \path[>=stealth,->] (a-1-1) edge node[above] {$ \tau_* $}(a-1-2);
\end{tikzpicture}
\end{center}
\begin{prop}\label{skewstar}
Let \(\blam/\bmu \in \mathscr{S}^\kappa\) be joinable, with the  top right node in \(\lambda^{(2)}\) having residue \(i\). Let \(\ttt \in \St(\blam/\bmu)\). Then
\begin{align*}
\deg \ttt^* = \deg \ttt - \left( \Lambda_{i}, \textup{cont}(\lambda^{(1)}/\mu^{(1)})\right)
\end{align*}
if \(\ttt(x_1,1,1) < \ttt(1,y_2,2)\), and
\begin{align*}
\deg \ttt_* = \deg \ttt - \left( \Lambda_{i+1}, \textup{cont}(\lambda^{(1)}/\mu^{(1)})\right)
\end{align*}
if \(\ttt(x_1,1,1) > \ttt(1,y_2,2)\).
\end{prop}
\begin{proof}
We prove the first statement. The second is similar. Let \(\nu = \lambda^* \backslash \sh((\ttt^{\blam})^*)\). Then by definition,
\begin{align}\label{join1}
\deg \ttt^* &= \deg \YYY(\ttt^*) - \deg \TTT^{(\mu^*)}\\
\deg \ttt &= \deg \YYY(\ttt) - \deg \TTT^{\bmu}\\
\deg \YYY(\ttt)^* &= \deg \YYY(\YYY(\ttt)^*) - \deg \TTT^\nu
\end{align}
Lemma \ref{t^*} gives us
\begin{align}\label{join2}
\deg \YYY(\ttt)^* = \deg \YYY(\ttt) - \left(\Lambda_i, \textup{cont}(\lambda^{(1)})\right)
\end{align}
Note that \(\YYY(\YYY(\ttt)^*)\) and \(\YYY(\ttt^*)\) agree outside of \(\mu^*\), so 
\begin{align}
\nonumber
\deg \YYY(\ttt)^* + \deg \TTT^\nu- \deg \YYY(\ttt^*)&= \deg \YYY(\YYY(\ttt)^*) - \deg \YYY(\ttt^*) \\
\nonumber
&= \deg \YYY(\YYY(\ttt)^*)_{\leq |\mu^*|} - \deg \YYY(\ttt^*)_{\leq |\mu^*|}\\
\nonumber
&= \deg \YYY\left(\left(\TTT^{\bmu}\right)^*\right) - \deg \TTT^{(\mu^*)}\\
\nonumber
&= \deg (\TTT^{\bmu})^* + \deg \TTT^\nu - \deg \TTT^{(\mu^*)}\\
\label{join3}
&= \deg \TTT^{\bmu}  - \left(\Lambda_i, \textup{cont}(\mu^{(1)})\right) + \deg \TTT^\nu - \deg \TTT^{(\mu^*)},
\end{align}
using (\ref{ucrit}) in the last step. Combining equations (\ref{join1})---(\ref{join3}) yields the result.
\end{proof}
\begin{lem}\label{skewcharsep}
Let \(\blam/\bmu  \in \mathscr{S}^\kappa\) be a joinable skew diagram, and assume the  top right node in \(\lambda^{(2)}\) has residue \(i\). With \(\lambda^*/\mu^* \in \mathscr{S}^{\kappa^*}\) and \(\lambda_*/\mu_* \in \mathscr{S}^{\kappa_*}\) defined as above, we have
\begin{align*}
 \CH_q(S^{\blam/\bmu}) = q^{d^*}\CH_q\left(S^{\lambda^*/\mu^*}\right) + q^{d_*}\CH_q\left(S^{\lambda_*/\mu_*}\right) =q^{d_{\blam/\bmu}} \CH_q\left( S^{\lambda^{(1)}/\mu^{(1)}} \circ S^{\lambda^{(2)}/\mu^{(2)}}\right),
\end{align*}
where
\begin{align*}
d^* &= \left( \Lambda_{i}, \textup{cont}(\lambda^{(1)}/\mu^{(1)})\right)\\
d_* &=   \left( \Lambda_{i+1}, \textup{cont}(\lambda^{(1)}/\mu^{(1)})\right)\\
d_{\blam/\bmu} &= \deg \ttt^{\blam/\bmu} - \deg \ttt^{\lambda^{(1)}/\mu^{(1)}} - \deg \ttt^{\lambda^{(2)}/\mu^{(2)}}.
\end{align*}
\end{lem}
\begin{proof}
The first equality follows from Corollary \ref{basisthm} and Proposition \ref{skewstar}, via the bijections that \(\tau^*\) and \(\tau_*\) induce on basis elements. The second equality is Theorem \ref{circprod}.
\end{proof}


\section{Cuspidal systems}\label{cuspidalsec}
\noindent Our primary motivation in developing the theory of skew Specht modules was to describe an important  class of \(R_\alpha\)-modules called {\it cuspidal modules}. In this section we very briefly describe the notion of cuspidal systems for KLR algebras of type \({\tt A}^{(1)}_{e-1}\). See \cite{cusp}, \cite{mcn} for a thorough treatment.
\subsection{Convex preorders}\label{preorders}
For the rest of this paper, we consider the case \(e>0\) and \(\mathcal{O}=F\), an arbitrary ground field. Recall from \S\ref{lie} the set of positive roots \(\Phi_+\). It is known that \(\Phi_+ = \Phi_+^{\textup{re}} \sqcup \Phi_+^{\textup{im}}\), where \(\Phi_+^{\textup{re}}\) are {\it real roots}, and \(\Phi_+^{\textup{im}}=\{n \delta \mid n \in \ZZ_{>0}\}\) are the {\it imaginary roots}, where \(\delta = \alpha_0 + \alpha_1 + \cdots + \alpha_{e-1}\) is the {\em null root}. Write \(\Psi:= \Phi_+^\textup{re} \cup \{\delta\}\). Take a {\it convex preorder} on \(\Phi_+\), i.e., a preorder \(\preceq\) such that for all \(\alpha,\beta\in \Phi_+\):
\begin{enumerate}
\item \(\alpha \preceq \beta \textup{ or } \beta \preceq \alpha;\)
\item \(\textup{if } \alpha \preceq \beta \textup{ and } \alpha + \beta \in \Phi_+, \textup{ then } \alpha \preceq \alpha+\beta  \preceq \beta;\)
\item \(\alpha \preceq \beta \textup{ and } \beta \preceq \alpha \textup{ if and only if } \alpha \textup{ and } \beta \textup{ are proportional}.\)
\end{enumerate}
Then \(\Phi_+^{\textup{re}} = \Phi_{\succ}^{\textup{re}} \sqcup \Phi_{\prec}^{\textup{re}}\), where
\(
\Phi_{\succ}^{\textup{re}} := \{ \alpha \in \Phi_+^{\textup{re}} \mid \alpha \succ \delta\}
\)
and
\(
\Phi_{\prec}^{\textup{re}} := \{ \alpha \in \Phi_+^{\textup{re}} \mid \alpha \prec \delta\}
\).
Let \({\tt C}'\) be the Cartan matrix of {\it finite type} corresponding to the subset of vertices \(I' = I \backslash \{0\}\), and let \(\Phi'\) be the corresponding root system with positive roots \(\Phi_+'\). In what follows we make the additional assumption that the convex preorder is {\it balanced}:
\begin{align*}
\Phi_{\succ}^{\textup{re}} &= \{ m \delta + \alpha \mid \alpha \in \Phi'_+, m \in \ZZ_{\geq 0}\} = \{ m \delta + \alpha_i + \alpha_{i+1}+ \cdots +\alpha_j \mid m \in \ZZ_{\geq 0}, 1 \leq i \leq j \leq e-1\}\\
\Phi_{\prec}^{\textup{re}} &=  \{ m \delta - \alpha \mid \alpha \in \Phi'_+, m \in \ZZ_{\geq 1}\} = \{ m \delta - \alpha_i - \alpha_{i+1}- \cdots -\alpha_j \mid m \in \ZZ_{\geq 1}, 1 \leq i \leq j \leq e-1\}.
\end{align*}
Balanced convex preorders exist, see \cite{bcp}. 

\subsection{Cuspidal systems}
Let \(\alpha \in \Phi_+\). Given an \(R_\alpha\)-module \(M\), we say \(M\) is {\em semicuspidal} (resp. {\em cuspidal}) if \(\Res_{\beta,\gamma}^{\alpha} M \neq 0\) implies that \(\beta\) is a sum of positive roots less than or equal to (resp. less than) \(\alpha\), and \(\gamma\) is a sum of positive roots greater than or equal to (resp. greater than) \(\alpha\).
The following is proved in \cite{cusp, imagsw, TW}:

\begin{thm}\(\)
\begin{enumerate}
\item For every \(\alpha \in \Phi_+^\textup{re}\), there is a unique simple cuspidal \(R_\alpha\)-module \(L_\alpha\).
\item For every \(n>0\), the simple semicuspidal \(R_{n\delta}\)-modules may be canonically labeled \(\{ L(\bnu) \mid \bnu \vdash n\}\), where \(\bnu = (\nu^{(1)}, \ldots, \nu^{(e-1)})\) ranges over \((e-1)\)-multipartitions of \(n\).
\end{enumerate}
\end{thm} 

Let \(\alpha \in Q_+\). Define the set \(\Pi(\alpha)\) of {\it root partitions} of \(\alpha\) to be the set of all pairs \((M, \bnu)\), where \(M=(n_\beta)_{\beta \in \Psi}\) is a tuple of nonnegative integers such that \(\sum_{\beta \in \Psi}n_\beta \beta = \alpha\), and \(\bnu\) is an \((e-1)\)-multipartition of \(n_\delta\). There is a bilexicographic partial order \(\leq\) on \(\Pi(\alpha)\), see \cite{cusp}. Given \((M, \bnu) \in \Pi(\alpha)\), define the {\it proper standard module}
\begin{align*}
\overline{\Delta}(M, \bnu):= L_{\beta_1}^{\circ n_{\beta_1}} \circ \cdots \circ L_{\beta_k}^{\circ n_{\beta_k}} \circ L(\bnu) \circ  L_{\beta_{k+1}}^{\circ n_{\beta_{k+1}}} \circ \cdots \circ L_{\beta_{t}}^{\circ n_{\beta_t}} \langle \textup{shift}(M,\bnu)\rangle,
\end{align*}
where \(\beta_1, \ldots, \beta_t\) are the real positive roots indexing nonzero entries in \(M\), labeled such that \(\beta_1 \succ \cdots \succ \beta_k \succ \delta \succ \beta_{k+1} \cdots \succ \beta_t\), and 
\(
\textup{shift}(M,\bnu) = \sum_{i=1 \neq t}(\beta_i, \beta_i)n_{\beta_i}(n_{\beta_i}-1)/4.
\)

Much of the importance of cuspidal systems lies in the following classification theorem:
\begin{thm} \label{cuspthm} \textup{\cite[Main Theorem]{cusp}} 
\(\)
\begin{enumerate}
\item[\textup{(i)}] For every root partition \((M, \bnu)\), the proper standard module \(\overline{\Delta}(M,\bnu)\) has irreducible head, denoted \(L(M,\bnu)\).
\item[\textup{(ii)}] \(\{L(M,\bnu) \mid (M, \bnu) \in \Pi(\alpha)\}\) is a complete and irredundant system of irreducible \(R_\alpha\)-modules up to isomorphism.
\item[\textup{(iii)}] \([\overline{\Delta}(M,\bnu):L(M,\bnu)]_q = 1\), and \([\overline{\Delta}(M,\bnu):L(M,\bzeta)]_q \neq 0\) implies \((N,\bnu) \leq (M,\bzeta)\).
\item[\textup{(iv)}] \(L(M, \bnu)^\circledast \cong L(M,\bnu)\).
\end{enumerate}
\end{thm}

\subsection{Minuscule imaginary representations}
The `smallest' simple semicuspidal imaginary modules, those in \(R_\delta\)-mod, are of particular importance. By the above, they are in bijection with \((e-1)\)-multipartitions of 1. We label them \(L_{\delta,i}\), for \(i \in I \backslash\{0\}\).
\begin{prop}\label{minuscule}
For each \(i \in I\backslash\{0\}\), \(L_{\delta,i}\) can be characterized up to isomorphism as the unique irreducible \(R_\delta\)-module such that \(i_1 = 0\) and \(i_e=i\) for all words \(\bi\) of \(L_{\delta,i}\).
\end{prop}
\begin{proof}
This is \cite[Lemma 5.1, Corollary 5.3]{cusp}.
\end{proof}
\subsection{Minimal pairs}
Let \(\rho \in \Phi^{\textup{re}}_+\). A pair of positive roots \((\beta,\gamma)\) is called a {\it minimal pair} for \(\rho\) if
\begin{enumerate}
\item[(i)] \(\beta + \gamma = \rho\) and \(\beta \succ \rho\);
\item[(ii)] for any other pair \((\beta',\gamma')\) satisfying (i) we have \(\beta' \succ \beta\) or \(\gamma' \prec \gamma\).
\end{enumerate}
\begin{lem}\label{minpair} Let \(\rho \in \Phi^{\textup{re}}_+\) and \((\beta,\gamma)\) be a minimal pair for \(\rho\). If \(L\) is a composition factor of \(\overline{\Delta}(\beta,\gamma) = L_\beta \circ L_\gamma\), then \(L \cong L(\beta, \gamma)\) or \(L \cong L_\rho\), up to shift. 
\end{lem}
\begin{proof}
This follows from the minimality of \((\beta, \gamma) \in \Pi(\rho) \backslash \{\rho\}\) and Theorem \ref{cuspthm}(iii).
\end{proof}
One can be more precise in the case that \((\beta, \gamma)\) be a {\it real} minimal pair for \(\rho\); i.e., when \(\beta,\gamma \in \Phi^{\textup{re}}_+\). Define
\begin{align}\label{defp}
p_{\beta,\gamma}:= \max \{n \in \ZZ_{\geq 0} \mid \beta - n \gamma \in \Phi_+\}.
\end{align}
\begin{lem}\label{realpair} \textup{\cite[Remark 6.5]{cusp}}.
Let \(\rho \in \Phi^{\textup{re}}_+\), and let \((\beta, \gamma)\) be a real minimal pair for \(\rho\). Then 
\begin{align*}
[L_\beta \circ L_\gamma] &= [L(\beta,\gamma)] + q^{p_{\beta,\gamma}- (\beta,\gamma)}[L_\rho],
\end{align*}
and
\begin{align*}
[L_\gamma \circ L_\beta] &= q^{-(\beta,\gamma)}[L(\beta,\gamma)] + q^{-p_{\beta,\gamma}}[L_\rho].
\end{align*}
\end{lem}
Lemmas \ref{minpair} and \ref{realpair} are useful in inductively constructing cuspidal modules. 
\subsection{Extremal words}
Let \(i \in I\). Define \(\theta^*_i: \langle I \rangle \to \langle I \rangle \) by 
\begin{align*}
\theta^*_i(\bj) =
\begin{cases}
j_1 \cdots j_{d-1}&\textup{if } j_d = i;\\
0 &\textup{otherwise}.
\end{cases}
\end{align*}
Extend \(\theta^*_i\) linearly to a map \(\theta^*_i: \mathscr{A}\langle I \rangle \to \mathscr{A}\langle I \rangle\). Let \(x \in \mathscr{A} \langle I \rangle\), and define
\begin{align*}
\varepsilon_i(x) := \max\{k \geq 0 \mid (\theta_i^*)^k(x) \neq 0\}.
\end{align*}
\begin{definition}
A word \(i_1^{a_1}\cdots i_b^{a_b} \in \langle I \rangle\), with \(a_1, \ldots, a_b \in \ZZ_{\geq 0}\), is called {\it extremal} for \(x\) if 
\begin{align*}
a_b = \varepsilon_{i_b}(x), a_{b-1} = \varepsilon_{i_{b-1}}((\theta^*_{i_b})^{a_b}(x)),\ldots,
a_1 = \varepsilon_{i_1}((\theta^*_{i_2})^{a_2} \cdots (\theta^*_{i_b})^{a_b}(x)).
\end{align*}
A word \(\bi \in \langle I \rangle\) is called {\it extremal} for \(M \in R_\alpha\)-mod if it is an extremal word for \(\CH_q M \in \mathscr{A}\bI\).
\end{definition} 
We have the {\it quantum integer} \([n] := (q^n - q^{-n})/(q - q^{-1}) \in \mathscr{A}\) for \(n \in \ZZ\), and the {\it quantum factorial} \([n]^! := [1] [2] \cdots [n]\). The following lemma is useful in establishing multiplicity-one results for \(R_\alpha\)-modules.
\begin{lem}\label{extremal}
\textup{\cite[Lemma 2.28]{cusp}}.
Let \(L\) be an irreducible \(R_\alpha\)-module, and \(\bi = i_1^{a_1}\cdots i_b^{a_b} \in \langle I \rangle_\alpha\) be an extremal word for \(L\). Then \(\dim_q L_{\bi} = [a_1]^! \cdots [a_b]^!\). 
\end{lem}


\section{Cuspidal modules and skew hook Specht modules}\label{cuspidalmodsec}
\noindent Take a balanced convex preorder \(\preceq\) on \(\Phi_+\), as described in Section \ref{preorders}. In this section we prove that the cuspidal modules \(L_\rho\), for \(\rho \in \Phi_+^\textup{re}\) are skew Specht modules associated to certain skew hook shapes, and provide an inductive process for identifying them as such.
\subsection{Cuspidal modules for a balanced convex preorder} 
Throughout this section we work with Young diagrams and skew diagrams of level \(l=1\). Let \(\kappa = (i)\). For \(i \in I\), Let \(\iota_i =(1) \in \mathscr{P}_{\alpha_i}^{\kappa}\). The following is clear:
\begin{lem}\label{basecase}
For \(i \in I\), \(L_{\alpha_i} \cong S^{\iota_i}\).
\end{lem}
Let \(\kappa = (0)\), and \(\eta_i \in \mathscr{S}^{\kappa}_\delta\) be the hook partition of content \(\delta\) with a node of residue \(i\) in the bottom row, depicted below with residues shown:
\begin{align*}
\begin{tikzpicture}[scale=0.42]
\tikzset{baseline=0mm}
\draw(0,0)--(2.5,0);
\draw(0,-1)--(2.5,-1);
\draw(3.5,0)--(6,0);
\draw(3.5,-1)--(6,-1);
\draw(0,0)--(0,-2.5);
\draw(1,0)--(1,-2.5);
\draw(0,-3.5)--(0,-6);
\draw(1,-3.5)--(1,-6);
\draw(2,0)--(2,-1);
\draw(4,0)--(4,-1);
\draw(5,0)--(5,-1);
\draw(6,0)--(6,-1);
\draw(0,-2)--(1,-2);
\draw(0,-4)--(1,-4);
\draw(0,-5)--(1,-5);
\draw(0,-6)--(1,-6);
\draw(3,-0.5) node{$\scriptstyle \cdots$};
\draw(0.5,-2.8) node{$\scriptstyle \vdots$};
\draw(0.5,-0.5) node{$\scriptstyle 0$};
\draw(1.5,-0.5) node{$\scriptstyle 1$};
\draw(0.5,-1.5) node{$\scriptstyle e\hspace{-0.5mm}-\hspace{-0.5mm}1$};
\draw(0.5,-4.5) node{$\scriptstyle i\hspace{-0.5mm}+\hspace{-0.5mm}1$};
\draw(0.5,-5.5) node{$\scriptstyle i$};
\draw(4.5,-0.5) node{$\scriptstyle i\hspace{-0.5mm}-\hspace{-0.5mm}2$};
\draw(5.5,-0.5) node{$\scriptstyle i\hspace{-0.5mm}-\hspace{-0.5mm}1$};
\end{tikzpicture}
\end{align*}
Let \(X_0 = 0\) and define \(X_{i-1}:=F\{v^\TTT \in S^{\eta_i} \mid \res_\TTT(e)=i-1\}\subseteq S^{\eta_i}\) for \(1<i\leq e-1\). 

\begin{lem}\label{imag}\(\)
\begin{enumerate}
\item \(X_{i-1}\) is a submodule of \(S^{\eta_i}\).
\item \(X_{i-1} \cong L_{\delta,i-1}\langle 1 \rangle\) if \(i>1\).
\item \(S^{\eta_i}/X_{i-1} \cong L_{\delta,i}\).
\end{enumerate}
\end{lem}
\begin{proof}
For \(i>1\), it is easy to see that 
\begin{align*}
\{\TTT \in \St(\eta_i) \mid \res_\TTT(e) = i-1\} = \{\TTT \in \St(\eta_i) \mid \deg \TTT = 1\}
\end{align*}
and
\begin{align*}
\{\TTT \in \St(\eta_i) \mid \res_\TTT(e) = i\} = \{\TTT \in \St(\eta_i) \mid \deg \TTT = 0\},
\end{align*}
give a partition of \(\St(\eta_i)\), and hence \(X_{i-1}\) is the span of degree 1 elements in \(S^{\eta_i}\). As there are no repeated entries in words of \(S^{\eta_1}\), it follows that every negatively-graded element of \(R_\delta\) acts as zero on \(S^{\eta_1}\), and hence \(X_{i-1}\) is a submodule. Moreover all words of \(X_{i-1}\) are of the form \((0,\ldots, i-1)\), and all word spaces are 1-dimensional and in degree 1. Thus it follows from Proposition \ref{minuscule} that \(X_{i-1} \cong L_{\delta,i-1}\langle 1 \rangle\). Then all words of \(S^{\eta_i}/X_{i-1}\) are of the form \((0,\ldots,i)\) and all word spaces are 1-dimensional and in degree 0, so again it follows from Proposition \ref{minuscule} that \(S^{\eta_i}/X_{i-1} \cong L_{\delta,i}\).
\end{proof}
For \(1 \leq i \leq e-1\), \(m \in \ZZ_{\geq 0}\), let \(\lambda^{m,i}/\mu^{m,i}\) be the skew hook diagram in \(\mathscr{S}^{\kappa}_{m\delta + \alpha_i}\), where \(l=1\), \(\kappa = ((1-m)i \pmod e)\), 
\begin{align*}
\lambda^{m,i} &= (mi +1,((m-1)i+1)^{e-i}, \ldots,(i+1)^{e-i},1^{e-i})
\end{align*}
and
\begin{align*}
\mu^{m,i} &= (((m-1)i)^{e-i}, \ldots, (2i)^{e-i},i^{e-i}).
\end{align*}
In other words, \(\lambda^{m,i}/\mu^{m,i}\) is the minimal (in the sense of Definition \ref{minimal}) skew hook diagram with residues shown below, with the 0-node appearing on the inner corners \(m\) times, and the \(i\)-node appearing on the outer corners \(m+1\) times.
\begin{align}\label{diagzeta}
\begin{tikzpicture}[scale=0.42]
\tikzset{baseline=0mm}
\draw(0,0)--(2.5,0);
\draw(0,-1)--(2.5,-1);
\draw(3.5,0)--(6,0);
\draw(3.5,-1)--(6,-1);
\draw(0,0)--(0,-2.5);
\draw(1,0)--(1,-2.5);
\draw(0,-3.5)--(0,-6);
\draw(1,-3.5)--(1,-6);
\draw(2,0)--(2,-1);
\draw(4,0)--(4,-1);
\draw(5,0)--(5,-1);
\draw(6,0)--(6,-1);
\draw(0,-2)--(1,-2);
\draw(0,-4)--(1,-4);
\draw(0,-5)--(1,-5);
\draw(0,-6)--(1,-6);
\draw(5,0)--(5,1.5);
\draw(6,0)--(6,1.5);
\draw(5,1)--(6,1);
\draw(7.5,3)--(10,3);
\draw(7.5,4)--(10,4);
\draw(8,3)--(8,4);
\draw(9,3)--(9,5.5);
\draw(10,3)--(10,5.5);
\draw(9,5)--(10,5);
\draw(9,6.5)--(9,9);
\draw(10,6.5)--(10,9);
\draw(9,9)--(11.5,9);
\draw(9,8)--(11.5,8);
\draw(9,7)--(10,7);
\draw(11,9)--(11,8);
\draw(12.5,9)--(15,9);
\draw(12.5,8)--(15,8);
\draw(15,9)--(15,8);
\draw(14,9)--(14,8);
\draw(13,9)--(13,8);
\draw(3,-0.5) node{$\scriptstyle\cdots$};
\draw(0.5,-2.8) node{$\scriptstyle\vdots$};
\draw(0.5,-0.5) node{$\scriptstyle 0$};
\draw(1.5,-0.5) node{$\scriptstyle 1$};
\draw(0.5,-1.5) node{$\scriptstyle e\hspace{-0.5mm}-\hspace{-0.5mm}1$};
\draw(0.5,-4.5) node{$\scriptstyle i\hspace{-0.5mm}+\hspace{-0.5mm}1$};
\draw(0.5,-5.5) node{$\scriptstyle i$};
\draw(4.5,-0.5) node{$\scriptstyle i\hspace{-0.5mm}-\hspace{-0.5mm}1$};
\draw(5.5,-0.5) node{$\scriptstyle i$};
\draw(5.5,0.5) node{$\scriptstyle i\hspace{-0.5mm}+\hspace{-0.5mm}1$};
\draw(6,3) node{$\scriptstyle \iddots$};
\draw(12,8.5) node{$\scriptstyle \cdots$};
\draw(9.5,6.2) node{$\scriptstyle \vdots$};
\draw(9.5,8.5) node{$\scriptstyle 0$};
\draw(10.5,8.5) node{$\scriptstyle 1$};
\draw(9.5,7.5) node{$ \scriptstyle e\hspace{-0.5mm}-\hspace{-0.5mm}1$};
\draw(9.5,4.5) node{$\scriptstyle i\hspace{-0.5mm}+\hspace{-0.5mm}1$};
\draw(9.5,3.5) node{$\scriptstyle i$};
\draw(13.5,8.5) node{$\scriptstyle i\hspace{-0.5mm}-\hspace{-0.5mm}1$};
\draw(14.5,8.5) node{$\scriptstyle i$};
\draw(8.5,3.5) node{$\scriptstyle i\hspace{-0.5mm}-\hspace{-0.5mm}1$};
\end{tikzpicture}
\end{align}
\begin{lem}\label{alphai} For \(1 \leq i \leq e-1\), \(m \in \ZZ_{\geq 0}\), \(L(m\delta + \alpha_i) \cong S^{\lambda^{m,i}/\mu^{m,i}}\).
\end{lem}
\begin{proof}
We prove this by induction on \(m\). As \(\lambda^{0,i}/\mu^{0,i} = \iota_i\), the claim follows by Lemma \ref{basecase}. Now assume that \(L(m\delta + \alpha_i) \cong S^{\lambda^{m,i}/\mu^{m,i}}\). It is easy to see that \((m\delta+ \alpha_i, \delta)\) is a minimal pair for \((m+1)\delta+ \alpha_i\) (see \cite[\S 6.1]{cusp}). By Lemma \ref{imag}, the factors of \(S^{\eta_i}\) are \(L_{\delta,i}\) and \(L_{\delta,i-1}\langle 1 \rangle\). Thus by Lemma \ref{minpair} the only possible factors (up to shift) of \(S^{\lambda^{m,i}/\mu^{m,i}} \circ S^{\eta_i}\) are 
\begin{align}\label{factors1}
L((m+1)\delta+\alpha_i) \hspace{10mm} \textup{and} \hspace{10mm} L( m\delta+\alpha_i, \delta^{(j)}), \textup{ for } j \in I \backslash\{0\}, 
\end{align} 
where we write \(\delta^{(j)}\) for the \((e-1)\)-multipartition of 1 which is \((1)\) in the \(j\)th component and empty elsewhere. 

Note that \(\blam/\bmu:=(\lambda^{m,i},\eta_i)/(\mu^{m,i},\varnothing)\) is joinable, with \(\lambda_*/\mu_* \) (as defined in \S\ref{joinablesec}) equal to  \(\lambda^{m+1,i}/\mu^{m+1,i}\), so by Lemma \ref{skewcharsep}, we have
\begin{align*}
\CH_q(S^{\lambda^{m,i}/\mu^{m,i}} \circ S^{\eta_i}) = q^a \CH_q(S^{\bzeta^{m+1,i}}) + q^b \CH_q(S^{\lambda^*/\mu^*})
\end{align*}
for some \(a,b \in \ZZ\). By injectivity of the character map \cite[Theorem 3.17]{kl}, it follows that the only factors of \(S^{\lambda^{m+1,i}/\mu^{m+1,i}}\) are those in (\ref{factors1}), up to some shift. If \(\ttt \in \St(\lambda^{m+1,i}/\mu^{m+1,i})\), with \(\bi(\ttt) = i_1\cdots i_k\), note that \(\alpha_{i_{k-e+1}} + \cdots + \alpha_{i_k} \neq \delta\), i.e., there is no sequence of removable nodes in \ref{diagzeta} whose residues add up to \(\delta\), as is easily seen. Thus \(\Res_{m\delta+\alpha_i,\delta}S^{\lambda^{m+1,i}/\mu^{m+1,i}} = 0\). But by adjointness and Theorem \ref{cuspthm}(i), \(\Res_{m\delta+\alpha_i, \delta} L(m\delta + \alpha_i, \delta^{(j)}) \neq 0\) for all \(j \in I\backslash\{0\}\). Hence \(L(m\delta+\alpha_i,\delta^{(j)})\) is not a factor of \(S^{\lambda^{m+1,i}/\mu^{m+1,i}}\) for any \(j\), and the only possible factor is \(L((m+1)\delta+\alpha_i)\) some number of times, with shifts. 

Consider the extremal word
\begin{align*}
\bi = 0^{m+1}
1^{m+1}
\cdots
(i-1)^{m+1}
(e-1)^{m+1}
\cdots
(i+1)^{m+1}
i^{m+2}
\end{align*}
of \(S^{\lambda^{m+1,i}/\mu^{m+1,i}}\). There are \(((m+1)!)^{e-1}(m+2)!\) distinct \(\ttt \in \St(\lambda^{m+1,i}/\mu^{m+1,i})\) such that \(\bi(\ttt) = \bi\), so this is the (ungraded) dimension of the \(\bi\)-word space of \(S^{\lambda^{m+1,i}/\mu^{m+1,i}}\). By Lemma \ref{extremal}, the dimension of a module with extremal word \(\bi\) must be exactly \(([m+1]^!)^{e-1}[m+2]^!\), which implies that \(L((m+1)\delta + \alpha_i)\) can only appear once in \(S^{\lambda^{m+1,i}/\mu^{m+1,i}}\), with some shift. 

Let \(\ttt^{\textup{top}} \in \St(\lambda^{m+1,i}/\mu^{m+1,i})\) be the tableau achieved by entering \(1, \ldots, m\) in the 0-nodes of \(\lambda^{m+1,i}/\mu^{m+1,i}\) from top to bottom, then \(m+1, \ldots, 2m\) in the 1-nodes from top to bottom, and so forth, until the \((i-1)\)-nodes are filled, then fill the nodes with residue \(e-1, e-2, \ldots, i\) in the same fashion, working from top to bottom. Then \(\ttt^{\textup{top}}\) has residue sequence \(\bi\), and 
\begin{align*}
\deg \ttt^{\textup{top}} = \frac{(e-1)m(m+1)}{2} + \frac{(m+1)(m+2)}{2}.
\end{align*}
Let \(\ttt^{\textup{bot}}\) be constructed in the same fashion, except with nodes filled from bottom to top. Then
\begin{align*}
\deg \ttt^{\textup{bot}} = -\frac{(e-1)m(m+1)}{2} - \frac{(m+1)(m+2)}{2}.
\end{align*}
As these degrees are the greatest and least in the expression \(([m+1]^!)^{e-1}[m+2]^!\) it follows that \(S^{\lambda^{m+1,i}/\mu^{m+1,i}}\) is symmetric with respect to grading, and hence \(S^{\lambda^{m+1,i}/\mu^{m+1,i}} \cong L((m+1)\delta + \alpha_i)\) with no shift. 
\end{proof}
For \(1 \leq i \leq e-1\), \(m \in \ZZ_{\geq 1}\), let \(l=1\), \(\kappa = ((1-m)i \pmod e)\), and let \(\lambda_{m,i}/\mu_{m,i}\) be the skew hook diagram in \(\mathscr{S}^{\kappa}\), where
\begin{align*}
\lambda_{m,i} &= (mi,((m-1)i+1)^{e-i }, \ldots,(i+1)^{e-i},1^{e-i-1})\\
\mu_{m,i} &= (((m-1)i)^{e-i}, \ldots, (2i)^{e-i},i^{e-i}).
\end{align*}
In other words, \(\lambda_{m,i}/\mu_{m,i}\) is the minimal skew hook diagram with residues shown below, with the 0-node appearing on the inner corners \(m\) times, and the \(i\)-node appearing on the outer corners \(m-1\) times.
\begin{align}\label{-diagzeta}
\begin{tikzpicture}[scale=0.42]
\tikzset{baseline=0mm}
\draw(0,0)--(2.5,0);
\draw(0,-1)--(2.5,-1);
\draw(3.5,0)--(6,0);
\draw(3.5,-1)--(6,-1);
\draw(0,0)--(0,-2.5);
\draw(1,0)--(1,-2.5);
\draw(0,-3.5)--(0,-6);
\draw(1,-3.5)--(1,-6);
\draw(2,0)--(2,-1);
\draw(4,0)--(4,-1);
\draw(5,0)--(5,-1);
\draw(6,0)--(6,-1);
\draw(0,-2)--(1,-2);
\draw(0,-4)--(1,-4);
\draw(0,-5)--(1,-5);
\draw(0,-6)--(1,-6);
\draw(5,0)--(5,1.5);
\draw(6,0)--(6,1.5);
\draw(5,1)--(6,1);
\draw(7.5,3)--(10,3);
\draw(7.5,4)--(10,4);
\draw(8,3)--(8,4);
\draw(9,3)--(9,5.5);
\draw(10,3)--(10,5.5);
\draw(9,5)--(10,5);
\draw(9,6.5)--(9,9);
\draw(10,6.5)--(10,9);
\draw(9,9)--(11.5,9);
\draw(9,8)--(11.5,8);
\draw(9,7)--(10,7);
\draw(11,9)--(11,8);
\draw(12.5,9)--(15,9);
\draw(12.5,8)--(15,8);
\draw(15,9)--(15,8);
\draw(14,9)--(14,8);
\draw(13,9)--(13,8);
\draw(3,-0.5) node{$\scriptstyle \cdots$};
\draw(0.5,-2.8) node{$\scriptstyle\vdots$};
\draw(0.5,-0.5) node{$\scriptstyle 0$};
\draw(1.5,-0.5) node{$\scriptstyle 1$};
\draw(0.5,-1.5) node{$\scriptstyle e\hspace{-0.5mm}-\hspace{-0.5mm}1$};
\draw(0.5,-4.5) node{$\scriptstyle i\hspace{-0.5mm}+\hspace{-0.5mm}2$};
\draw(0.5,-5.5) node{$\scriptstyle i\hspace{-0.5mm}+\hspace{-0.5mm}1$};
\draw(4.5,-0.5) node{$\scriptstyle i\hspace{-0.5mm}-\hspace{-0.5mm}1$};
\draw(5.5,-0.5) node{$\scriptstyle i$};
\draw(5.5,0.5) node{$\scriptstyle i\hspace{-0.5mm}+\hspace{-0.5mm}1$};
\draw(6,3) node{$\scriptstyle\iddots$};
\draw(12,8.5) node{$\scriptstyle\cdots$};
\draw(9.5,6.2) node{$\scriptstyle\vdots$};
\draw(9.5,8.5) node{$\scriptstyle0$};
\draw(10.5,8.5) node{$\scriptstyle1$};
\draw(9.5,7.5) node{$\scriptstyle e\hspace{-0.5mm}-\hspace{-0.5mm}1$};
\draw(9.5,4.5) node{$\scriptstyle i\hspace{-0.5mm}+\hspace{-0.5mm}1$};
\draw(9.5,3.5) node{$\scriptstyle i$};
\draw(13.5,8.5) node{$\scriptstyle i\hspace{-0.5mm}-\hspace{-0.5mm}2$};
\draw(14.5,8.5) node{$\scriptstyle i\hspace{-0.5mm}-\hspace{-0.5mm}1$};
\draw(8.5,3.5) node{$\scriptstyle i\hspace{-0.5mm}-\hspace{-0.5mm}1$};
\end{tikzpicture}
\end{align}
\begin{lem}\label{-alphai}
For \(1 \leq i \leq e-1\), \(m \in \ZZ_{\geq 1}\), \(L(m\delta- \alpha_i) \cong S^{\lambda_{m,i}/\mu_{m,i}}\langle 1-m\rangle\).
\end{lem}
\begin{proof}
We go by induction on \(m\), and the proof proceeds in the same manner as Lemma \ref{alphai}. The base case is slightly different however. \(S^{\lambda_{1,i}/\mu_{1,i}}\) is the following hook partition, with residues shown:
\begin{align*}
\begin{tikzpicture}[scale=0.42]
\tikzset{baseline=0mm}
\draw(0,0)--(2.5,0);
\draw(0,-1)--(2.5,-1);
\draw(3.5,0)--(6,0);
\draw(3.5,-1)--(6,-1);
\draw(0,0)--(0,-2.5);
\draw(1,0)--(1,-2.5);
\draw(0,-3.5)--(0,-6);
\draw(1,-3.5)--(1,-6);
\draw(2,0)--(2,-1);
\draw(4,0)--(4,-1);
\draw(5,0)--(5,-1);
\draw(6,0)--(6,-1);
\draw(0,-2)--(1,-2);
\draw(0,-4)--(1,-4);
\draw(0,-5)--(1,-5);
\draw(0,-6)--(1,-6);
\draw(3,-0.5) node{$\scriptstyle \cdots$};
\draw(0.5,-2.8) node{$\scriptstyle \vdots$};
\draw(0.5,-0.5) node{$\scriptstyle 0$};
\draw(1.5,-0.5) node{$\scriptstyle 1$};
\draw(0.5,-1.5) node{$\scriptstyle e\hspace{-0.5mm}-\hspace{-0.5mm}1$};
\draw(0.5,-4.5) node{$\scriptstyle i\hspace{-0.5mm}+\hspace{-0.5mm}2$};
\draw(0.5,-5.5) node{$\scriptstyle i\hspace{-0.5mm}+\hspace{-0.5mm}1$};
\draw(4.5,-0.5) node{$\scriptstyle i\hspace{-0.5mm}-\hspace{-0.5mm}2$};
\draw(5.5,-0.5) node{$\scriptstyle i\hspace{-0.5mm}-\hspace{-0.5mm}1$};
\end{tikzpicture}
\end{align*}
By \cite[Lemma 5.2]{cusp}, \(L(\delta-\alpha_i)\) factors through the cyclotomic quotient to become the unique irreducible \(R^{\Lambda_0}_{\delta-\alpha_i}\)-module. Consideration of the words of \(S^{\lambda_{1,i}/\mu_{1,i}}\) shows that it factors through the cyclotomic quotient as well. Moreover, all of its word spaces are 1-dimensional and in degree 0, so it follows that \(S^{\lambda_{1,i}/\mu_{1,i}} \cong L(\delta-\alpha_i)\). 

The induction step proceeds as in Lemma \ref{alphai}, with \((\delta,m\delta - \alpha_i)\) used as a minimal pair for \((m+1)\delta- \alpha_i\). Considering the induction product \(S^{\eta_i}\circ S^{\lambda_{1,i}/\mu_{1,i}}\) and using Lemma \ref{skewcharsep}, one sees that the only possible factor of \(S^{\lambda_{1,i}/\mu_{1,i}}\) is \(L((m+1)\delta-\alpha_i)\), some number of times, with shifts. Consideration of the extremal word
\begin{align*}
\bi = 0^{m+1}
1^{m+1}
\cdots
(i-1)^{m+1}
(e-1)^{m+1}
\cdots
(i+1)^{m+1}
i^m
\end{align*}
shows that \(L((m+1)\delta-\alpha_i)\) appears but once as a factor of \(S^{\lambda_{1,i}/\mu_{1,i}}\), with some shift. \(L((m+1)\delta - \alpha_i)\) must have \(\bi\)-word space of graded dimension \(([m+1]^!)^{e-1}[m]^!\). As before, we define two standard \(\lambda_{1,i}/\mu_{1,i}\)-tableaux; \(\ttt^{\textup{top}}\), where the nodes are filled in from top to bottom according to their order in \(\bi\), and \(\ttt^{\textup{bot}}\), where the nodes are filled similarly from bottom to top. Then
\begin{align*}
\deg \ttt^{\textup{top}} &=\left[ \frac{(e-1)m(m+1)}{2} + \frac{(m-1)m}{2}\right] -m\\
\deg \ttt^{\textup{bot}} &=\left[ -\frac{(e-1)m(m+1)}{2} - \frac{(m-1)m}{2}\right] -m.
\end{align*}
On the right we have the greatest and least degrees in the expression \(([m+1]^!)^{e-1}[m]^!\), shifted by \(-m\), hence \(L((m+1)\delta - \alpha_i) \cong S^{\lambda_{1,i}/\mu_{1,i}}\langle 1-(m+1)\rangle\), completing the proof.
\end{proof}
\subsection{Identifying cuspidal modules as skew hook Specht modules} We now present an inductive process for identifying cuspidal modules as skew hook Specht modules with a certain shift.
\begin{prop}\label{induct}
 Let \(\alpha\) be a real positive root, and assume that for all real positive roots \(\beta\) with \(\textup{ht}(\beta)< \textup{ht}(\alpha)\), we have \(L_\beta \cong S^{\lambda_\beta/\mu_\beta}\langle c_\beta \rangle\) for  some skew hook diagram \(\lambda_\beta/\mu_\beta \in \mathscr{S}_{\beta}^\kappa\), where \(\kappa = (k)\) for some \(k \in I\) and \(c_\beta \in \ZZ\). Then the following process gives a skew hook diagram \(\lambda_\alpha/\mu_\alpha\) and \(c_\alpha\in \ZZ\) such that \(L_\alpha \cong S^{\lambda_\alpha/\mu_\alpha}\langle c_\alpha \rangle\).
\begin{enumerate}
\item If \(\alpha = m\delta + \alpha_i\) for some \(m \in \ZZ_{\geq 0}\) and \(i \in I\backslash\{0\}\), then \(\lambda_\alpha/\mu_\alpha = \lambda^{m,i}/\mu^{m,i}\) and \(c_\alpha=0\).
\item If \(\alpha = m\delta - \alpha_i\) for some \(m \in \ZZ_{\geq 1}\) and \(i \in I\backslash\{0\}\), then \(\lambda_\alpha/\mu_\alpha = \lambda_{m,i}/\mu_{m,i}\) and \(c_\alpha=1-m\).
\item Else there is a real minimal pair \((\beta,\gamma)\) for \(\alpha\).
\begin{enumerate}
\item If \(\blam/\bmu:=(\lambda_\beta,\lambda_\gamma)/(\mu_\beta,\mu_\gamma)\) is joinable, then \(\lambda_\alpha/\mu_\alpha = \lambda_*/\mu_*\), and
\begin{align*}
c_\alpha = c_\beta + c_\gamma - p_{\beta,\gamma} + (\beta,\gamma) + d_* - d_{\blam/\bmu};
\end{align*}
\item else \(\blam/\bmu:=(\lambda_\gamma, \lambda_\beta)/(\mu_\gamma,\mu_\beta)\) is joinable, \(\lambda_\alpha/\mu_\alpha = \lambda^*/\mu^*\), and
\begin{align*}
c_\alpha = c_\beta + c_\gamma + p_{\beta,\gamma} + d^* - d_{\blam/\bmu},
\end{align*}
\end{enumerate}
where \(d_*,d^*\) are as in Lemma \ref{skewcharsep}, \(d_{\blam/\bmu}\) as in Lemma \ref{circprod}, and \(p_{\beta,\gamma}\) as in (\ref{defp}).
\end{enumerate}
\end{prop}
\begin{proof}
(1) and (2) are Lemmas \ref{alphai} and \ref{-alphai}, so assume we are in case (3). There exists a real minimal pair \((\beta,\gamma)\) for \(\alpha\) by \cite[Lemma 6.9]{cusp}. By assumption \(L(\beta) \cong S^{\lambda_\beta/\mu_\beta}\langle c_\beta \rangle\), and \(L(\gamma) \cong S^{\lambda_\gamma/\mu_\gamma}\langle c_\gamma \rangle\). We have \(\beta = m\delta+(-1)^s( \alpha_i + \cdots + \alpha_j)\)  for some \(s \in \{0,1\}\), \(1 \leq i \leq j \leq e-1\) and \(\gamma = m'\delta+(-1)^{s'}( \alpha_{i'} + \cdots + \alpha_{j'})\)  for some \(s' \in \{0,1\}\), \(1 \leq i' \leq j' \leq e-1\).
 Since \(\beta + \gamma\) is a real root, one of the following must be true:
 \begin{align*}
s=s', j+1=i' \hspace{5mm}\textup{ or }\hspace{5mm}
s=s', j'+1=i \hspace{5mm}\textup{ or }\hspace{5mm}
s=-s', j=j' \hspace{5mm}\textup{ or }\hspace{5mm}
s=-s', i=i'.
 \end{align*}
Note that since \(\lambda_\beta/\mu_\beta\) (resp. \(\lambda_\gamma/\mu_\gamma\)) is a skew hook diagram, \(s=0\) (resp. \(s'=0\)) implies that the lower left node of \(\lambda_\beta/\mu_\beta\) (resp. \(\lambda_\gamma/\mu_\gamma\)) has residue \(i\) (resp. \(i'\)), and the top right node has residue \(j\) (resp. \(j'\)). If \(s=1\) (resp. \(s'=1\)), then the lower left node of \(\lambda_\beta/\mu_\beta\) (resp. \(\lambda_\gamma/\mu_\gamma\)) has residue \(j+1\) (resp. \(j'+1\)), and the top right node has residue \(i-1\) (resp. \(i'-1\)). In any case then, we see that one of \((\lambda_\beta,\lambda_\gamma)/(\mu_\beta,\mu_\gamma)\) or \((\lambda_\gamma, \lambda_\beta)/(\mu_\gamma,\mu_\beta)\) must be joinable. 

Assume the former, and set \(\blam/\bmu:=(\lambda_\beta,\lambda_\gamma)/(\mu_\beta,\mu_\gamma)\). Then, using Lemma \ref{realpair},
\begin{align*}
[S^{\lambda_\beta/\mu_\beta} \circ S^{\lambda_\gamma/\mu_\gamma}]&=q^{-c_\beta-c_\gamma}[L_\beta \circ L_\gamma] = q^{-c_\beta-c_\gamma}[L(\beta,\gamma)] + q^{-c_\beta-c_\gamma+p_{\beta,\gamma}-(\beta,\gamma)}[L_\alpha].
\end{align*}
Using Lemma \ref{skewcharsep} and the fact that \(\CH_q\) is injective on \([R_\alpha\textup{-mod}]\), we also have
\begin{align*}
[S^{\lambda_\beta/\mu_\beta} \circ S^{\lambda_\gamma/\mu_\gamma}]&= q^{d^*-d_{\blam/\bmu}}[S^{\lambda^*/\mu^*}] + q^{d_*-d_{\blam/\bmu}}[S^{\lambda_*/\mu_*}].
\end{align*}
Thus, \(L_\alpha\) must be (a shift of) \(S^{\lambda^*/\mu^*}\) or \(S^{\lambda_*/\mu_*}\). But \(1_{\beta,\gamma}z^{\lambda^*/\mu^*} \neq 0\), so \(\Res_{\beta,\gamma}S^{\lambda^*/\mu^*} \neq 0\), and thus the cuspidality property of \(L(\alpha)\) implies it must be the latter, proving the validity of step (3)(a). If instead,  \(\blam/\bmu:= (\lambda_\beta, \lambda_\gamma, \mu_\beta, \mu_\gamma)\) is joinable, we use the second statement in Lemma \ref{realpair} and a similar argument to prove the validity of step (3)(b).
\end{proof}
\begin{cor}\label{cuspskew}
For a balanced convex preorder, all real cuspidal modules of \(R_\alpha\) are skew hook Specht modules up to some shift. 
\end{cor}
\begin{proof}
Apply Proposition \ref{induct} inductively, with base case given by Lemma \ref{basecase}.
\end{proof}
\begin{remark}
In \cite[\S8.4]{lyndon}, Kleshchev and Ram showed that in {\it finite} type \({\tt A}\), the cuspidal modules (associated to a convex lexicographic order) are Specht modules associated to hook partitions. Thus one can view Corollary \ref{cuspskew} as an affine analogue of this fact.
\end{remark}
\subsection{Cuspidal modules for a special preorder}
To give a complete picture, we explicitly describe the skew Specht modules corresponding to real positive roots in the case of a certain balanced {\it \(e\)-row preorder} (in the sense of \cite{ito}) on \(\Phi_+\), where the associated skew hook diagrams take on a very regular pattern. Take
\begin{enumerate}
\item \(m\delta + \alpha \succ m'\delta \succ m''\delta - \alpha\), for all \(m \in \ZZ_{\geq 0}\), \(m',m'' \in \ZZ_{\geq 1}\), \(\alpha \in \Phi'_+\).
\item \(m\delta + \alpha_i + \cdots + \alpha_j \succ m'\delta + \alpha_{i'} + \cdots + \alpha_{j'}\) if
\begin{align*}
i<i'; \hspace{8mm}\textup{or}\hspace{8mm}
i=i', m<m';\hspace{8mm}\textup{or}\hspace{8mm}
i=i', m=m', j<j'.
\end{align*}
\item \(m\delta - \alpha_i - \cdots - \alpha_j \succ m'\delta - \alpha_{i'} - \cdots - \alpha_{j'}\) if
\begin{align*}
i>i'; \hspace{8mm}\textup{or}\hspace{8mm}
i=i', m>m';\hspace{8mm}\textup{or}\hspace{8mm}
i=i', m=m', j<j'.
\end{align*}
\end{enumerate}
Under this preorder, it is easy to see that for any \(\alpha \succ \delta\) not of the form \(m\delta + \alpha_i\), the positive root \(\beta \succ \alpha\) immediately preceding \(\alpha\) in the order constitutes the lefthand side of a real minimal pair \((\beta,\alpha-\beta)\) for \(\alpha\). Similarly, for \(\alpha \prec \delta\) not of the form \(m\delta - \alpha_i\), the positive root \(\alpha \succ \beta\) immediately succeeding \(\alpha\) in the order constitutes the righthand side of a real minimal pair \((\alpha-\beta,\beta)\) for \(\alpha\). Then, applying the inductive process in Proposition \ref{induct}, we arrive at:
\begin{enumerate}
\item For \(1 \leq i\leq j\leq e-1\) and \(m \in \ZZ_{\geq 0}\), \(L(m\delta + \alpha_i + \cdots + \alpha_j) \cong S^{\lambda/\mu}\), where \(\lambda/\mu\) is the minimal skew hook diagram with residues shown on the left below, with the 0-node appearing on the inner corners \(m\) times, and the \(i\)-node appearing on the outer corners \(m+1\) times. 
\item For \(1 \leq i\leq j \leq e-1\) and \(m \in \ZZ_{\geq 1}\),  \(L(m\delta - \alpha_i - \cdots - \alpha_j) \cong S^{\lambda/\mu}\langle 1-m\rangle\), where \(\lambda/\mu\) is the minimal skew hook diagram with residues shown on the right below, with the 0-node appearing on the inner corners \(m\) times, and the \(i\)-node appearing on the outer corners \(m-1\) times.
\end{enumerate} 
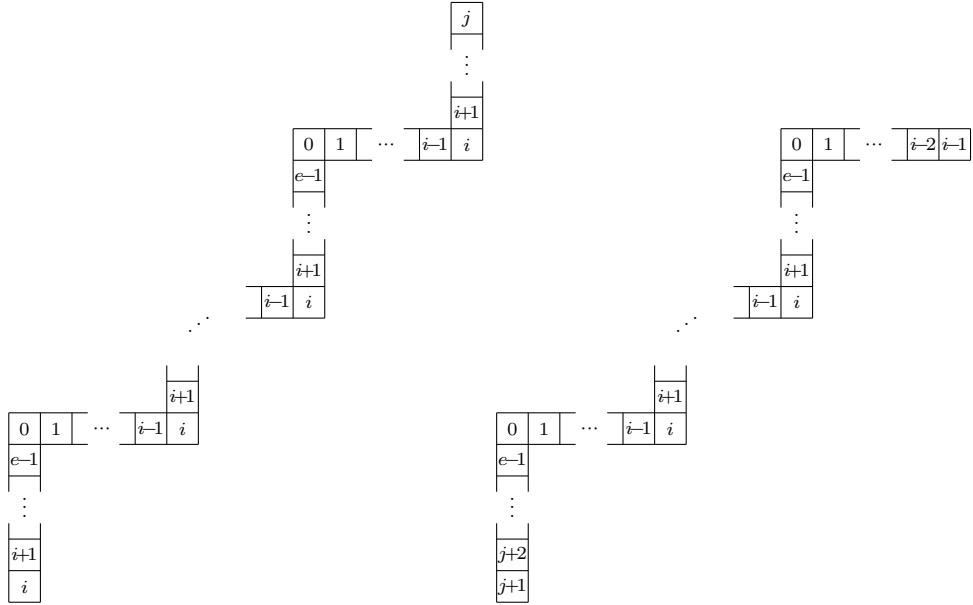
\begin{align*}
\begin{tikzpicture}[scale=0.42]
\tikzset{baseline=0mm}
\draw(0,0)--(2.5,0);
\draw(0,-1)--(2.5,-1);
\draw(3.5,0)--(6,0);
\draw(3.5,-1)--(6,-1);
\draw(0,0)--(0,-2.5);
\draw(1,0)--(1,-2.5);
\draw(0,-3.5)--(0,-6);
\draw(1,-3.5)--(1,-6);
\draw(2,0)--(2,-1);
\draw(4,0)--(4,-1);
\draw(5,0)--(5,-1);
\draw(6,0)--(6,-1);
\draw(0,-2)--(1,-2);
\draw(0,-4)--(1,-4);
\draw(0,-5)--(1,-5);
\draw(0,-6)--(1,-6);
\draw(5,0)--(5,1.5);
\draw(6,0)--(6,1.5);
\draw(5,1)--(6,1);
\draw(7.5,3)--(10,3);
\draw(7.5,4)--(10,4);
\draw(8,3)--(8,4);
\draw(9,3)--(9,5.5);
\draw(10,3)--(10,5.5);
\draw(9,5)--(10,5);
\draw(9,6.5)--(9,9);
\draw(10,6.5)--(10,9);
\draw(9,9)--(11.5,9);
\draw(9,8)--(11.5,8);
\draw(9,7)--(10,7);
\draw(11,9)--(11,8);
\draw(12.5,9)--(15,9);
\draw(12.5,8)--(15,8);
\draw(15,9)--(15,8);
\draw(14,9)--(14,8);
\draw(13,9)--(13,8);
\draw(15,9)--(15,10.5);
\draw(14,9)--(14,10.5);
\draw(15,11.5)--(15,13);
\draw(14,11.5)--(14,13);
\draw(14,13)--(15,13);
\draw(14,12)--(15,12);
\draw(15,10)--(14,10);
\draw(3,-0.5) node{$\scriptstyle \cdots$};
\draw(0.5,-2.8) node{$\scriptstyle \vdots$};
\draw(0.5,-0.5) node{$\scriptstyle 0$};
\draw(1.5,-0.5) node{$\scriptstyle 1$};
\draw(0.5,-1.5) node{$\scriptstyle e\hspace{-0.5mm}-\hspace{-0.5mm}1$};
\draw(0.5,-4.5) node{$\scriptstyle i\hspace{-0.5mm}+\hspace{-0.5mm}1$};
\draw(0.5,-5.5) node{$\scriptstyle i$};
\draw(4.5,-0.5) node{$\scriptstyle i\hspace{-0.5mm}-\hspace{-0.5mm}1$};
\draw(5.5,-0.5) node{$\scriptstyle i$};
\draw(5.5,0.5) node{$\scriptstyle i\hspace{-0.5mm}+\hspace{-0.5mm}1$};
\draw(6,3) node{$\scriptstyle \iddots$};
\draw(12,8.5) node{$\scriptstyle \cdots$};
\draw(9.5,6.2) node{$\scriptstyle \vdots$};
\draw(9.5,8.5) node{$\scriptstyle 0$};
\draw(10.5,8.5) node{$\scriptstyle 1$};
\draw(9.5,7.5) node{$\scriptstyle e\hspace{-0.5mm}-\hspace{-0.5mm}1$};
\draw(9.5,4.5) node{$\scriptstyle i\hspace{-0.5mm}+\hspace{-0.5mm}1$};
\draw(9.5,3.5) node{$\scriptstyle i$};
\draw(13.5,8.5) node{$\scriptstyle i\hspace{-0.5mm}-\hspace{-0.5mm}1$};
\draw(14.5,8.5) node{$\scriptstyle i$};
\draw(8.5,3.5) node{$\scriptstyle i\hspace{-0.5mm}-\hspace{-0.5mm}1$};
\draw(14.5,9.5) node{$\scriptstyle i\hspace{-0.5mm}+\hspace{-0.5mm}1$};
\draw(14.5,12.5) node{$\scriptstyle j$};
\draw(14.5,11.2) node{$\scriptstyle \vdots$};
\end{tikzpicture}
\hspace{0cm}
\begin{tikzpicture}[scale=0.42]
\tikzset{baseline=0mm}
\draw(0,0)--(2.5,0);
\draw(0,-1)--(2.5,-1);
\draw(3.5,0)--(6,0);
\draw(3.5,-1)--(6,-1);
\draw(0,0)--(0,-2.5);
\draw(1,0)--(1,-2.5);
\draw(0,-3.5)--(0,-6);
\draw(1,-3.5)--(1,-6);
\draw(2,0)--(2,-1);
\draw(4,0)--(4,-1);
\draw(5,0)--(5,-1);
\draw(6,0)--(6,-1);
\draw(0,-2)--(1,-2);
\draw(0,-4)--(1,-4);
\draw(0,-5)--(1,-5);
\draw(0,-6)--(1,-6);
\draw(5,0)--(5,1.5);
\draw(6,0)--(6,1.5);
\draw(5,1)--(6,1);
\draw(7.5,3)--(10,3);
\draw(7.5,4)--(10,4);
\draw(8,3)--(8,4);
\draw(9,3)--(9,5.5);
\draw(10,3)--(10,5.5);
\draw(9,5)--(10,5);
\draw(9,6.5)--(9,9);
\draw(10,6.5)--(10,9);
\draw(9,9)--(11.5,9);
\draw(9,8)--(11.5,8);
\draw(9,7)--(10,7);
\draw(11,9)--(11,8);
\draw(12.5,9)--(15,9);
\draw(12.5,8)--(15,8);
\draw(15,9)--(15,8);
\draw(14,9)--(14,8);
\draw(13,9)--(13,8);
\draw(3,-0.5) node{$\scriptstyle \cdots$};
\draw(0.5,-2.8) node{$\scriptstyle \vdots$};
\draw(0.5,-0.5) node{$\scriptstyle 0$};
\draw(1.5,-0.5) node{$\scriptstyle 1$};
\draw(0.5,-1.5) node{$\scriptstyle e\hspace{-0.5mm}-\hspace{-0.5mm}1$};
\draw(0.5,-4.5) node{$\scriptstyle j\hspace{-0.5mm}+\hspace{-0.5mm}2$};
\draw(0.5,-5.5) node{$\scriptstyle j\hspace{-0.5mm}+\hspace{-0.5mm}1$};
\draw(4.5,-0.5) node{$\scriptstyle i\hspace{-0.5mm}-\hspace{-0.5mm}1$};
\draw(5.5,-0.5) node{$\scriptstyle i$};
\draw(5.5,0.5) node{$\scriptstyle i\hspace{-0.5mm}+\hspace{-0.5mm}1$};
\draw(6,3) node{$\scriptstyle \iddots$};
\draw(12,8.5) node{$\scriptstyle \cdots$};
\draw(9.5,6.2) node{$\scriptstyle \vdots$};
\draw(9.5,8.5) node{$\scriptstyle 0$};
\draw(10.5,8.5) node{$\scriptstyle 1$};
\draw(9.5,7.5) node{$\scriptstyle e\hspace{-0.5mm}-\hspace{-0.5mm}1$};
\draw(9.5,4.5) node{$\scriptstyle i\hspace{-0.5mm}+\hspace{-0.5mm}1$};
\draw(9.5,3.5) node{$\scriptstyle i$};
\draw(13.5,8.5) node{$\scriptstyle i\hspace{-0.5mm}-\hspace{-0.5mm}2$};
\draw(14.5,8.5) node{$\scriptstyle i\hspace{-0.5mm}-\hspace{-0.5mm}1$};
\draw(8.5,3.5) node{$\scriptstyle i\hspace{-0.5mm}-\hspace{-0.5mm}1$};
\end{tikzpicture}
\end{align*}


\vspace{0.5cm}





\begin{thebibliography}{1}

\bibitem{bcp} J. Beck, V. Chari and A. Pressley, An algebraic characterization of the affine canonical basis, {\em Duke Math. J.} {\bf 99} (1999), 455--487.

\bibitem{bk} J. Brundan, A. Kleshchev, Blocks of cyclotomic Hecke algebras and Khovanov-Lauda algebras, {\em Invent. Math.}, {\bf 178} (2009), 451--484.

\bibitem{bkw} J. Brundan, A. Kleshchev and W. Wang, Graded Specht modules, {\em Journal fur die reine und angewandte Mathematik}, {\bf 655} (2011), 61--87.

\bibitem{fs} M. Fayers, L. Speyer, Generalised column removal for graded homomorphisms between Specht modules, {\tt arXiv:1404.4415}. 

\bibitem{ito} K. Ito, The classification of convex orders on affine root systems, {\tt arXiv:math/9912020}.

\bibitem{JP} G. D. James and M. H. Peel, Specht series for skew representations of symmetric groups, {\em J. Algebra} {\bf 56} (1979), 343-364.

\bibitem{kl} M. Khovanov and A. Lauda, A diagrammatic approach to categorification of quantum groups I, {\em Representation Theory} {\bf 13} (2009), 309--347.

\bibitem{cusp} A. Kleshchev, Cuspidal systems for affine Khovanov-Lauda-Rouquier algebras, {\em Mathematische Zeitschrift}, {\bf 276} (2014), 691--726. 

\bibitem{klesh} A. Kleshchev, {\em Linear and Projective Representations of Symmetric Groups}, Cambridge University Press, Cambridge, 2005.

\bibitem{kmr} A. Kleshchev, A. Mathas, and A. Ram, Universal Specht modules for cyclotomic Hecke algebras, {\em Proceedings of the London Mathematical Society (2)} {\bf 105} (2012), 1245--1289.

\bibitem{imagsw} A. Kleshchev, R. Muth, Imaginary Schur-Weyl duality, {\tt arXiv:1312.6104}.

\bibitem{KMStrat}
A. Kleshchev and R. Muth, Stratifying KLR algebras of affine Lie type, {\tt arXiv:1511.05511}


\bibitem{lyndon} A. Kleshchev, A. Ram, Representations of Khovanov-Lauda-Rouquier algebras and combinatorics of Lyndon words, {\em Math. Ann.} {\bf 349} (2011), 943--975.

\bibitem{mathas}A. Mathas, {\em Iwahori-Hecke algebras and Schur algebras of the symmetric group}, University Lecture Series 15, American Mathematical Society, Providence, RI, 1999.

\bibitem{mcnfin} P. McNamara, Finite dimensional representations of Khovanov-Lauda-Rouquier algebras I: Finite type, {\tt arXiv:1207.5860}

\bibitem{mcn} P. McNamara, Representations of Khovanov-Lauda-Rouquier algebras III, Symmetric Affine Type, {\tt arXiv:1407.7304}

\bibitem{rouq} R. Rouquier, 2-Kac-Moody algebras, {\tt arXiv:0812.5023}.

\bibitem{TW}
P. Tingley and B. Webster, Mirkovic-Vilonen polytopes and Khovanov-Lauda-Rouquier algebras, {\tt arXiv:1210.6921}. 

  \end{thebibliography}
\end{document}